\documentclass[11pt]{article}

\usepackage[margin=1.3in]{geometry}
\usepackage{amsmath,amssymb,amsthm,mathtools}
\usepackage{microtype}
\usepackage{enumitem}
\usepackage[hidelinks]{hyperref}

\newtheorem{theorem}{Theorem}[section]
\newtheorem{proposition}[theorem]{Proposition}
\newtheorem{lemma}[theorem]{Lemma}
\newtheorem{corollary}[theorem]{Corollary}

\theoremstyle{definition}
\newtheorem{definition}[theorem]{Definition}
\theoremstyle{remark}
\newtheorem{remark}[theorem]{Remark}

\DeclareMathOperator{\Mat}{Mat}
\DeclareMathOperator{\Jac}{Jac}
\DeclareMathOperator{\Aut}{Aut}
\DeclareMathOperator{\GL}{GL}
\DeclareMathOperator{\Spec}{Spec}

\DeclareMathOperator{\rank}{rank}
\DeclareMathOperator{\Frac}{Frac}
\DeclareMathOperator{\Crit}{Crit}
\DeclareMathOperator{\im}{ima}

\DeclareMathOperator{\codim}{codim}
\DeclareMathOperator{\fpt}{fpt}
\DeclareMathOperator{\Grass}{Gr}
\DeclareMathOperator{\Ann}{Ann}
\DeclareMathOperator{\Hess}{Hess}
\DeclareMathOperator{\Stab}{Stab}
\DeclareMathOperator{\JacMat}{JacMat}
\DeclareMathOperator{\Tj}{Tj}
\DeclareMathOperator{\Char}{char}
\DeclareMathOperator{\Hom}{Hom}
\DeclareMathOperator{\Spf}{Spf}
\DeclareMathOperator{\init}{in}

\DeclareMathOperator{\height}{ht}
\newcommand{\m}{\mathfrak m}
\newcommand{\nF}{\mathfrak n_F}
\newcommand{\Ttilde}{\widetilde T}
\newcommand{\A}{\mathbb A}
\newcommand{\R}{\mathcal R}
\newcommand{\Gl}{\mathcal G_l}
\newcommand{\Gr}{\mathcal G_r}
\newcommand{\Glr}{\mathcal G_{lr}}
\DeclareMathOperator{\rmod}{Rmod}
\DeclareMathOperator{\propermod}{pmod}
\DeclareMathOperator{\Fmod}{Fmod}
\DeclareMathOperator{\ord}{ord}
\DeclareMathOperator{\length}{length}
\DeclareMathOperator{\Span}{Span}

\title{Determinacy and Modality in Positive Characteristic}
\author{Yotam Svoray}
\date{ }

\AtEndDocument{\bigskip{\footnotesize%
  \textsc{Einstein Institute of Mathematics, The Hebrew University of Jerusalem, Edmond J.
Safra Campus, Giv’at Ram, Jerusalem, 91904, Israel} \par  
   \textit{E-mail address}: \texttt{ yotam.svoray@huji.mail.ac.il} 
}}

\begin{document}
\maketitle

\begin{abstract}
We study finite determinacy and modality in positive characteristic through three complementary constructions: prescribed-jet transversality, Frobenius deformation directions, and a finite flat critical cover. For matrix germs over an arbitrary field, prescribed-jet transversality to the rank stratification shows that finite left--right determinacy is equivalent to finite codimension of the tangent image. Analogous results hold for pure right equivalence, for one-sided actions over infinite fields, and in specified ranges over finite fields. These methods also yield a characterization of finite contact determinacy for positive-dimensional ideals and imply expected-height, Cohen--Macaulayness, reducedness, and normality results for determinantal ideals. For isolated hypersurface singularities over an algebraically closed field of characteristic \(p>0\), we introduce a Frobenius algebra whose length and Loewy structure govern the Milnor number, determinacy, and essential corank. Using a finite flat critical cover, we prove that proper modality coincides with right modality. It follows that, in fixed characteristic, bounded right modality gives an ambient-dimension-independent bound on determinacy and, up to nonsingular quadratic suspension, a finite algebraic parametrization by modular families. The same Frobenius methods yield uniform finiteness results for \(F\)-jumping numbers and show that the \(F\)-pure threshold is determined at a single Frobenius level.
\end{abstract}

\section{Introduction}\label{sec:introduction}

Finite determinacy is one of the main tools in the study and classification of
singularities.  It allows one to replace a formal power series, or a matrix of power
series, by a finite jet without changing its equivalence class.  Over a field of
characteristic zero, finite determinacy is closely related to the tangent space of the
corresponding orbit.  In positive characteristic the relation is subtler, since an
orbit map need not be separable and its tangent image can be strictly smaller than the
tangent space to the orbit; see \cite{GP18,GLS}.\\

Understanding this subtly lets us approach this problem from a common principle. Specifically, finite determinacy allows geometric information obtained from sufficiently high-order perturbations to be transported back to the original germ. For matrix germs, the perturbations are chosen by prescribed-jet transversality. For hypersurface germs, Frobenius supplies deformation directions that survive modulo the infinitesimal right action, while the critical-cover construction converts Milnor numbers into local degrees of a finite map. These mechanisms lead, respectively, to tangent-image criteria, uniform bounds at fixed modality, and the comparison of proper and right modality. We state the structural results first and then record the $F$-singularity consequences.\\

Consequently, we begin with matrix germs over an arbitrary field.  Greuel--Pham proved that finite
codimension of the appropriate tangent image implies finite determinacy over an
arbitrary field, and proved the converse when the relevant orbit maps are eventually
separable \cite{GP18}.  The converse over arbitrary fields is stated as
Conjecture~3.1.16 in \cite{GLS}.  Our prescribed-jet approximation theorem,
Theorem~\ref{thm:stratified-jet}, applies to maps into a variety with a finite smooth
stratification.  Over finite fields we use Poonen's closed-point sieve with prescribed Taylor data~\cite{Poonen}, imposing simultaneously the first-jet smoothness conditions required on the finitely many strata; over infinite fields the corresponding statement follows from an incidence-dimension argument. Applied to the rank stratification, it gives
Theorem~\ref{thm:jettrans}.  The one-sided actions are then reduced to kernel and image
maps to Grassmannians by Proposition~\ref{prop:one-sided-criterion}, and the remaining
punctured critical loci are eliminated by the Frobenius gauges of
Lemma~\ref{lem:frobenius-gauge-flexibility}.  The pure right action is treated separately
in Theorem~\ref{thm:pure-right}.

\begin{theorem}\label{thm:intro-matrix}
Let $K$ be a field, put
$R=K[[x_1,\ldots,x_s]]$, $\m=\langle x_1,\ldots,x_s\rangle$, and
$M_{m,n}=\Mat(m,n,R)$.  We write
$\R=\Aut_K(R)$,
$\Gl=\GL_m(R)\rtimes\R$,
$\Gr=\GL_n(R)^{\mathrm{op}}\rtimes\R$, and
$\Glr=(\GL_m(R)\times\GL_n(R)^{\mathrm{op}})\rtimes\R$.
For $1\le t\le\min(m,n)$, let $I_t(A)$ be the ideal generated by the
$t\times t$ minors of $A$ and put
$c_t=(m-t+1)(n-t+1)$. Then:
\begin{enumerate}[label=\textup{(\roman*)},leftmargin=*]
\item For $G\in\{\R,\Glr\}$, finite $G$-determinacy is equivalent to
finite codimension of the $G$-tangent image.  If $K$ is infinite, the
same equivalence holds for $G\in\{\Gl,\Gr\}$.  Over an arbitrary field,
the one-sided equivalence also holds when $n\le m$ and $s\le m-n+1$ for $\Gl$, and when $m\le n$ and $ s\le n-m+1$ for $\Gr$; see Theorems~\ref{thm:glr},
\ref{thm:one-sided-complete}, \ref{thm:one-sided-ranges},
and~\ref{thm:pure-right}.

\item A proper ideal $I\subseteq R$ with $\dim\left(\frac{R}{I}\right)>0$
is finitely contact determined if and only if
$\frac{R}{I}$ is an isolated complete-intersection singularity; see
Corollary~\ref{cor:icis-all-fields}.

\item If $A$ is finitely determined for one of the actions
$\R,\Gl,\Gr,\Glr$, then, for every
$1\le t\le\min\{m,n\}$ we have $\height I_t(A)
=
\min\{s,(m-t+1)(n-t+1)\}$;
see Theorem~\ref{thm:det-geometry-new}.

\item Under the hypotheses of \textup{(iii)}, the quotient $\frac{R}{I_t(A)}$
is Cohen--Macaulay.  Moreover, it is reduced if $(m-t+1)(n-t+1)<s$, and it is a normal domain if $(m-t+1)(n-t+1)\le s-2$; 
see Theorem~\ref{thm:det-geometry-new}.
\end{enumerate}
\end{theorem}


Keeping these distinctions in view, Theorem~\ref{thm:intro-matrix} gives the tangent-image criterion for the full left--right and pure right actions over every field, and for all four actions over every infinite field. Over finite fields, the one-sided criteria are established in the dimension ranges stated in Assertion~\textup{(i)}.\\

Yet another setting is that of isolated hypersurface singularities over an algebraically closed field $k$ of characteristic $p>0$. The Frobenius algebra and Frobenius modality are defined in Definition~\ref{def:Fmod}. Proposition~\ref{prop:Flower} shows that the Frobenius directions give genuine moduli. Their length and Loewy structure then control the Milnor number, separability of high jet orbits, determinacy, multiplicity, and essential corank; see Theorems~\ref{thm:length}, \ref{thm:smooth}, and \ref{thm:eta-radius}.\\

Another main input is the splitting lemma recorded in Lemma~\ref{lem:standard-inputs}\ref{std:split}. We write $c(f)$ for the \emph{essential corank}, namely the number of variables in its residual germ after a nonsingular quadratic summand has been split off. Finally, the critical-cover argument of Section~\ref{sec:boundedness} identifies proper modality with right modality in Theorem~\ref{thm:pmod-rmod-general}.

\begin{theorem}[Modality and uniform determinacy]\label{thm:intro-hypersurface}
Let $k$ be algebraically closed of characteristic $p>0$, and let $f$ be an isolated
hypersurface germ.
\begin{enumerate}[label=\textup{(\roman*)},leftmargin=*]
\item We have $\Fmod(f)\le\propermod(f)=\rmod(f)$; see Proposition~\ref{prop:Flower} and
Theorem~\ref{thm:pmod-rmod-general}.
\item We have
$\mu(f)+c(f)\le p^{c(f)}\bigl(\Fmod(f)+1\bigr)$. Consequently, we have that
$\mu(f)\le p^{\,2+\sqrt[3]{6\rmod(f)}}\bigl(\rmod(f)+1\bigr)$; see Corollary~\ref{cor:stable-milnor-growth}.

\item For every $D\ge0$, the residual number of variables, Milnor number,
multiplicity, and right-classification radius of all isolated germs with
$\rmod(f)\le D$ are bounded in terms of $p$ and $D$.  Moreover, there is an integer
$\widehat N_{p,D}$, independent of the original ambient dimension, such that every such
$f$ is right $\widehat N_{p,D}$-determined; see Theorem~\ref{thm:uniform} annd
Corollary~\ref{cor:uniform-full-determinacy}.
\end{enumerate}
\end{theorem}


Leveraging a Rosenlicht stratification of those jet spaces and a subsequent finite-multisection argument yields a finite algebraic parametrization. The relevant finite-multisection principle is Lemma~\ref{lem:finite-modular-principle}, and stable cancellation is Lemma~\ref{lem:stable-cancellation}.

\begin{theorem}[Finite modular classification]\label{thm:intro-classification}
Fix $p>0$ and $D\ge0$.  Up to nonsingular quadratic suspension, the isolated germs $f$
with $\rmod(f)\le D$ are covered by finitely many algebraic families
$F_i(x,t)$, with $t\in T_i$, such that:
\begin{enumerate}[label=\textup{(\roman*)},leftmargin=*]
\item $\dim T_i\le D$, the number of residual variables is bounded in terms of $p,D$,
and the degrees of the representatives are bounded in terms of $p,D$;
\item every stable right-equivalence class occurs;
\item each stable class occurs at only finitely many points of each $T_i$.
\end{enumerate}
The families may be refined so that Milnor number and right modality are constant.
\end{theorem}


Lastly, we apply the bounded polynomial representatives to $F$-singularity invariants. The denominator theorem of Blickle--Mustaţă--Smith then places all $F$-jumping numbers in $(0,1]$ in a finite set depending only on $p$ and the modality bound. The resulting uniform separation of possible thresholds implies that one sufficiently high Frobenius level determines the $F$-pure threshold. The remaining applications use the results of Katzman--Lyubeznik--Zhang and Jiang~\cite{KLZ2011,Jiang2021} on test ideals, and of Blickle--Stäbler~\cite{StaeblerBS} on positive-characteristic Bernstein--Sato roots. 

\begin{theorem}[$F$-singularity consequences]\label{thm:intro-Fapplications}
Fix $p>0$ and $D\ge0$, and let $f$ be isolated with $\rmod(f)\le D$.
\begin{enumerate}[label=\textup{(\roman*)},leftmargin=*]
\item Every $F$-jumping number of $\langle f\rangle$ in $(0,1]$ belongs to a finite set
$\Xi_{p,D}$ depending only on $p,D$; see Theorem~\ref{thm:universal-Fjump-spectrum}.
\item There is an integer $E_{p,D}$ such that the single Frobenius-level invariant
$\nu_{E_{p,D}}(f)$ determines $\fpt(f)$; see Theorem~\ref{thm:one-level-fpt}. 
\item If $\frac{R}{\langle f\rangle}$ is reduced, then the number of $F$-jumping coefficients in
$[0,1]$ and the colength of its finitistic test ideal are bounded in terms of $p,D$.
Moreover, a power of the maximal ideal, with exponent depending only on $p,D$, is
contained in the Jacobian ideal and hence in the finitistic test ideal; see Corollary~\ref{cor:test-uniform}. 
\item Only finitely many eventual sequences of level-$e$ positive-characteristic
Bernstein--Sato root sets occur at fixed $p,D$; Corollary~\ref{cor:finite-BS-sequences}.
\end{enumerate}
In addition, the numerical $F$-jumping spectrum is uniformly $\m$-adically constant in
the sense of Corollary~\ref{cor:uniform-m-adic-constancy}.
\end{theorem}

\textbf{Organization.}  In Section~\ref{sec:matrix} we study matrix
determinacy and prove the prescribed-jet approximation results.  In
Section~\ref{sec:frobenius} we study high-jet stabilizers and define the Frobenius
algebra.  In Section~\ref{sec:boundedness} we use
Lemma~\ref{lem:standard-inputs}\ref{std:split} and the critical cover to obtain uniform
bounds and to prove equality of proper and right modality.  In
Section~\ref{sec:classification} we prove the finite algebraic parametrization theorem.
Finally, in Section~\ref{sec:F-applications} we give the applications to $F$-singularity
invariants.\\

\textbf{Acknowledgments.} We wish to thank Gari Chua, Yotam Hendel, Francesco M. Saettone, and Karl Schwede for  productive
mathematical discussions and their inputs on the different versions of this paper. 

\section{Matrix determinacy over arbitrary fields}\label{sec:matrix}

In this section, we study finite determinacy of matrices over an arbitrary field.  We
first recall the tangent-image criterion of Greuel--Pham and the two questions from
\cite{GLS} that we address.  We then prove the prescribed-jet approximation theorem and
apply it to the full left--right action, to one-sided actions, and to determinantal
ideals.\\

Retain the notation $K,m,n,s,R,\m,M_{m,n},\R,\Gl,\Gr,\Glr$ from the
Introduction.  Put
$P=K[x_1,\ldots,x_s]$, $\m_0=\langle x_1,\ldots,x_s\rangle\subseteq P$, and
$S_0=P_{\m_0}$.  Thus $R$ is the $\m_0$-adic completion of $S_0$ and
$\m=\m_0R$.  We write $M_m(R)=M_{m,m}(R)$, $\height(I)$ for the height of an
ideal $I$, and $\kappa(\mathfrak p)$ for the residue field of $R_{\mathfrak p}$.  The notation $\Grass(r,N)$ denotes the Grassmannian of
$r$-dimensional subspaces of an $N$-dimensional vector space.  If $M$ is a
finite-length module over a local ring $A$, then $\length_A(M)$ denotes its
module length.  For a vector-valued germ $H=(H_1,\ldots,H_q)$, we write
$\JacMat(H)=\left(\frac{\partial H_i}{\partial x_j}\right)_{i,j}$ for its Jacobian matrix. For $\ell\ge0$ put
$R_\ell=\frac{R}{\m^{\ell+1}}$ and $ D_\ell=\dim_K R_\ell=\binom{s+\ell}{s}$. We use multi-index notation throughout.  If
$\alpha=(\alpha_1,\ldots,\alpha_s)\in\mathbb N^s$, then
$|\alpha|=\sum_i\alpha_i$ and $x^\alpha=x_1^{\alpha_1}\cdots x_s^{\alpha_s}$.
For $\beta\le\alpha$, the inequality is coordinatewise and
$\binom{\alpha}{\beta}=\prod_i\binom{\alpha_i}{\beta_i}$.  The symbol $p$ is reserved
for a prime number and, in the hypersurface part, always denotes the characteristic.\\

We use the four matrix actions $\R,\Gl,\Gr,\Glr$ introduced in the
Introduction; these are the conventions of Greuel--Lossen--Shustin
\cite[\S3.1.3]{GLS}.  For $G\in\{\R,\Gl,\Gr,\Glr\}$ and $k\ge0$, the \emph{$k$-jet} of
$A\in\m M_{m,n}$ is its class $j^kA$ modulo $\m^{k+1}M_{m,n}$.  The matrix
$A$ is \emph{$k$-$G$-determined} if every $B\in\m M_{m,n}$ with $j^kB=j^kA$ is
$G$-equivalent to $A$; it is \emph{finitely $G$-determined} if this holds for some $k$.\\

For an algebraic jet-group action, the \emph{tangent image} at a jet is the image of
the differential at the identity of the orbit map.  Passing back to formal power series
gives the tangent-image modules used here.  For example, the extended left--right
tangent image is
$\Ttilde_A^{\,e}(\Glr A)=M_m(R)A+A M_n(R)+\sum_{\nu=1}^sR\,\frac{\partial A}{\partial x_\nu}$,
while the ordinary tangent image has $\m$ instead of $R$ in front of the
derivative terms.  The analogous formulas are obtained by omitting the
right, left, or both matrix terms for $\Gl,\Gr,\R$, respectively;
see \cite[Definition~3.1.11]{GLS} and \cite[\S2]{GP18}.\\

Theorem~3.1.12 of \cite{GLS}, originating in \cite{GP18}, proves that
finite codimension of the appropriate tangent image is sufficient for
finite determinacy over an arbitrary field.  The general problem is the
necessity of this condition.  For the pure right and full left--right
actions, we prove necessity over an arbitrary field; see
Theorems~\ref{thm:pure-right} and~\ref{thm:glr}.  For the one-sided
actions, we prove necessity over every infinite field in
Theorem~\ref{thm:one-sided-complete}.  Over an arbitrary field, including
a finite field, we obtain the one-sided equivalence in the dimension
ranges of Theorem~\ref{thm:one-sided-ranges}.  The one-sided argument over
infinite fields reduces the problem to critical loci of the kernel and
image maps to Grassmannians and eliminates these loci by combining
deep-gauge general position with Frobenius gauges.\\

\medskip

\noindent\textbf{Question A (GLS Conjecture 3.1.16).}
Is the tangent-image criterion also necessary for finite determinacy over every field for
$G\in\{\R,\Gl,\Gr,\Glr\}$?\\

The book records this as open in positive characteristic; see
\cite[Conjecture 3.1.16 and Remark 3.1.16.1]{GLS}.  It is known when the
relevant orbit maps are eventually separable, in particular in
characteristic zero; see \cite[Theorem 3.1.21]{GLS} and
\cite[Theorem 4.3]{GP18}.  It is also known in several special cases,
including scalar functions and one-column left--right matrices under the
hypotheses stated in \cite[Theorems 3.1.9 and 3.1.23]{GLS}.  Theorems~\ref{thm:glr} and~\ref{thm:pure-right} give an affirmative
answer for the full left--right and pure right actions over every field.
For the one-sided actions, Theorem~\ref{thm:one-sided-complete} gives an
affirmative answer over every infinite field, while
Theorem~\ref{thm:one-sided-ranges} gives the corresponding
arbitrary-field statement in the ranges $n\le m$ and $s\le m-n+1$ for $\Gl$, and $m\le n$ and $s\le n-m+1$ for $\Gr$.\\

\begin{remark}
Finite determinacy for smooth and analytic map germs goes back to the work of Mather on
finitely determined map germs \cite{Mather1968}.  For hypersurface functions over
fields of arbitrary characteristic, finite right determinacy is equivalent to finiteness
of the Milnor number by Boubakri--Greuel--Markwig \cite{BoubakriGreuelMarkwig2012}.
For matrices in positive characteristic, Greuel--Pham introduced the tangent image in
order to separate the infinitesimal image of the orbit map from the tangent space of a
possibly inseparable orbit \cite{GP18}.  This distinction is the source of Question~A.
For determinacy and Mather--Yau-type phenomena over complete regular local rings in
mixed characteristic, see \cite{SvorayDVR2025}.
\end{remark}

\vspace{5mm}

For ideals, let $I=\langle h_1,\ldots,h_r\rangle\subseteq R$ be minimally generated.  Two ideals
$I,J\subseteq R$ are \emph{contact equivalent} if a $K$-algebra automorphism of $R$
sends $I$ onto $J$; finite contact determinacy means finite determinacy for this
equivalence, expressed through a minimal generator column.

Let $I=\langle h_1,\ldots,h_r\rangle\subseteq R$ be
minimally generated. We say that $R/I$ is an \emph{isolated complete-intersection singularity} (ICIS),
if $h_1,\ldots,h_r$ is a
regular sequence and $\dim_K
   \left(\frac{R}
   {I+I_r(\JacMat(h_1,\ldots,h_r))}\right)
   <\infty$.  Thus the Jacobian condition holds at every non-maximal prime of $\frac{R}{I}$. For algebraic representatives over a perfect ground field, this is equivalent to the usual punctured smoothness condition over $K$. Note that we also include the nonsingular case. Regard the generators as the column
$A_I=(h_1,\ldots,h_r)^{\mathsf T}\in M_{r,1}(R)$.  Contact equivalence of such ideals agrees with
left--right equivalence of their minimal generator columns; see
\cite[Proposition~3.1.26]{GLS} and \cite[Proposition~4.3]{GP19}.\\

\medskip

\noindent\textbf{Question B (GLS Problem 3.1.29).}
If $K$ is finite and $\frac{R}{I}$ has positive dimension, does finite contact
determinacy imply that $\frac{R}{I}$ is an ICIS?\\

The reverse implication and the equivalence between finite Tjurina number
and the ICIS condition are already valid over arbitrary fields; the
missing case recorded in \cite[Problem 3.1.29]{GLS} is precisely the
finite-field implication from finite contact determinacy.

\begin{lemma}
\label{lem:dense-transfer}
Let a group $G$ act on an $\m$-adically filtered set $E$, and let $P$ be a
$G$-invariant property.  Suppose that for every $A\in E$ and every $N$
there exists $B\in A+\m^N E$ satisfying $P$.  Then every finitely
$G$-determined element of $E$ satisfies $P$.
\end{lemma}

\begin{proof}
If $A$ is $k$-$G$-determined, choose $B\in A+\m^{k+1}E$ satisfying $P$.
Then $A\sim_G B$, so $P(A)$ follows from $G$-invariance.
\end{proof}

\begin{proposition}
\label{prop:generic-rank}
Let $G\in\{\R,\Gl,\Gr,\Glr\}$ and let $A\in\m M_{m,n}$ be finitely $G$-determined.  Then $\rank_{\Frac(R)}A=\min(m,n)$.
\end{proposition}

\begin{proof}
Put $r=\min(m,n)$ and suppose that $A$ is $k$-$G$-determined.  Fix an
$r\times r$ submatrix $A_0$ of $A$, and choose $N\ge k+1$.  For
$\lambda=(\lambda_1,\ldots,\lambda_r)\in K^r$, perturb the diagonal of
that block by
$\lambda_1x_1^N,\ldots,\lambda_rx_1^N$
(with $x_1$ used repeatedly; no distinct variables are needed), leaving
all other entries unchanged.  Call the resulting matrix $B_\lambda$.
Then
$B_\lambda-A\in\m^N M_{m,n}$.
The determinant of the chosen block of $B_\lambda$ is a polynomial
$P(\lambda_1,\ldots,\lambda_r)\in R[\lambda_1,\ldots,\lambda_r]$
which has degree at most one in each $\lambda_i$ and whose coefficient of
$\lambda_1\cdots\lambda_r$ is $x_1^{Nr}\ne0$.\\

A nonzero polynomial of degree at most one in each variable cannot vanish
on all of $K^r$: this follows by induction on $r$ from the one-variable
fact, valid over every field.  Hence for some $\lambda\in K^r$ the chosen
$r\times r$ minor of $B_\lambda$ is nonzero.  Therefore
$\rank_{\Frac(R)}B_\lambda=r$. Finite determinacy gives $B_\lambda\sim_G A$.  Generic rank is invariant
under source automorphisms and under invertible row and column operations.
Therefore $A$ also has generic rank $r$.
\end{proof}

\begin{proposition}
\label{prop:dimension-bounds}
Let $A\in\m M_{m,n}$ be finitely determined.  Then:
\begin{enumerate}[label=\textup{(\roman*)}]
\item if $A$ is finitely $\R$-determined, then
$mn\le s$;
\item if $A$ is finitely $\Gl$-determined, then
$mn\le m^2+s$;
\item if $A$ is finitely $\Gr$-determined, then
$mn\le n^2+s$.
\end{enumerate}
In particular, when $n>m$, finite $\Gl$-determinacy implies $m(n-m)\le s$, and when $m>n$, finite $\Gr$-determinacy implies $n(m-n)\le s$.
\end{proposition}

\begin{proof}
Suppose $A$ is $k$-determined and fix $\ell>k$.  The set of $\ell$-jets
having the same $k$-jet as $A$ is an affine cylinder of dimension
$mn(D_\ell-D_k)$.
For an infinite field, all its $K$-points lie in the $\ell$-jet orbit of
$A$, hence the Zariski closure of that orbit has dimension at least this
number. The $\ell$-jet automorphism group is the open subset of $\m R_\ell^s$
on which the linear part is invertible.  Hence its dimension is exactly
$s(D_\ell-1)$.  Likewise, $\GL_m(R_\ell)$ and $\GL_n(R_\ell)$ are open
subsets of $M_m(R_\ell)$ and $M_n(R_\ell)$, so their dimensions are exactly
$m^2D_\ell$ and $n^2D_\ell$.  Since orbit dimension is at most group
dimension, division by $D_\ell$ followed by $\ell\to\infty$ gives the
asserted inequalities.\\

If $K=\mathbb F_q$, we use cardinalities instead.  The cylinder contains
$q^{mn(D_\ell-D_k)}$ points.  Reduction modulo $\m$ gives
\[
 \begin{aligned}
 |\Aut_K(R_\ell)|&=|\GL_s(\mathbb F_q)|q^{s(D_\ell-1-s)},\\
 |\GL_m(R_\ell)|&=|\GL_m(\mathbb F_q)|q^{m^2(D_\ell-1)},\\
 |\GL_n(R_\ell)|&=|\GL_n(\mathbb F_q)|q^{n^2(D_\ell-1)}.
 \end{aligned}
\]
The orbit has cardinality at most that of the acting jet group.  Taking $\log_q$,
dividing by $D_\ell$, and passing to the limit gives the same bounds.
\end{proof}

We next prove the approximation result over finite fields.  We use the Bertini theorem of Poonen with Taylor conditions to keep a prescribed finite jet fixed while imposing smoothness away from the origin.\\

We first explain what we mean by prescribing Taylor data.  If $Z$ is a finite subscheme of an ambient affine or projective space, by \emph{Taylor
data on $Z$} we mean a prescribed section of $H^0(Z,\mathcal O_Z)$, namely the
restriction of the defining polynomial or hypersurface section to $Z$.  In the case
used repeatedly below, $Z=\Spec\!\left(\frac{P}{\m_0^N}\right)$ is the $(N-1)$-st
infinitesimal neighborhood of the origin, and prescribing Taylor data is exactly
prescribing a residue class modulo $\m_0^N$, equivalently an $(N-1)$-jet.  In
multi-index notation, if $h=\sum_{\alpha\in\mathbb N^s}c_{\alpha}x^{\alpha}$,
then this Taylor datum is represented by
$\sum_{|\alpha|<N}c_{\alpha}x^{\alpha}$.  Thus \emph{Taylor data} here does not mean
a choice of derivatives divided by factorials; it means the scheme-theoretic
restriction to the chosen finite subscheme.  This is the terminology used by Poonen in
his Bertini theorem with Taylor conditions \cite[Theorem~1.2]{Poonen}.

\begin{lemma}
\label{lem:poonen}
Let $K=\mathbb F_q$.  Fix a quasiprojective embedding of a $K$-scheme $X$ into
projective space, let $Z$ be a finite subscheme of the ambient projective space, and suppose $U=X-(X\cap Z)$ is smooth of pure dimension $d>0$.  Given arbitrary prescribed Taylor
data on $Z$, there exists, in sufficiently large degree, a hypersurface
having exactly those Taylor data and whose intersection with $U$ is smooth
of pure dimension $d-1$.
\end{lemma}

\begin{proof}
This is the Bertini theorem of Poonen with Taylor conditions
\cite[Theorem~1.2]{Poonen}.  If $T$ is the singleton consisting of the
prescribed restriction on $Z$, the density appearing in the theorem of Poonen
is positive.  Importantly, $X$ need not be smooth along $Z$; only
$U=X-(X\cap Z)$ is required to be smooth.
\end{proof}

For the stratified argument we use the following direct consequence of the
infinitely-many-local-conditions theorem of Poonen \cite[Theorem~1.3]{Poonen}.  We record it explicitly because the simultaneous
conditions are one of the main inputs in the proof of Theorem~\ref{thm:stratified-jet}.

\begin{lemma}[Local-conditions principle over finite fields]\label{lem:poonen-local}
Let $K=\mathbb F_q$, let $Z$ be a finite subscheme of a projective space, and let $X_1,\ldots,X_u$ be smooth quasiprojective subschemes. Fix Taylor data on $Z$. Then, in sufficiently large degree, there exists a hypersurface realizing those Taylor data such that, away from $Z$, its intersection with every positive-dimensional $X_i$ is smooth of the expected dimension and it avoids every zero-dimensional $X_i$.
\end{lemma}

\begin{proof}
Apply \cite[Theorem~1.3]{Poonen}. At a closed point $P\notin Z$, simultaneous smoothness is a condition on the value and first derivatives of the hypersurface section. For each $X_i$ containing $P$, the bad first jets are those for which the value and all tangent derivatives along $X_i$ vanish. These are precisely the local smoothness conditions occurring in Poonen's closed-point sieve. The prescribed Taylor class changes only the finitely many local conditions supported on $Z$. Poonen's convergence estimates for the sieve therefore give a positive limiting density, and hence such hypersurfaces exist in sufficiently large degree.
\end{proof}

\begin{remark}
Classical Bertini arguments over infinite fields obtain smooth members by removing a
proper closed subset of a linear system.  Over a finite field such an open set need not
have a rational point.  Poonen replaced this genericity argument by a closed-point sieve
and proved positive-density Bertini theorems, including prescribed Taylor conditions
\cite{Poonen}.  Recent work of Bertucci develops further forms of Taylor conditions over
finite fields \cite{Bertucci}; Ajit--Bertucci prove complementary finite-field Bertini
theorems preserving Hilbert--Samuel multiplicity \cite{AjitBertucci2026}.
\end{remark}

\begin{proposition}
\label{prop:icis-density}
Let $K=\mathbb F_q$ and let $r<s$.  For arbitrary
$h_1,\ldots,h_r\in\m$ and every $N\ge1$, there exist $g_1,\ldots,g_r\in\m$ such that $g_i\equiv h_i\pmod{\m^N}$ and $\frac{R}{\langle g_1,\ldots,g_r\rangle}$ is an isolated complete-intersection
singularity.
\end{proposition}

\begin{proof}
The successive construction in the proof of Theorem~\ref{thm:icis}
uses no finite-determinacy assumption until its final transport step.
Applying Lemma~\ref{lem:poonen} successively produces a complete
intersection that is smooth off the origin and realizes the prescribed
jet.  The height and Jacobian arguments given below then show that its
completion is an ICIS.
\end{proof}

\begin{theorem}
\label{thm:icis}
Let $K=\mathbb F_q$, let $R=K[[x_1,\ldots,x_s]]$, and let $I$ be a proper
ideal with $\dim \left(\frac{R}{I}\right)>0$.  If $I$ is finitely contact determined, then
$\frac{R}{I}$ is an isolated complete-intersection singularity.
\end{theorem}

\begin{proof}
Choose a minimal system $I=\langle h_1,\ldots,h_r\rangle$.  We first show that $r<s$.
The field-independent part of \cite[Theorem~3.1.22(2)]{GLS}, equivalently
\cite[Theorem~2.6]{GP19}, implies that if a finitely determined
$r\times1$ column has $r\ge s$, then the ideal generated by its entries
contains a power of $\m$.  This would make $\frac{R}{I}$ zero-dimensional.  Hence
$r<s$.\\

Choose a determinacy order $k$, and choose $N\ge k+1$.  Let
$Z=\Spec\!\left(\frac{P}{\m_0^N}\right)$.
For each $i$, let $h_i^{<N}\in\m_0$ be the polynomial truncation of $h_i$
modulo $\m^N$.  Its class in $\frac{P}{\m_0^N}=H^0(Z,\mathcal O_Z)$ is the Taylor
data to be prescribed.\\

We construct polynomials $g_1,\ldots,g_r$ successively, each satisfying
$g_i\equiv h_i^{<N}\pmod{\m_0^N}$, equivalently $g_i-h_i\in\m^NR$ after
completion.  Put $X_0=\A_K^s$, and after choosing the
first $i$ equations put
$X_i=V(g_1,\ldots,g_i)$.
We require $X_i-\{0\}$ to be smooth of pure dimension $s-i$.\\

The assertion holds for $X_0$.  Assuming it for $X_{i-1}$, apply
Lemma~\ref{lem:poonen} to $X_{i-1}$, with the fixed fat point $Z$ and the
prescribed Taylor class of $h_i^{<N}$.  Since $Z$ is supported at the origin,
the smooth open set in \cite[Theorem~1.2]{Poonen} is precisely
$X_{i-1}-\{0\}$.  We obtain $g_i$ such that $X_i-\{0\}$ is smooth of pure
dimension $s-i$.  Induction gives $X_r-\{0\}$ smooth of dimension
$s-r>0$.\\

Let $J=\langle g_1,\ldots,g_r\rangle S_0$.  Since $J$ has $r$
generators, every minimal prime has height at most $r$.  A minimal prime
cannot be $\m_0$, because $\height\m_0=s>r$.  Therefore each irreducible
component meets the punctured spectrum.  On the punctured spectrum $\frac{S_0}{J}$
is smooth of pure dimension $s-r$, hence each minimal prime has height
exactly $r$.  Thus $J$ has height $r$.  Since $S_0$ is regular, hence Cohen--Macaulay, and $J$ has height $r$, the $r$ generators form a regular sequence \cite[Lemmas~10.106.3 and~10.104.2, Tags~00NQ and~02JN]{stacks-project}.\\

Because a finite field is perfect, the Jacobian criterion applies \cite[Lemma~10.137.15, Tag~00TE]{stacks-project}.  Let
$\JacMat(g)$ be the $r\times s$ Jacobian matrix.  Smoothness of
$X_r-\{0\}$ implies that
$J+I_r(\JacMat(g))$
is either the unit ideal or has radical $\m_0$.  Hence it contains a power
of $\m_0$.  Since completion is faithfully flat, it preserves the regular sequence
\cite[Lemma~10.68.5, Tag~00LM]{stacks-project}; the maximal-ideal-primary Jacobian
condition is preserved as well.  Consequently
$\frac{R}{\langle g_1,\ldots,g_r\rangle}$ is an ICIS in the sense of
\cite[Definition 3.1.27 and Theorem 3.1.28]{GLS}.\\

By construction, every $g_i-h_i$ lies in
$\m^N\subseteq\m^{k+1}$.  Finite contact $k$-determinacy therefore gives
an isomorphism of complete local $K$-algebras
$\frac{R}{I}\cong \frac{R}{\langle g_1,\ldots,g_r\rangle}$.
The complete-intersection property and regularity away from the closed
point are intrinsic properties of this local algebra.  Hence $\frac{R}{I}$ is an
ICIS.
\end{proof}

The converse is already part of the arbitrary-field theory summarized in
\cite[Theorem~3.1.28]{GLS}.  Combining that theorem with the known infinite-field case and Theorem~\ref{thm:icis} gives a field-independent statement.  Thus no
assumption on the cardinality of the ground field remains in the contact-determinacy
criterion.

\begin{corollary}
\label{cor:icis-all-fields}
Let $K$ be an arbitrary field and let $I\subseteq R$ be proper with
$\dim \left(\frac{R}{I}\right)>0$.  Then $I$ is finitely contact determined if and only if $\frac{R}{I}$ is an ICIS.
\end{corollary}

We next consider matrix germs.  The natural stratification for the left--right action is the stratification by rank.  We first record its tangent spaces and then impose transversality to all rank strata simultaneously.\\

Let
$Y=\Mat_{m,n}(K)$ and set $ d=mn$.
For $0\le r\le\min(m,n)$, let $\Sigma_r\subseteq Y$ be the smooth locally
closed variety of matrices of rank exactly $r$.

\begin{lemma}
\label{lem:rank}
For $C\in\Sigma_r$, one has $T_C\Sigma_r=M_m(K)C+C M_n(K)$.
\end{lemma}

\begin{proof}
The rank-$r$ stratum is the
$\GL_m(K)\times\GL_n(K)$-orbit of $C$.  Differentiating the orbit map gives
$M_m(K)C+C M_n(K)\subseteq T_C\Sigma_r$.
Putting $C$ into normal form with an $r\times r$ identity block shows that
the left-hand side consists precisely of matrices whose lower-right
$(m-r)\times(n-r)$ block vanishes.  Its dimension is $r(m+n-r)$, which is
the dimension of $\Sigma_r$.  Therefore equality holds in every characteristic.
\end{proof}

For a morphism $F:W\to V$ between smooth $K$-schemes and a smooth locally
closed subscheme $S\subseteq V$, we say that $F$ is \emph{transverse to $S$ at}
$w\in W$ with $F(w)\in S$ if $dF_w(T_wW)+T_{F(w)}S=T_{F(w)}V$; it is transverse to
$S$ on an open set if this holds at every point of the inverse image of $S$ in that
open set.

\begin{theorem}
\label{thm:stratified-jet}
Let $V=\A_K^d=\coprod_{\alpha=1}^tS_\alpha$ be a finite decomposition into smooth
locally closed $K$-subschemes.  For $A\in\m R^d$ and $N\ge2$, there is
$B\in A+\m^N R^d$ such that:
\begin{enumerate}[label=\textup{(\roman*)},leftmargin=*]
\item $B$ is the completion of an algebraic map $F:W\to V$ from an étale
neighborhood $W\to\A_K^s$ of the origin;
\item $F$ is transverse to every $S_\alpha$ away from the distinguished point.
\end{enumerate}
\end{theorem}

\begin{proof}
Write the coordinates of $V$ as $y_1,\ldots,y_d$, put
$E=\A_K^s\times V$, and let $z=(0,0)$.  Write
$A=(A_1,\ldots,A_d)$, and let $A_\beta^{<N}$ be the polynomial truncation of
$A_\beta$ modulo $\m^N$.\\

We construct polynomial equations $h_1,\ldots,h_d$ recursively so that
$h_\beta$ has the prescribed $(N-1)$-jet
$y_\beta-A_\beta^{<N}(x)$
at $z$.  Put $\Gamma_i=V(h_1,\ldots,h_i)$.  The inductive requirement is
that, away from $z$, $\Gamma_i$ and every nonempty
$\Gamma_i\cap(\A_K^s\times S_\alpha)$
are smooth of the expected dimensions. Assume $h_1,\ldots,h_{i-1}$ have been chosen.  We regard the relevant affine
schemes as quasiprojective subschemes of one projective space and homogenize the
polynomials.  This allows us to use one ambient linear system throughout; no condition is
imposed at infinity.\\

Suppose first that $K$ is finite.  By the inductive hypothesis, the quasiprojective
schemes
\[
 X_0=\Gamma_{i-1}\setminus\{z\},\qquad
 X_\alpha=\bigl(\Gamma_{i-1}\cap(\A_K^s\times S_\alpha)\bigr)\setminus\{z\}
\]
are smooth whenever nonempty; zero-dimensional members are allowed.  We use these finitely
many schemes as the $X_j$ in Lemma~\ref{lem:poonen-local}.  At $z$ we prescribe the unique
coset modulo the $N$-th power of the maximal ideal giving the required $(N-1)$-jet.  This
coset is open and closed in the completed local ring and hence has positive Haar measure and
boundary of measure zero.\\

Now let $P\ne z$ be a closed point.  For every positive-dimensional $X_j$ containing $P$ we
allow precisely those first jets for which either $h_i(P)\ne0$, or $h_i(P)=0$ and
\[
 dh_i(P)|_{T_PX_j}\ne0.
\]
This is exactly the condition that the hypersurface section of $X_j$ be smooth of dimension
$\dim X_j-1$ at $P$.  If $X_j$ is zero-dimensional, the expected dimension is $-1$, and we
require $h_i(P)\ne0$.  At points lying on none of the $X_j$ there is no condition.  For each fixed $X_j$ containing $P$, the bad first jets form a proper linear subspace of the value-and-first-derivative space. Imposing the condition simultaneously for all $X_j$ through $P$ removes a finite union of such proper linear subspaces. In particular, the allowed set is a finite-jet cylinder of positive local measure—every jet with $h_i(P)\ne0$ is allowed—and its boundary has measure zero.  At a point of a stratum $S_\alpha$, the tangent space of $\Gamma_{i-1}\cap(\mathbb A_K^s\times S_\alpha)$ is contained in $T_P\Gamma_{i-1}$. Hence, whenever both conditions occur, the imposed condition on the stratified intersection is compatible with the ambient smoothness condition.\\

All hypotheses of Lemma~\ref{lem:poonen-local} are therefore satisfied, and an admissible
$h_i$ exists.  This is the same closed-point-sieve mechanism used in
\cite[Theorem~3.3]{Poonen}, now with the finite Taylor condition at $z$ kept independent of
the punctured smoothness conditions.\\

Now suppose that $K$ is infinite.  Let $L_D$ be the affine space of polynomials of degree
at most $D$ with the prescribed jet at $z$.  For $D\gg0$, its translation space separates
first jets at every geometric point of $E\setminus\{z\}$.  Indeed, at a point
$P\ne z$ one may choose a linear form vanishing at $z$ and nonvanishing at $P$, raise it
to the $N$-th power, and multiply by affine linear functions to obtain arbitrary value and
tangent data at $P$.  Let $X$ be one of the finitely many smooth schemes occurring in the
induction and put $e=\dim X$.  For fixed $P\in X$, failure of a hypersurface section to
have the expected smoothness is given by the $e+1$ independent linear conditions $h_i(P)=0$ and $dh_i(P)|_{T_PX}=0$ when $e>0$, and by the single condition $h_i(P)=0$ when $e=0$.  Hence the corresponding incidence variety in $L_D\times X$ has dimension at most $\dim L_D-1$. The Zariski closure of its projection in $L_D$ therefore has dimension at most $\dim L_D-1$, and is a proper closed subset.  Removing the finitely many bad
projections leaves a nonempty Zariski-open subset of $L_D$, and an infinite field has a
$K$-rational point in this open set.  Choosing that point completes the induction.\\

After $d$ steps put $\Gamma=\Gamma_d$.  At $z$ we have
$\left(\frac{\partial h_\alpha}{\partial y_\beta}\right)(z)=I_d$,
so the projection $\Gamma\to\A_K^s$ is étale at $z$.  The formal
implicit-function theorem identifies the completed germ of $\Gamma$ at
$z$ with the graph of a unique formal map $B(x)$, and the prescribed jets
give $B-A\in\m^N R^d$.  After shrinking to an étale neighborhood
$W$ of the distinguished point, the $y$-coordinate map $F:W\to V$ is
transverse to every $S_\alpha$ away from that point.
\end{proof}

\begin{theorem}
\label{thm:jettrans}
Let $K$ be an arbitrary field, let
$A\in\m M_{m,n}(K[[x_1,\ldots,x_s]])$, and let $N\ge2$.  There exists $B$
with the same $(N-1)$-jet as $A$ such that the extended $\Glr$-tangent
image of $B$ has finite $K$-codimension.
\end{theorem}

\begin{proof}
Apply Theorem~\ref{thm:stratified-jet} to the finite rank stratification
$Y=\coprod_r\Sigma_r$.  At every point away from the origin, transversality and
Lemma~\ref{lem:rank} therefore give that
$dB(T)+M_m B+B M_n=M_{m,n}$.
Equivalently, the coherent homomorphism
\[
M_m(\mathcal O_W)\oplus M_n(\mathcal O_W)\oplus T_W
\longrightarrow M_{m,n}(\mathcal O_W),
\qquad
(U,V,\xi)\longmapsto UB+BV+dB(\xi),
\]
is surjective away from the distinguished point.  Its cokernel is
therefore supported at that point and has finite length.  Completion
identifies its image with the extended tangent image of the formal germ
$B$.
\end{proof}

The ordinary tangent image differs from the extended tangent image only
by replacing the $R$-span of the finitely many derivative matrices by
their $\m$-span.  Their quotient is a quotient of a finite-dimensional
$K$-space generated by the classes of the $s$ derivative matrices.
Therefore finite codimension of the ordinary and extended tangent images
are equivalent.

\begin{corollary}
\label{cor:glr-dense}
For every $A\in\m M_{m,n}$ and every
$N$ there exists $B\equiv A\pmod{\m^N}$ that is finitely
$\Glr$-determined.
\end{corollary}

\begin{proof}
Theorem~\ref{thm:jettrans} gives $B$ with finite tangent-image codimension.
The arbitrary-field sufficient criterion
\cite[Theorem~3.1.12]{GLS} then gives finite determinacy.
\end{proof}

\begin{theorem}
\label{thm:glr}
For every field $K$ and every
$A\in\m M_{m,n}(K[[x_1,\ldots,x_s]])$, finite $\Glr$-determinacy implies
finite codimension of the tangent image.  Together with the sufficient
criterion of Greuel--Pham, this gives the equivalence predicted by
Conjecture~3.1.16 for the full left--right group.
\end{theorem}

\begin{proof}
Suppose $A$ is $k$-$\Glr$-determined, and choose $N\ge k+1$.  By
Theorem~\ref{thm:jettrans}, choose $B$ with
$B-A\in\m^N M_{m,n}$ and with tangent image of finite codimension.  Finite
determinacy gives $A\sim_{\Glr}B$. Finite codimension of the tangent image is invariant under
$\Glr$-equivalence, by the chain rule and invertibility of the row,
column, and coordinate transformations.  Hence the tangent image of $A$
has finite codimension.  The converse is the arbitrary-characteristic
sufficient criterion \cite[Theorem 3.1.12]{GLS}.
\end{proof}

Combining the necessity proved above with the arbitrary-characteristic sufficient
criterion gives a field-independent equivalence for the full left--right action.  We
record it explicitly because it will be used in the one-sided and determinantal
arguments below:
\[
A\text{ finitely }\Glr\text{-determined}
\iff
\dim_K \left(\frac{M_{m,n}}{\Ttilde_A^{\,e}(\Glr A)}\right)<\infty.
\]
This extends the previously known one-column result and the $2\times2$ infinite-field
result; see \cite[Theorem 3.1.23]{GLS} and \cite{PMM20}.  For the corresponding tangent-image formalism for right, contact, and left--right
equivalence, see also \cite[\S3.1]{GLS}.\\

The argument for matrices uses only the existence of a finite stratification whose
tangent spaces are infinitesimal orbit directions.  We record this observation in an
abstract form before returning to the one-sided matrix actions.  Let $H$ be a smooth
affine algebraic group acting on a finite-dimensional $K$-vector space $V$.  For
$v\in V$, write $\mathfrak h\cdot v=\im\bigl(d_e(H\to H\cdot v)\bigr)\subseteq T_vV$.

\begin{definition}\label{def:infinitesimal-orbit-stratification}
A finite smooth stratification
$V=\coprod_{\alpha=1}^t S_\alpha$
is called an \emph{infinitesimal-orbit stratification} if, after every
field extension, at every point $v\in S_\alpha$ we have that
$\mathfrak h\cdot v=T_vS_\alpha$.
\end{definition}

The \emph{current group} associated with the action $H\curvearrowright V$ is
$\mathcal G_H=H(R)\rtimes\Aut_K(R)$.  Here an element of $H(R)$ acts pointwise on
the target value of a formal map $A:(K^s,0)\to V$, while
$\Aut_K(R)$ acts by change of source coordinates; the semidirect product records the
resulting change of the $R$-valued point of $H$.  Its extended infinitesimal
current-orbit module is
$\Ttilde_A^{\,e}=R\langle\mathfrak h\cdot A\rangle+\sum_{\nu=1}^sR\,\frac{\partial A}{\partial x_\nu}$.

\begin{theorem}
\label{thm:abstract-current}
Assume that $V$ admits a finite infinitesimal-orbit stratification.  Then,
by Theorem~\ref{thm:stratified-jet}, every finitely
$\mathcal G_H$-determined formal map $A$ satisfies
\[
\dim_K \left(\frac{V\otimes_KR}{\Ttilde_A^{\,e}}\right)<\infty.
\]
Whenever finite codimension of this infinitesimal module is known to be
sufficient for finite determinacy for the action under consideration, it
is therefore equivalent to finite determinacy.
\end{theorem}

\begin{proof}
Theorem~\ref{thm:stratified-jet} shows that every $\m$-adic cylinder around
$A$ contains a germ $B$ transverse away from the origin to all strata.
At a point with $B(x)\in S_\alpha$, we have that
$dB_x(T_x)+T_{B(x)}S_\alpha=T_{B(x)}V$.
By the infinitesimal-orbit hypothesis, we therefore get that
$dB_x(T_x)+\mathfrak h\cdot B(x)=T_{B(x)}V$.
Therefore the coherent cokernel of the infinitesimal current-orbit map is
supported only at the origin, so its completion has finite length.
Finite codimension of the infinitesimal module is invariant under the
current action.  Lemma~\ref{lem:dense-transfer} now transports the property
from $B$ to every finitely determined $A$.
\end{proof}

\begin{remark}
For $V=M_{m,n}$ and $H=\GL_m\times\GL_n$, the rank strata form an
infinitesimal-orbit stratification by Lemma~\ref{lem:rank}.  Thus
Theorem~\ref{thm:glr} is the basic matrix instance of
Theorem~\ref{thm:abstract-current}.
\end{remark}

We next study the one-sided matrix actions.  Once rank transversality is known, the remaining directions are measured by the kernel and image maps to Grassmannians.\\

For $1\le t\le\min(m,n)$, let $I_t(A)\subseteq R$ denote the \emph{$t$-th
determinantal ideal}, generated by the $t\times t$ minors of $A$.  If $A$ is a formal
matrix germ, we say that it is \emph{transverse to the rank strata on the punctured
formal spectrum} if for every nonmaximal prime $\mathfrak p\subseteq R$, after base
change to the residue field $\kappa(\mathfrak p)$, the induced map satisfies the tangent-space
transversality condition above at the corresponding point.  Equivalently, the
left--right infinitesimal orbit map is surjective after localization at every
$\mathfrak p\ne\m$.\\

For $C$ in the rank-$r$ stratum $\Sigma_r\subseteq M_{m,n}$ define
$\kappa_r(C)=\ker C\in\Grass(n-r,n)$ and
$\iota_r(C)=\im C\in\Grass(r,m)$.\\

When a one-dimensional completed determinantal locus is reduced, by a \emph{formal
branch} we mean one of its irreducible components, equivalently one of its minimal
primes.  At the generic point $\mathfrak q$ of such a branch, the \emph{continuous
cotangent space} means the one-dimensional cotangent space obtained from continuous
$K$-differentials on the completed local algebra and then tensoring with
$\kappa(\mathfrak q)$.  In dimension one every punctured prime of a reduced local ring is
minimal, so checking a differential on these branch generic points checks it on the whole
punctured spectrum.

\begin{lemma}\label{lem:grass-quotient}
At $C\in\Sigma_r$, the kernel of $d\kappa_r$ is $M_mC$ and the kernel of
$d\iota_r$ is $CM_n$.  Hence the quotient of $T_C\Sigma_r$ by the left matrix orbit
is $T_{\kappa_r(C)}\Grass(n-r,n)$, while the corresponding quotient by the right
matrix orbit is $T_{\iota_r(C)}\Grass(r,m)$.
\end{lemma}

\begin{proof}
The fibers of $\kappa_r$ are the $\GL_m$-orbits inside $\Sigma_r$, while the fibers
of $\iota_r$ are the $\GL_n$-orbits.  Both maps are smooth homogeneous fibrations.
Differentiating gives the statement.
\end{proof}

\begin{proposition}
\label{prop:one-sided-criterion}
Assume that $A$ is transverse to all rank strata on the punctured spectrum.  Then the
extended $\Gl$ tangent image has finite codimension exactly when every induced kernel
map $\kappa_r\circ A$ is submersive on $A^{-1}(\Sigma_r)$ away from the origin.
Dually, the extended $\Gr$ tangent image has finite codimension exactly when all
induced image maps $\iota_r\circ A$ are submersive there.
\end{proposition}

\begin{proof}
Rank transversality supplies all normal directions to $\Sigma_r$.  By
Lemma~\ref{lem:grass-quotient}, the directions missing from the left matrix orbit are
exactly the tangent directions of the kernel Grassmannian, and they are supplied by the
differential of $\kappa_r\circ A$.  Thus localized surjectivity of the extended left
tangent image is exactly submersivity of the kernel map.  The right case is dual.
Finite codimension is equivalent to localized surjectivity at every nonmaximal prime.
\end{proof}

\begin{corollary}
\label{cor:one-sided-reduction}
If $A$ is finitely $\Gl$-determined, then the only possible obstruction at this stage to finite
codimension of its tangent image is a punctured critical locus of the kernel maps.
Dually, for a finitely $\Gr$-determined matrix the only possible obstruction at this stage is a
punctured critical locus of the image maps.
\end{corollary}

\begin{proof}
Finite one-sided determinacy implies finite $\Glr$-determinacy.  Theorem~\ref{thm:glr}
and Lemma~\ref{lem:rank} therefore give rank-stratum transversality away from the
origin.  Apply Proposition~\ref{prop:one-sided-criterion}.
\end{proof}

\begin{proposition}
\label{prop:one-sided-full-rank}
If $n\le m$ and $I_n(A)$ is $\m$-primary, then $A$ is finitely
$\Gl$-determined and its extended $\Gl$ tangent image has finite codimension.
Dually, if $m\le n$ and $I_m(A)$ is $\m$-primary, then $A$ is finitely
$\Gr$-determined.
\end{proposition}

\begin{proof}
At a punctured point, full column rank means that $C:K^n\to K^m$ is injective.  Every
$B:K^n\to K^m$ can then be written $B=UC$ for some $U\in M_m(K)$, so the left
matrix directions already fill the target.  Therefore the tangent cokernel has finite length,
and the sufficient tangent-image criterion \cite[Theorem~3.1.12]{GLS} gives finite
determinacy.  The right
case is the transpose.
\end{proof}

\begin{theorem}
\label{thm:one-sided-ranges}
If $n\le m$ and $s\le m-n+1$, then finite $\Gl$-determinacy, $\m$-primaryness of
$I_n(A)$, and finite codimension of the extended $\Gl$ tangent image are equivalent.
Dually, if $m\le n$ and $s\le n-m+1$, the analogous three conditions for $\Gr$ and
$I_m(A)$ are equivalent.
\end{theorem}

\begin{proof}
Only necessity of the determinantal condition remains.  Finite one-sided determinacy
implies finite $\Glr$-determinacy.  Theorem~\ref{thm:det-geometry-new} gives
$\height I_n(A)=\min\{s,m-n+1\}=s$ in the left case, hence $I_n(A)$ is
$\m$-primary.  The right case is dual.  The remaining implications are
Proposition~\ref{prop:one-sided-full-rank} and the sufficient criterion
\cite[Theorem~3.1.12]{GLS}.
\end{proof}

\begin{lemma}\label{lem:deep-gauge-rigidity}
\begin{enumerate}[label=\textup{(\roman*)},leftmargin=*]
\item If $A$ is $k$-$\Gl$-determined and $V-I_n\in\m^kM_n(R)$, then $AV\sim_{\Gl}A$, and on every rank-$r$ locus the kernel maps differ by a source automorphism.
\item Dually, if $A$ is $k$-$\Gr$-determined and $U-I_m\in\m^kM_m(R)$, then $UA\sim_{\Gr}A$, and on every rank-$r$ locus the image maps differ by a source automorphism.
\end{enumerate}
\end{lemma}

\begin{proof}
Since $A\in\m M_{m,n}$, the condition $V-I_n\in\m^kM_n(R)$ gives
$AV-A\in\m^{k+1}M_{m,n}$.  Finite $\Gl$-determinacy therefore gives
$AV=U\varphi(A)$ for some $U\in\GL_m(R)$ and $\varphi\in\Aut_K(R)$.  Right
multiplication by $V$ and left multiplication by $U$ do not change matrix rank, and
$I_t(AV)=I_t(A)$ while $I_t(U\varphi(A))=\varphi(I_t(A))$.  Hence $\varphi$ preserves
every determinantal rank locus.  On the rank-$r$ locus, left multiplication does not
change the kernel and right multiplication sends the kernel by $V^{-1}$.  Thus the two
kernel maps differ only by the source automorphism induced by $\varphi$.  The proof for
$\Gr$ is obtained by transposing the argument: left gauges act on image subspaces,
whereas right multiplication does not change them.
\end{proof}

In positive characteristic, Frobenius gives sufficiently deep gauges whose
differentials vanish.  We first record the only form of this observation that
will be needed below.  For $N\ge1$, call
$V\in\GL_n(R)$ \emph{$N$-deep} if $V-I_n\in\m^NM_n(R)$. 
For a polynomial local gauge over $S_0$, we use the same terminology with
$\m_0S_0$ in place of $\m$.

\begin{lemma}\label{lem:frobenius-gauge-flexibility}
Assume that $\Char K=p>0$ and let $N\ge1$.  Let
$\mathfrak q\ne\m$ be a prime of $R$, and put $L=\Frac\left(\frac{R}{\mathfrak q}\right)$.  Then the values at $\mathfrak q$ of $N$-deep polynomial local gauges
$V\in\GL_n(S_0)$ satisfying $dV=0$ are Zariski dense in $\GL_n(L)$.
\end{lemma}

\begin{proof}
Put $\eta=\mathfrak q\cap S_0$. 
Since $\mathfrak q\ne\m$, we have that $\eta\ne\m_0$.  Therefore we may choose $a\in\m_0^N\setminus\eta$.
For $B=[b_{ij}]\in M_n(S_0)$, write $B^{[p]}=(b_{ij}^p)$.
and set $V_B=I_n+a^pB^{[p]}$. 
Since $a\in\m_0$, we get that $\det(V_B)\equiv1\pmod{\m_0}$, so $V_B\in\GL_n(S_0)$.  Moreover,
\[
V_B-I_n\in\m_0^{Np}M_n(S_0)\subseteq\m_0^NM_n(S_0),
\]
and, since all entries added to the identity are $p$-th powers, we have that $dV_B=0$.  The domain $S_0/\eta$ is infinite.  Indeed, this is immediate when $K$ is
infinite, while when $K$ is finite it follows from
$\eta\ne\m_0$, because then $S_0/\eta$ has positive dimension.  Frobenius is
injective on this domain, and therefore its set of $p$-th powers is infinite.
Consequently its image in $L$ is Zariski dense in $\A_L^1$. It follows entrywise that the values $I_n+\overline a^{\,p}B(\mathfrak q)^{[p]}$ are Zariski dense in $M_n(L)$.  Intersecting with the nonempty open subset
$\GL_n(L)$ therefore proves the assertion.
\end{proof}

For a morphism $F:X\to Y$ between smooth schemes, we write
$\Crit(F)$ for the degeneracy subscheme on which $dF$ is not surjective.
Thus, when $\dim X=d$ and $\dim Y=t\le d$, the critical scheme is defined by
the maximal minors of a local matrix for $dF:T_X\longrightarrow F^*T_Y$.\\

The following general-position lemma replaces the generic-smoothness step.
Its point is that we control the dimension of the critical locus in the
\emph{source}; no assertion concerning an open subset of the target is
required.

\begin{lemma}
\label{lem:deep-gauge-general-position}
Let $K$ be infinite, let $U\subseteq\mathbb A_K^s$ be an
open neighborhood of the origin, and let
$X\subseteq U\setminus\{0\}$ be a smooth locally closed
$K$-subscheme of pure dimension $d$.
Let $F:X\longrightarrow\Grass(a,n)$ where $t=a(n-a)$ and assume $1\le t\le d$. For every $N\ge1$, there exists a polynomial matrix $V=I_n+H$ with $H\in M_n\bigl(\langle x_1,\ldots,x_s\rangle^N K[x_1,\ldots,x_s]\bigr)$ 
such that on $X_V=X\cap D(\det V)$ we have that $\dim\Crit(V^{-1}\!\cdot F)\le t-1$.  In particular, $V$ defines an $N$-deep element of
$\GL_n(S_0)$. If $t=0$, the critical locus is empty.
\end{lemma}

\begin{proof}
Fix $D\ge N+1$, and let $ E_D=\langle x_1,\ldots,x_s\rangle^N \cap K[x_1,\ldots,x_s]_{\le D}$ and $\mathcal{P}_D=M_n(E_D)$. For every geometric point $x\ne0$, evaluation of first jets
gives a surjection $E_D\otimes_K\kappa(x)
   \longrightarrow
   \mathcal O_{\mathbb A^s_{\kappa(x)},x}/\mathfrak m_x^2$. Indeed, choose $x_\nu(x)\ne0$. Multiplication by $x_\nu^N$
is invertible on the first-jet algebra, and affine-linear
polynomials supply all first jets.\\

Consider the open subset $\Omega= \{(x,H)\in X\times\mathcal P_D \colon \det(I_n+H(x))\ne0\}$.
On $\Omega$, the expression $(I_n+H)^{-1}\!\cdot F$ is
defined. Fix a geometric point $x$ and an invertible value
$C=I_n+H(x)$. Keeping this value fixed, the first derivatives
of $H$ can vary arbitrarily. Since the orbit map $\GL_n\longrightarrow\Grass(a,n)$ is smooth, these variations induce arbitrary variations of
the differential in $\Hom_{\kappa(x)} \bigl(T_xX,T_{C^{-1}\cdot F(x)}\Grass(a,n)\bigr)$. The locus of maps from a $d$-dimensional space to a
$t$-dimensional space that are not surjective has codimension
$d-t+1$.\\

Consequently, the critical incidence locus
$\mathcal Z\subseteq\Omega$ satisfies
\[
   \dim\mathcal Z
   \le\dim\mathcal P_D+d-(d-t+1)
   =\dim\mathcal P_D+t-1.
\]
The generic fiber-dimension theorem, applied to the finitely
many irreducible components of $\mathcal Z$, gives a
nonempty open subset $\mathcal U\subseteq\mathcal P_D$
over which all fibers have dimension at most $t-1$.
Since $K$ is infinite, choose $H\in\mathcal U(K)$.
Finally, $V(0)=I_n$, so $\det V$ is a unit in $S_0$.
\end{proof}

We can now prove the one-sided criterion over infinite fields.  The argument
uses two different kinds of deep gauges.  Lemma~\ref{lem:deep-gauge-general-position}
produces a gauge for which the critical locus has dimension strictly smaller
than the Grassmannian.  Lemma~\ref{lem:frobenius-gauge-flexibility} then shows
that a nonempty critical locus would nevertheless have to move densely in the
Grassmannian.  These two conclusions are incompatible.

\begin{theorem}\label{thm:one-sided-complete}
Let $A\in\m M_{m,n}$. If $K$ is infinite, then:
\begin{enumerate}[label=\textup{(\roman*)},leftmargin=*]
\item $A$ is finitely $\Gl$-determined if and only if $\dim_K\left(
\frac{M_{m,n}}
{\Ttilde_A^{\,e}(\Gl A)}
\right)<\infty$.
\item $A$ is finitely $\Gr$-determined if and only if $\dim_K\left(
\frac{M_{m,n}}
{\Ttilde_A^{\,e}(\Gr A)}
\right)<\infty$. 
\end{enumerate}
The same equivalences hold for the ordinary tangent images. Over an arbitrary field, including a finite field, the corresponding
equivalences hold in the dimension ranges of
Theorem~\ref{thm:one-sided-ranges}.
\end{theorem}

\begin{proof}
Finite codimension of the tangent image is sufficient for finite determinacy
over an arbitrary field by the Greuel--Pham finite-determinacy criterion
\cite[Theorem~3.1.12]{GLS}.  Over a finite field, the assertion in the
dimension ranges stated above is precisely
Theorem~\ref{thm:one-sided-ranges}, together with the equivalence between
ordinary and extended tangent-image codimension.  It therefore remains to
prove necessity when $K$ is infinite. In characteristic zero, necessity follows from separability of the finite-jet
orbit maps; see \cite[Theorem~3.1.21]{GLS} and
\cite[Theorem~4.3]{GP18}. We may therefore assume that $\Char K=p>0$. We prove the assertion for $\Gl$.  The $\Gr$ case follows by transposition.\\

Assume that $A$ is finitely $\Gl$-determined, and let $k$ be a determinacy
order.  Replacing $A$ by a polynomial truncation in its $\Gl$-orbit, we may
suppose that the entries of $A$ belong to $S_0$.  We use the same symbol $A$
for this algebraic representative and for its completion. Finite $\Gl$-determinacy implies finite $\Glr$-determinacy.  Therefore
Theorem~\ref{thm:glr}, together with
Corollary~\ref{cor:rank-trans}, shows that $A$ is transverse to every
constant-rank stratum on the punctured formal spectrum.\\ 

Fix a rank $r$.  If the rank-$r$ locus has no punctured point, there is
nothing to prove.  We therefore assume that it does.  Put $c=(m-r)(n-r)$ and $d=s-c$.  By Theorem~\ref{thm:det-geometry-new}, the rank-at-most-$r$ locus has the
expected codimension $c$, and its punctured rank-$r$ stratum is smooth of
pure dimension $d$. Let $F=\kappa_r\circ A$ be the kernel map on the rank-$r$ stratum, and put $t=\dim\Grass(n-r,n)=r(n-r)$. If $t=0$, the Grassmannian is a point, and therefore there is no kernel-map
obstruction.  Thus assume that $t>0$.\\

We carry out all incidence and image-dimension arguments on
algebraic representatives in an open neighborhood
$U\subseteq\mathbb A_K^s$ of the origin, before localizing.
After shrinking $U$, the punctured rank-$r$ locus is smooth
of pure dimension $d=s-(m-r)(n-r)$.  When passing to the local ring and its completion, we use
the reduced closures of the relevant locally closed loci.
Their completed germs are identified by the formal source
automorphisms arising from finite determinacy. Completion
preserves the dimension of these closed local germs.
Components of an algebraic closure that do not contain the
origin are removed by shrinking $U$.
Thus the resulting dimension bounds also bound the
algebraic image closures used below.\\

We first show that $t\le d$. 
Let $X_{r,1},\ldots,X_{r,a}$ be the finitely many algebraic irreducible components of the punctured
rank-$r$ locus, and let $Y_j=\overline{F(X_{r,j})}
\subseteq\Grass(n-r,n)$ 
be their image closures.  Since $\dim X_{r,j}=d$, 
we have that $\dim Y_j\le d$ for every $j$. Choose a generic prime $\mathfrak q$ of a formal irreducible component of
the completed rank-$r$ locus, and put $L=\Frac\left(\frac{R}{\mathfrak q}\right)$. \\

Let $V$ be a sufficiently deep polynomial gauge with $dV=0$.
Lemma~\ref{lem:deep-gauge-rigidity} gives a formal source automorphism
$\varphi_V$ preserving every determinantal rank locus such that, on the
completed rank-$r$ locus we have that $V^{-1}\!\cdot F=F\circ\varphi_V$. 
The automorphism $\varphi_V$ permutes the finitely many formal irreducible
components.  Every such formal component lies over one of the algebraic
components $X_{r,j}$.  Therefore $V(\mathfrak q)^{-1}\!\cdot F(\mathfrak q)
\in
\bigcup_{j=1}^a (Y_j)_L$.  By Lemma~\ref{lem:frobenius-gauge-flexibility}, the possible values
$V(\mathfrak q)$ are Zariski dense in $\GL_n(L)$.  Since $\GL_n$ acts
transitively on $\Grass(n-r,n)$, the corresponding points $V(\mathfrak q)^{-1}\!\cdot F(\mathfrak q)$
are Zariski dense in the Grassmannian after base change to $L$.
Consequently
\[
\Grass(n-r,n)_L
\subseteq
\bigcup_{j=1}^a(Y_j)_L.
\]
The Grassmannian is irreducible.  Hence one of the $Y_j$ is the whole
Grassmannian.  Therefore $t\le d$, 
as claimed.\\

We may now apply
Lemma~\ref{lem:deep-gauge-general-position} to the punctured algebraic
rank-$r$ locus.  Taking the depth at least $k$, we obtain some $V\in\GL_n(S_0)$ such that $
V-I_n\in\m_0^kM_n(S_0)$ and $\dim\Crit(V^{-1}\!\cdot F)\le t-1$. Since $A\in\m M_{m,n}$, we have that $AV-A\in\m^{k+1}M_{m,n}$. 
Finite $\Gl$-$k$-determinacy therefore gives $AV\sim_{\Gl}A$.  By Lemma~\ref{lem:deep-gauge-rigidity}, the corresponding completed kernel
maps differ by a formal source automorphism.  The chain rule therefore
identifies their completed critical schemes.  Completion preserves the
dimension of the corresponding local schemes, and hence the critical locus
of the original completed kernel map also satisfies $\dim\Crit(F)\le t-1$. \\

We claim that this critical locus is in fact empty on the punctured spectrum.
Suppose otherwise.  Let $Z_1,\ldots,Z_b$
be the algebraic irreducible components of the critical locus of $F$ on the
punctured rank-$r$ stratum.  The preceding bound gives $\dim Z_i\le t-1$ for every $i$.  Put $W_i=\overline{F(Z_i)}
\subseteq\Grass(n-r,n)$.
Thus $\dim W_i\le t-1$. \\

Choose a generic prime $\mathfrak q$ of a formal irreducible component of
the completed critical locus.  Let $V$ now range over sufficiently deep
polynomial gauges satisfying $dV=0$. 
Therefore, the differential of $V^{-1}\!\cdot F$ is obtained from $dF$
by composition with an invertible linear map on the target tangent space.
Therefore the critical scheme of $V^{-1}\!\cdot F$ is exactly the critical
scheme of $F$.

On the other hand, Lemma~\ref{lem:deep-gauge-rigidity} gives $V^{-1}\!\cdot F=F\circ\varphi_V$. 
Thus $\varphi_V$ preserves the completed critical locus and permutes its
finitely many formal irreducible components.  Therefore, for $L=\Frac\left(\frac{R}{\mathfrak q}\right)$,  we get that 
\[
V(\mathfrak q)^{-1}\!\cdot F(\mathfrak q)
\in
\bigcup_{i=1}^b(W_i)_L.
\]

Lemma~\ref{lem:frobenius-gauge-flexibility} again shows that the values
$V(\mathfrak q)$ are Zariski dense in $\GL_n(L)$.  Hence the left-hand side
is Zariski dense in $\Grass(n-r,n)_L$.  We would therefore obtain $\Grass(n-r,n)_L
\subseteq
\bigcup_{i=1}^b(W_i)_L$.
This is impossible since the left-hand side is irreducible of dimension
$t$, whereas every $W_i$ has dimension at most $t-1$.  Thus we can conclude that $\Crit(F)=\varnothing$
on the punctured rank-$r$ locus. Therefore, we can conclude that every kernel map of the form  $\kappa_r\circ A$ is submersive on the corresponding punctured constant-rank stratum.
Proposition~\ref{prop:one-sided-criterion} therefore implies that the
extended $\Gl$ tangent image has finite $K$-codimension. Finally, the ordinary tangent image differs from the extended tangent image
only by replacing the $R$-span of the finitely many derivative matrices by
their $\m$-span.  The quotient between these two derivative contributions is
finite-dimensional over $K$.  Thus finite codimension of the ordinary and
extended tangent images are equivalent.
\end{proof}


\begin{remark}
Over an infinite field,
Theorem~\ref{thm:glr}, Theorem~\ref{thm:pure-right}, and
Theorem~\ref{thm:one-sided-complete} give necessity and sufficiency of the
tangent-image criterion for all four actions, namely $\R, \Gl, \Gr,\Glr$. 
Over a finite field, the full left--right and pure right criteria remain
valid over arbitrary fields, while the one-sided equivalence is established
in the dimension ranges of Theorem~\ref{thm:one-sided-ranges}.  The
unrestricted finite-field one-sided case would require an additional
finite-field general-position theorem replacing
Lemma~\ref{lem:deep-gauge-general-position}.
\end{remark}

\begin{remark}
Finite codimension of the tangent image can be checked by the
standard-basis methods developed by Greuel--Pham
\cite{GreuelPham2020}.  Consequently, over an infinite field the preceding
theorems give a decision criterion for finite determinacy for all four
matrix actions whenever the ground field is given in an effective form.
Over finite fields, the same conclusion holds for the full left--right and
pure right actions and, for the one-sided actions, in the ranges of
Theorem~\ref{thm:one-sided-ranges}.
\end{remark}

We next consider pure right equivalence for vector-valued germs.  When the target has
dimension greater than one, finite right determinacy is much more rigid.  For a
vector-valued germ $F\in\m R^q$, \emph{pure right equivalence} means equivalence
under source coordinate changes only: $F$ and $G$ are right equivalent if
$G=F\circ\varphi$ for some $\varphi\in\Aut_K(R)$; no target coordinate change is
allowed.  Write $F=(F_1,\ldots,F_q)$ and let
$\JacMat(F)=\left(\frac{\partial F_i}{\partial x_j}\right)_{i,j}$ be its $q\times s$
Jacobian matrix.

\begin{lemma}\label{lem:deep-shear-new}
Let $T=K[[y_1,\ldots,y_q]]$ with maximal ideal $\mathfrak m_T$.  Suppose a proper ideal
$D\subseteq T$ is invariant under every sufficiently deep triangular automorphism and
contains a nonzero polynomial.  Then $\sqrt D=\mathfrak m_T$.
\end{lemma}

\begin{proof}
Choose $0\ne P\in D\cap K[y_1,\ldots,y_q]$.  Fix $j$.  For $i\ne j$ choose
integers $M_i$ larger than the depth from which triangular automorphisms preserve $D$
and, successively, so large that the map
$\alpha\longmapsto
  \alpha_j+\sum_{i\ne j}M_i\alpha_i.$
is injective on the finite support of $P$.  Apply the triangular automorphism
$\sigma_j$ given by $y_i\mapsto y_i+y_j^{M_i}$ for $i\ne j$ and
$y_j\mapsto y_j$.  In $\sigma_j(P)$, select from each substituted factor the
pure $y_j$-term.  Injectivity of the displayed weight on the support implies that
one original monomial contributes a unique largest pure power of $y_j$, with
nonzero coefficient.  Fix once and for all a degree-compatible monomial order on
$K[y_1,\ldots,y_q]$.  Because every $M_i>1$, within the expansion of any fixed
monomial of $P$ the term obtained by choosing the pure $y_j^{M_i}$ contribution from
each substituted factor has strictly largest total degree.  Choosing the $M_i$
successively and sufficiently large makes these maximal total degrees distinct for the
finitely many monomials in the support of $P$.  Thus $\sigma_j(P)$ has a unique
highest-total-degree term, and that term is a positive power of $y_j$; in particular it
is the leading monomial for the fixed degree-compatible order.\\

Put $J=D\cap K[y_1,\ldots,y_q]$. The initial ideal of $J$
contains a power of every variable, so $K[y]/J$ is
zero-dimensional. Since $J\subseteq\langle y\rangle$,
its localization at $\langle y\rangle$ is a proper
zero-dimensional local ring. Therefore, for some $b$, we have that $\langle y\rangle^bK[y]_{\langle y\rangle}
  \subseteq JK[y]_{\langle y\rangle}$. Extending to $K[[y]]$ gives
$\mathfrak m_T^b\subseteq D$, and therefore
$\sqrt D=\mathfrak m_T$.
\end{proof}

\begin{lemma}\label{lem:critical-height-new}
Assume $2\le q\le s$ and let $F\in\m K[x]^q$ satisfy
$\rank\bigl(\JacMat(F)(0)\bigr)<q$.  For every $N$ there is $H\in\m^N K[x]^q$ such
that, for $G=F+H$, the ideal $I_q(\JacMat(G))$ has height $s-q+1$ after localization
and completion at the origin.  Moreover the \emph{critical-value ideal}
\[
 \Delta_G=\ker\!\left(
 K[[y_1,\ldots,y_q]]\longrightarrow \frac{R}{I_q(\JacMat(G))},
 \quad y_i\longmapsto G_i
 \right)
\]
contains a nonzero polynomial.
\end{lemma}

\begin{proof}
First replace $N$ by a larger integer if necessary.  Choose pairwise distinct integers
$a_1,\ldots,a_q\ge N$, none divisible by $\Char K$, and larger than the
degrees occurring in $F$, and set
$H_i=\sum_{j=1}^s x_j^{a_i}$.
For a $q$-element subset $J=\{j_1,\ldots,j_q\}$ of the source variables, the part of
highest total degree in the corresponding $q\times q$ Jacobian minor is
$\det\bigl[a_i x_{j_\ell}^{a_i-1}\bigr]_{1\le i,\ell\le q}$.
This generalized alternant is nonzero: after expansion, distinct permutations give
distinct exponent vectors because the $a_i$ are distinct, and every coefficient
$\prod_i a_i$ is nonzero in $K$. Choose a degree-compatible monomial order. For every
$q$-element subset $J$, the leading monomial of the
corresponding Jacobian minor is therefore a monomial of
the highest-total-degree alternant. Its support is exactly
$J$, and each variable indexed by $J$ has positive exponent.  Its leading monomial has a strictly positive exponent on every
variable indexed by $J$ and no other variables.  Therefore a positive power of
$\prod_{j\in J}x_j$ belongs to the initial ideal.  Hence the radical of the initial ideal
of $I_q(\JacMat(G))$ contains the squarefree monomial ideal generated by
$\prod_{j\in J}x_j$ over all $q$-element subsets $J$.  That squarefree ideal has height
$s-q+1$.  A Gr\"obner degeneration preserves Krull dimension \cite[Chapter~15]{Eisenbud1995}, so
$\height I_q(\JacMat(G))=\height\init(I_q(\JacMat(G)))\ge s-q+1$; the standard determinantal height bound \cite[Chapter~2]{BrunsVetter} gives the reverse inequality.  Thus equality holds in
$K[x_1,\ldots,x_s]$.  Since $\rank\bigl(\JacMat(G)(0)\bigr)<q$, every maximal minor vanishes at the
origin, so $I_q(\JacMat(G))\subseteq \langle x_1,\ldots,x_s\rangle$.  Every prime contained in this maximal
ideal and minimal over $I_q(\JacMat(G))$ has height at least $s-q+1$, while the determinantal
height bound remains valid after localization.  Hence the localized ideal has height
exactly $s-q+1$.  Completion preserves the dimension of the quotient
\cite[Lemma~15.44.1, Tag~07NV]{stacks-project}, and therefore preserves the height in
this regular local ring.  This proves the stated local and completed assertion.\\

Every minimal component of the critical locus consequently has dimension at most
$q-1$.  The restrictions of the $q$ functions $G_1,\ldots,G_q$ to such a component
are algebraically dependent, so for each minimal component there is a nonzero
polynomial $P(y_1,\ldots,y_q)$ vanishing on its image.  Multiplying these relations over
the finitely many minimal components gives a nonzero polynomial in the radical of the
kernel of
$K[y_1,\ldots,y_q]\longrightarrow \frac{K[x_1,\ldots,x_s]}{I_q(\JacMat(G))}, \qquad y_i\longmapsto G_i$.
Taking a suitable power produces a nonzero polynomial in the kernel itself.  This is the
asserted critical-value relation.
\end{proof}

\begin{theorem}\label{thm:pure-right}
For every field $K$ and every $F\in\m R^q$,
\[
 F\text{ is finitely }\R\text{-determined}
 \iff \dim_K \left(\frac{R^q}{\im\bigl(\JacMat(F)\bigr)}\right)<\infty.
\]
If $q>1$, these conditions are equivalent to $\rank\bigl(\JacMat(F)(0)\bigr)=q$.  In that case
$q\le s$, the germ $F$ is right equivalent to the coordinate projection
$(x_1,\ldots,x_q)$, and $F$ is already $1$-$\R$-determined.
\end{theorem}

\begin{proof}
For $q=1$, the assertion is the arbitrary-characteristic finite-Milnor-number
criterion \cite[Theorem~2.1]{BoubakriGreuelMarkwig2012}.  Assume $q>1$.  Proposition~\ref{prop:dimension-bounds} gives $q\le s$.
After replacing $F$ by a polynomial truncation in its right orbit, define its
\emph{relation ideal} and \emph{critical-value ideal} by
$E_F=\ker(K[[y_1,\ldots,y_q]]\to R,\ y_i\mapsto F_i)$ and
$\Delta_F=\ker(K[[y_1,\ldots,y_q]]\to \frac{R}{I_q(\JacMat(F))},\ y_i\mapsto F_i)$,
respectively.  Let $k$ be a right-determinacy order for $F$.  If a target automorphism $\psi$ satisfies
$\psi(y)-y\in\mathfrak m_T^{k+1}T^q$, then $F_i\in\m$ gives
$\psi\circ F-F\in\m^{k+1}R^q$.  Hence $\psi\circ F$ is right equivalent to $F$, so
there is a source automorphism $\varphi$ with $\psi\circ F=F\circ\varphi$.  For such
$\psi$, the chain rule gives
$\JacMat(\psi\circ F)=D\psi(F)\JacMat(F)$ with $D\psi(F)$ invertible.  Thus the maximal-minor
ideal is unchanged up to the induced source automorphism, and both $E_F$ and
$\Delta_F$ are invariant under these deep target automorphisms.\\

We first record the following consequence of right determinacy, which we use in the
proof: the
components of $F$ minimally generate the ideal $\langle F_1,\ldots,F_q\rangle$.  Indeed, suppose
after reordering that $F_q\in I_0=\langle F_1,\ldots,F_{q-1}\rangle$.  Since $q\le s$, the ideal
$I_0$ has at most $s-1$ generators and therefore is not $\m$-primary.  Thus, for every
$N$, one can choose $h\in\m^N\setminus I_0$.  Take $N$ larger than a right-determinacy
order and set $G=(F_1,\ldots,F_{q-1},F_q+h)$.  Then $G=\varphi(F)$ for some source
automorphism $\varphi$.  For $i<q$ this identity gives $\varphi(F_i)=F_i$, whereas
applying $\varphi$ to a relation $F_q=\sum_{i<q}a_iF_i$ gives
$F_q+h=\sum_{i<q}\varphi(a_i)F_i$, so $h\in I_0$, a contradiction.\\

Suppose $\rank\bigl(\JacMat(F)(0)\bigr)<q$.  Apply Lemma~\ref{lem:critical-height-new} beyond the
determinacy order.  The resulting right-equivalent $G$ has a critical-value ideal
containing a nonzero polynomial, so Lemma~\ref{lem:deep-shear-new} gives
$\sqrt{\Delta_G}=\mathfrak m_T$.  Thus every component $G_i$ belongs to the radical of
$I_q(\JacMat(G))$.\\

Right determinacy implies finite contact determinacy of the generator column.  If
$q<s$, Corollary~\ref{cor:icis-all-fields} shows that
$\langle G_1,\ldots,G_q\rangle$ defines an ICIS.  Therefore
$\langle G_1,\ldots,G_q\rangle+I_q(\JacMat(G))$ is $\m$-primary.  Since every $G_i$ lies in
$\sqrt{I_q(\JacMat(G))}$, these two ideals have the same radical, and $I_q(\JacMat(G))$ would itself
be $\m$-primary.  If $q=s$ and $\langle G_1,\ldots,G_q\rangle$ were not $\m$-primary, then it would
have positive-dimensional quotient; Corollary~\ref{cor:icis-all-fields} would make its
$q=s$ minimal generators a regular sequence, forcing a zero-dimensional quotient, a
contradiction.  Thus $\langle G_1,\ldots,G_q\rangle$ is $\m$-primary.  The same radical
containment therefore forces
$I_q(\JacMat(G))$ to be $\m$-primary.  Both alternatives are impossible because a proper
maximal-minor ideal of a $q\times s$ matrix has height at most $s-q+1<s$.  Hence
$\rank\bigl(\JacMat(F)(0)\bigr)=q$.\\

A full-rank Jacobian at the origin makes $F$ a formal submersion, so a source coordinate
change sends it to $(x_1,\ldots,x_q)$.  Any germ with the same $1$-jet as $F$ again has
Jacobian rank $q$ at the origin and is therefore right equivalent to the same coordinate
projection.  Hence $F$ is $1$-$\R$-determined.  Conversely, if
$\frac{R^q}{\im\bigl(\JacMat(F)\bigr)}$ has finite length, then the Jacobian map is surjective after
localization at every nonmaximal prime.  Thus no nonmaximal prime contains the ideal
$I_q(\JacMat(F))$ of maximal minors, and hence $I_q(\JacMat(F))$ is either $\m$-primary or the unit
ideal.  Tensoring with $\Frac(R)$ also shows that $\JacMat(F)$ has row rank $q$, so necessarily
$q\le s$.  The determinantal height bound \cite[Chapter~2]{BrunsVetter}
$\height I_q(\JacMat(F))\le s-q+1<s$ excludes the proper $\m$-primary case when $q>1$,
so $I_q(\JacMat(F))=R$.
\end{proof}

\begin{remark}
The rigidity in Theorem~\ref{thm:pure-right} is specific to pure right equivalence.  When
$q>1$, a finitely right-determined map germ cannot have a singular differential at the origin:
it is already a formal submersion and hence a coordinate projection.  For example,
\[
 (x_1^2,x_2,\ldots,x_q):(K^s,0)\longrightarrow(K^q,0)
\]
is never finitely pure-right determined for $q>1$, in any characteristic.  This should be
contrasted with left--right and contact equivalence, where non-submersive finitely determined
germs are the principal objects of singularity theory.
\end{remark}

\begin{corollary}\label{cor:smooth-target-right}
Let $(Y,y)$ be smooth with $y\in Y(K)$ and $\dim_yY=q>1$, and let
$F:(\Spf R,\m)\to(\widehat Y_y,y)$ be a formal map germ.  Then $F$ is
finitely right determined if and only if $dF_0:T_0\A_K^s\to T_yY$ is surjective.  In
this case $F$ is $1$-right-determined.  These statements do not depend on the choice of
formal parameters at $y$.
\end{corollary}

\begin{proof}
Since $Y$ is smooth at the $K$-rational point $y$, its completed local ring is
(noncanonically) isomorphic to $K[[y_1,\ldots,y_q]]$. Under such an
isomorphism the assertion is exactly Theorem~\ref{thm:pure-right}. Replacing
one regular system of parameters by another composes the coordinate
representative with a target formal automorphism; this preserves
surjectivity of $dF_0$ and the property that all germs with the same
$1$-jet lie in one source-coordinate orbit.
\end{proof}

\begin{corollary}\label{cor:pure-right-matrix}
Let $A\in\m M_{m,n}$ over an arbitrary field.  Then $A$ is finitely
$\R$-determined if and only if its ordinary, equivalently its extended,
$\R$-tangent image has finite $K$-codimension.  If $mn>1$, these
conditions are equivalent to surjectivity at the origin of the
differential of the entry map $A:(K^s,0)\longrightarrow(K^{mn},0)$, 
and in this case $A$ is already $1$-$\R$-determined. Consequently, Theorems~\ref{thm:glr} and~\ref{thm:pure-right} settle the
full left--right and pure right cases of GLS Conjecture~3.1.16 over every
field.  Together with Theorem~\ref{thm:one-sided-complete}, they settle
all four actions over every infinite field.  Over an arbitrary field,
the one-sided cases are additionally settled in the dimension ranges of
Theorem~\ref{thm:one-sided-ranges}.
\end{corollary}

\begin{proof}
Vectorize the entries of $A$ and apply
Theorem~\ref{thm:pure-right} with $q=mn$.  The extended $\R$-tangent
image is the $R$-span of the derivative matrices, whereas the ordinary
$\R$-tangent image is their $\m$-span.  The quotient of the former
derivative contribution by the latter is generated over $K$ by the
finitely many derivative matrices.  Therefore finite codimension of the
ordinary tangent image is equivalent to finite codimension of the
extended tangent image.

If $mn>1$, Theorem~\ref{thm:pure-right} shows that finite right
determinacy is equivalent to having $\rank\bigl(\JacMat(A)(0)\bigr)=mn$.  This is exactly surjectivity at the origin of the differential of the
entry map $A:(K^s,0)\longrightarrow(K^{mn},0)$. 
The same theorem then shows that $A$ is $1$-$\R$-determined.  For
$mn=1$, the assertion is the scalar finite-right-determinacy criterion.
\end{proof}

Finally, we translate punctured rank transversality into statements about determinantal
ideals.  The homological study of determinantal ideals begins with the Eagon--Northcott
complex \cite{EagonNorthcott1962}; the expected-height perfection and
Cohen--Macaulayness results were subsequently developed in the generic-perfection work
of Hochster--Eagon \cite{HochsterEagon1971}.  We use the modern treatment in
\cite{BrunsVetter}.  For $1\le t\le\min(m,n)$ put
$c_t=(m-t+1)(n-t+1)$.

\begin{corollary}\label{cor:rank-trans}
A matrix germ is finitely $\Glr$-determined exactly when, on the punctured formal
spectrum, it is transverse to every constant-rank stratum.
\end{corollary}

\begin{proof}
By Lemma~\ref{lem:rank}, punctured transversality is equivalent to localized
surjectivity of the extended left--right tangent image.  For a finitely generated
cokernel over $R$, this is equivalent to finite length.  Apply Theorem~\ref{thm:glr}.
\end{proof}

\begin{theorem}
\label{thm:det-geometry-new}
Let $G\in\{\R,\Gl,\Gr,\Glr\}$ and let $A$ be finitely $G$-determined.  Then
\[
 \height I_t(A)=\min\{s,c_t\}\qquad(1\le t\le\min(m,n)).
\]
Moreover $\frac{R}{I_t(A)}$ is Cohen--Macaulay; if $c_t<s$ it is reduced; and if
$c_t\le s-2$ it is a normal domain.
\end{theorem}

\begin{proof}
Finite $G$-determinacy implies finite $\Glr$-determinacy.  At every nonmaximal point
lying over $I_t(A)$, Corollary~\ref{cor:rank-trans} gives the expected inverse-image
codimension.  Combined with the standard determinantal height bound \cite[Chapter~2]{BrunsVetter}, this proves the
height formula.  If $c_t\le s$, the equality $\height I_t(A)=c_t$ is the
expected-height case, so determinantal perfection gives Cohen--Macaulayness
\cite[Chapter~2]{BrunsVetter}.  If $c_t>s$, then $\height I_t(A)=s$ and
$\frac{R}{I_t(A)}$ is zero-dimensional, hence Cohen--Macaulay.\\

If $c_t<s$, every minimal prime lies in the punctured spectrum and has generic matrix
rank $t-1$.  Transversality makes the quotient smooth at every generic point.  It is therefore generically reduced.  Since a Cohen--Macaulay ring has no embedded associated primes \cite[Lemma~10.103.7, Tag~0BUS]{stacks-project}, it is reduced.\\

Assume $c_t\le s-2$.  Then
\[
 \dim\left(\frac{R}{I_t(A)}\right)=s-c_t\ge2,
\]
so the closed point has codimension at least two in the determinantal quotient and is irrelevant
for Serre condition $R_1$.  On the punctured spectrum, transversality identifies the singular
locus with the inverse image of the matrices of rank at most $t-2$.  If $t=1$ this locus is empty.
If $t\ge2$, its codimension in $\frac{R}{I_t(A)}$ is at least $ \min\{s,c_{t-1}\}-c_t\ge2$. Thus every codimension-one point is regular, so $R_1$ holds.  Cohen--Macaulayness gives $S_2$,
and the Serre criterion \cite[Lemma~10.157.4, Tag~031S]{stacks-project} gives normality.  Since
the quotient is local, a normal ring is a domain.
\end{proof}

\begin{corollary}
\label{cor:det-approx-new}
Every prescribed finite matrix jet over any field admits a higher-order completion for
which all determinantal ideals simultaneously have the heights and regularity properties
of Theorem~\ref{thm:det-geometry-new}.
\end{corollary}

\begin{proof}
Apply Theorem~\ref{thm:jettrans}.  The resulting matrix
has finite left--right tangent-image codimension and is therefore finitely
$\Glr$-determined.  Apply Theorem~\ref{thm:det-geometry-new}.
\end{proof}

\section{Smooth stabilizers and Frobenius modality}\label{sec:frobenius}

In this section, we study the right action on isolated hypersurface singularities.  We
first prove that sufficiently high jet stabilizers are smooth.  We then define the
Frobenius algebra and show how its length and Loewy structure control modality,
determinacy, separability, and the Milnor number.\\

Throughout this section, $k$ is algebraically closed of characteristic $p>0$.  Put
$\underline{x}=(x_1,\ldots,x_n)$, $R=k[[\underline{x}]]$, and
$\m=\langle x_1,\ldots,x_n\rangle$, and let $f\in\m^2$.  We write
$\partial_i=\frac{\partial}{\partial x_i}$ and
$\Jac(f)=\langle \partial_1f,\ldots,\partial_nf\rangle$ for the Jacobian ideal.  Throughout,
$\Jac(-)$ denotes the Jacobian ideal of a scalar germ, whereas $\JacMat(-)$ denotes the
Jacobian matrix of a vector-valued germ, as defined in Section~\ref{sec:matrix}.  We set
$\mu(f)=\dim_k\left(\frac{R}{\Jac(f)}\right)$ and $\Tj(f)=\dim_k\left(\frac{R}{\langle f\rangle+\Jac(f)}\right)$.  Thus
$\mu(f)$ and $\Tj(f)$ are the Milnor and Tjurina\footnote{The Tjurina number of $f$ is usually denoted by $\tau(f)$ (see, e.g~\cite{GLS}). Since we reserve the notation $\tau(-)$ for the test ideal, we use here the notation $\Tj(f)$.}  numbers, respectively. We call $f$ isolated if
$\mu(f)<\infty$, equivalently if $\Jac(f)$ is $\m$-primary.  We reserve $\tau(-)$ exclusively for generalized test ideals in Section~\ref{sec:F-applications}.  The right group is
$\R=\Aut_k(R)$.\\

For $N\ge1$ let $\mathcal J_N=\frac{\m}{\m^{N+1}}$ and let $\R_N$ be the algebraic group of
$N$-jets of local coordinate changes.  Its tangent space at the identity is
$\left(\frac{\m}{\m^{N+1}}\right)^n$.  For the orbit map
$o_{f,N}:\R_N\to \mathcal J_N$, $\varphi\mapsto\varphi(j^Nf)$, the differential at the identity
and the corresponding tangent image are
\begin{align}
 d_eo_{f,N}(v_1,\ldots,v_n)
 &=\sum_i v_i\partial_i f\pmod{\m^{N+1}},\label{eq:orbit-diff}\\
 \Ttilde_{j^Nf}(\R_N\cdot j^Nf)
 &=\frac{\m\Jac(f)+\m^{N+1}}{\m^{N+1}}.\label{eq:tangent-image}
\end{align}

\begin{definition}\label{def:rosenlicht-modality}
Let an algebraic group $G$ act on a variety $X$.  A \emph{Rosenlicht
stratification} is a finite disjoint decomposition $X=\coprod_aX_a$ into
$G$-invariant locally closed subvarieties for which geometric quotient morphisms
$\pi_a:X_a\to Q_a$ exist.  Here a \emph{geometric quotient} means a surjective
$G$-invariant morphism whose geometric fibers are precisely the $G$-orbits and whose
regular functions are the invariant regular functions; the symbol $Q_a$ denotes the
corresponding quotient variety.  Such stratifications exist by
Rosenlicht \cite{Rosenlicht1956,Rosenlicht1963}.  For an open neighborhood $U$ of $x\in X$, set
$G\text{-mod}(U)=\max_a\dim\pi_a(U\cap X_a)$ and define $G\text{-mod}(x)$ as the
minimum of this number over all open neighborhoods $U$ of $x$.  The result is
independent of the chosen Rosenlicht stratification
\cite[Definition~2.1 and Corollary~A.3]{GreuelNguyen2016}.  Thus
$G\text{-mod}(x)$ measures the local dimension of the orbit space near $x$---the
number of essential parameters left after quotienting by $G$---rather than the
dimension of the orbit itself.\\

For an isolated hypersurface germ $f$, an integer $N$ is \emph{sufficiently large for
right equivalence} if some neighborhood of $j^Nf$ in $\mathcal J_N$ consists entirely of right
$N$-determined germs.  The \emph{right modality} is
$\rmod(f)=\R_N\text{-mod}(j^Nf)$ for any such $N$; it is independent of the choice
of sufficiently large $N$ \cite[Definitions~2.1, 2.3 and Proposition~2.4]{GreuelNguyen2016}.
\end{definition}

\begin{definition}\label{def:proper-modality}
Let $(T,t_0)$ be a pointed algebraic variety.  An \emph{unfolding with trivial section}
of $f$ over $(T,t_0)$ is a family $F(x,t)\in\mathcal O(T)[[x]]$ such that
$F(x,t_0)=f$ and $F(0,t)=0$ for every $t$.  It is \emph{right-complete} if every
unfolding of $f$ becomes, after passage to an étale neighborhood of its base point,
fiberwise right equivalent to a pullback of $F$; see
\cite[Definition~2.9]{GreuelNguyen2016}.  A semiuniversal unfolding with trivial
section has the form
$F(x,t)=f(x)+\sum_{j=1}^{b_f}t_jg_j(x)$, where the classes of the $g_j$ form a
$k$-basis of $\frac{\m}{\m\Jac(f)}$ and $b_f=\mu(f)+n-1$; it is right-complete
\cite[Proposition~2.14]{GreuelNguyen2016}.\\

For any unfolding $F$, the notation $\R\text{-mod}_F(f)$ denotes the local
number of parameters that remain after quotienting the image of the parameter-to-jet map
$t\mapsto j^NF_t$ by the $\R_N$-action, computed on a Rosenlicht stratification; more
precisely it is the relative Rosenlicht modality at $t_0$ for a sufficiently large
$N$.  It is independent of that choice \cite[Definition~2.5]{GreuelNguyen2016}.  For a semiuniversal unfolding
with trivial section, let
$\mathcal C_{\mu(f)}(F)=\{t\mid\mu(F_t)=\mu(f)\}$ be its Milnor-constant stratum.
Following Nguyen, the \emph{proper modality} is
$\propermod(f)=\dim_{t_0}\mathcal C_{\mu(f)}(F)$ \cite[\S2.2]{Nguyen2019}.
\end{definition}

\begin{remark}
The notion of modality was introduced into singularity theory in the classification of
simple, unimodal, and bimodal singularities by Arnold and collaborators; see \cite{AGV}.
The algebraic formulation used here replaces a local analytic orbit space by Rosenlicht
quotients and was developed in arbitrary characteristic by Greuel--Nguyen
\cite{GreuelNguyen2016}.  In the complex analytic setting, Gabrielov proved that modality
agrees with proper modality \cite{Gabrielov1974}.  Nguyen established the corresponding
positive-characteristic equality first for univariate germs and later for all right
unimodal and bimodal germs \cite{Nguyen2014,Nguyen2019}.
\end{remark}

For a dominant morphism of integral varieties, we use \emph{separable} in the usual
function-field sense.  Thus an orbit map is separable if the induced extension of
function fields is separable; for the finite jet actions below this is equivalent to
surjectivity of the differential onto the tangent space of the orbit
\cite[Corollary~4.2]{GreuelPham2020}.\\

The external inputs used repeatedly in the hypersurface part are collected in one lemma.
This also fixes the jet-order conventions used below.

\begin{lemma}[Standard inputs for right equivalence]\label{lem:standard-inputs}
Let $f\in\m^2\subseteq R=k[[x_1,\ldots,x_n]]$ be isolated.
\begin{enumerate}[label=\textup{(\roman*)},leftmargin=*]
\item\label{std:suff} Every $N\ge 2\mu(f)$ is sufficiently large for right equivalence in the
sense of Definition~\ref{def:rosenlicht-modality}.
\item\label{std:complete} For every unfolding $F$ of $f$ one has
$\R\text{-mod}_F(f)\le\rmod(f)$; equality holds for a right-complete unfolding.  In
particular, a semiuniversal unfolding with trivial section is right-complete.
\item\label{std:sep} For an algebraic finite-jet action, the orbit map is separable if and only
if its differential at the identity is surjective onto the tangent space of the orbit; equivalently,
the tangent image equals the orbit tangent space.
\item\label{std:group} Over a perfect field, the reduction of a finite-type group scheme is a
closed subgroup scheme, and every reduced finite-type group scheme is smooth.
\item\label{std:det}  If
$\m^{N+2}\subseteq\m^2\Jac(f)$, then $f$ is right
$(2N-\ord(f)+2)$-determined.
\item\label{std:MY} Put
\[
 M_N(f)=\frac{R}{\m^N\Jac(f)},
\]
and regard it as a $k[[t]]$-algebra via $t\mapsto f$.  If $s=\ord(f)$ and
$d(f)=\ord(\Jac(f))$, then for every $N$ satisfying
\[
 \m^{\left\lfloor\frac{N+s+d(f)+1}{2}\right\rfloor}\subseteq\m^2\Jac(f)
\]
one has
\[
 f\overset{\R}{\sim}g
 \quad\Longleftrightarrow\quad
 M_N(f)\cong M_N(g)\quad\text{as $k[[t]]$-algebras}.
\]
\item\label{std:split}
There is a right equivalence $f\overset{\mathcal R}{\sim}
g(x_1,\ldots,x_c)+Q(x_{c+1},\ldots,x_{n})$, 
where $Q$ is a nonsingular quadratic form in the split variables, with the zero-variable case allowed, and $\Hess(g)(0)=0$. If $p\ne2$, then $g\in\langle x_1,\ldots,x_c\rangle^3$; if $p=2$, the quadratic part of $g$ is the square of a linear form. The residual germ is unique up to right equivalence. In particular, $c$ is invariant under right equivalence, and right modality is invariant under nonsingular quadratic suspension.
\end{enumerate}
\end{lemma}

\begin{proof}
For Assertion~\ref{std:suff}, choose a jet order large enough that the Jacobian colength is
computed in the corresponding truncated power-series ring.  On that finite-dimensional jet
space, $\mu$ is the dimension of the cokernel of a matrix whose entries depend algebraically
on the jet coefficients, and is therefore upper semicontinuous.  Hence, after shrinking around
$j^Nf$, every nearby germ $g$ has $\mu(g)\le\mu(f)$.  Boubakri--Greuel--Markwig
\cite[Theorem~2.1(a) and Corollary~2.4(a)]{BoubakriGreuelMarkwig2012} show that an
isolated $g$ is right determined of order at most
$2\mu(g)-\ord(g)+2\le2\mu(g)$.  Thus every $N\ge2\mu(f)$ is sufficiently large in
the sense of Definition~\ref{def:rosenlicht-modality}.  The independence of right modality
from the choice of a sufficiently large jet is \cite[Proposition~2.4]{GreuelNguyen2016}.
Assertion~\ref{std:complete} is \cite[Proposition~2.12(ii) and
Proposition~2.14]{GreuelNguyen2016}.  Assertion~\ref{std:sep} is the standard separability
criterion for algebraic group orbits; see \cite[Corollary~4.2]{GreuelPham2020} and
\cite[Proposition~3.1.17]{GLS}.  Assertion~\ref{std:group} is
\cite[Lemmas~39.7.6 and~39.8.4, Tags~047R and~047P]{stacks-project}.
Assertion~\ref{std:det} is \cite[Theorem~2.1(a)]{BoubakriGreuelMarkwig2012}, and
Assertion~\ref{std:MY} is \cite[Theorem~2.4]{GreuelPhamMY2017}.  Finally,
Assertion~\ref{std:split} is the existence-and-uniqueness form of the splitting
lemma of Greuel--Pfister \cite{GreuelPfister2026}; the invariance of right modality
under nonsingular quadratic suspension follows from
\cite[\S3]{GreuelNguyen2016}.
\end{proof}

\begin{remark}
Over the complex numbers, the Mather--Yau theorem determines an isolated hypersurface
singularity from its moduli algebra \cite{MatherYau1982}.  The literal analogue fails in
positive characteristic; the higher Milnor-algebra form used in
Lemma~\ref{lem:standard-inputs}\ref{std:MY} is the replacement proved by
Greuel--Pham \cite{GreuelPhamMY2017}.  The splitting lemma is classical in
characteristic different from two; the form used here, including characteristic two,
is due to Greuel--Pfister \cite{GreuelPfister2026}.
\end{remark}

We first study the stabilizer of a sufficiently high right jet.  The argument has three
steps.  First, the Jacobian partials form a regular sequence.  Second, exact
infinitesimal stabilizers can be integrated without dividing by factorials.  Third,
sufficiently deep approximate stabilizers already fix the chosen jet.  These three
facts give smoothness of the high-jet stabilizer.

\begin{lemma}\label{lem:regular}
If $f$ is isolated, then $\partial_1f,\ldots,\partial_nf$ is a regular sequence.  Every
syzygy $z=(z_1,\ldots,z_n)\in R^n$ satisfying $\sum_i z_i\partial_i f=0$ is generated
by Koszul syzygies; in particular every $z_i$ belongs to $\Jac(f)$.
\end{lemma}

\begin{proof}
The $\m$-primary ideal $\Jac(f)$ has height $n$ and is generated by $n$ elements.
The regular local ring $R$ is Cohen--Macaulay of dimension $n$
\cite[Lemma~10.106.3, Tag~00NQ]{stacks-project}.  Hence these generators form a system
of parameters and therefore a regular sequence
\cite[Lemma~10.104.2, Tag~02JN]{stacks-project}.  Exactness of the Koszul complex
\cite[Lemma~15.31.2, Tag~062F]{stacks-project} gives the syzygy assertion, and every
coordinate of a Koszul syzygy lies in $\Jac(f)$.
\end{proof}

\begin{lemma}
\label{lem:taylor}
For $f\in R$ and $Y=(Y_1,\ldots,Y_n)$ one has
\begin{equation}\label{eq:taylor}
 f(x+Y)-f(x)-\sum_{i=1}^n(\partial_i f)(x)Y_i
 \in\langle Y_1,\ldots,Y_n\rangle^2R[[Y_1,\ldots,Y_n]].
\end{equation}
This holds over an arbitrary field.
\end{lemma}

\begin{proof}
It suffices to consider a monomial $x^\alpha$.  The multivariable binomial
formula gives
\[
 (x+Y)^\alpha
 =x^\alpha+\sum_i\alpha_i x^{\alpha-e_i}Y_i
 +\sum_{\substack{\beta\le\alpha\\|\beta|\ge2}}
 \binom{\alpha}{\beta}x^{\alpha-\beta}Y^\beta.
\]
The middle sum is
$\sum_i\left(\frac{\partial x^\alpha}{\partial x_i}\right)Y_i$, while every term in the last
sum lies in $\langle Y_1,\ldots,Y_n\rangle^2$.  Applying this termwise to
$f=\sum_\alpha c_\alpha x^\alpha$ is legitimate in the $\langle x_1,\ldots,x_n,Y_1,\ldots,Y_n\rangle$-adic
topology.  No division by factorials occurs.
\end{proof}

\begin{lemma}
\label{lem:integrate}
Assume that $f\in\m^2$ is isolated and suppose $z\in\m R^n$ satisfies $\sum_i z_i\partial_i f=0$.  Then there is an automorphism $\Phi_t\in\Aut_{k[[t]]}k[[t,x_1,\ldots,x_n]]$ with $\Phi_0=\mathrm{id}$, tangent vector $z$ at $t=0$, and
$\Phi_t(f)=f$.
\end{lemma}

\begin{proof}
By Lemma~\ref{lem:regular}, $z_i\in \Jac(f)\subseteq\m$.  Seek
$\Phi_t(x_i)=x_i+U_i(t,x),\qquad U_i=tz_i+\sum_{r\ge2}t^ru_i^{(r)},\qquad u_i^{(r)}\in \Jac(f)$.
Assume the coefficients below order $r$ have been chosen so that
$f(x+U)-f(x)\equiv0\pmod{t^r}$, and let $E_r$ be the coefficient of $t^r$
before $u^{(r)}$ is chosen.  In \eqref{eq:taylor}, every nonlinear term has
at least two $U$-factors, hence lies in $\Jac(f)^2$.  The already chosen linear
part has no new $t^r$ term.  Thus
$E_r\in \Jac(f)^2=\sum_i(\partial_i f)\Jac(f)$, so choose $u_i^{(r)}\in \Jac(f)$ with
$\sum_i(\partial_i f)u_i^{(r)}=-E_r$.  The $t$-adic recursion produces
$U\in t\,\Jac(f)[[t]]^n$ with $f(x+U)=f(x)$.  Its linear coefficient matrix in the
$x$-variables is congruent to the identity modulo $t$, so the formal inverse-function theorem gives the required automorphism.
\end{proof}

To control approximate stabilizers, we measure how deeply the tangent image contains
powers of the maximal ideal.  Since $f$ is isolated, $\Jac(f)$ is $\m$-primary, so
such a containment exists.  The least exponent will enter both the determinacy and the
separability estimates below.
\begin{definition}\label{def:eta}
For isolated $f$, set
$\eta(f)=\min\{e\ge1\mid\m^e\subseteq\m\Jac(f)\}$.
If $d=\ord(f)$, then $\eta(f)\ge d$, because $\m\Jac(f)\subseteq\m^d$.
\end{definition}

\begin{lemma}\label{lem:order-taylor}
Let $f\in\m^d$ and $w=(w_1,\ldots,w_n)\in\m^qR^n$, where $q\ge1$.  Then
\[
 f(x+tw)-f(x)-t\sum_iw_i\partial_i f
 \in\m^{d+2(q-1)}k[[t,x]].
\]
Consequently, if
$\sum_iw_i\partial_i f\in\m^{N+1}$ and $d+2(q-1)\ge N+1$, then
$x\mapsto x+tw$ fixes $j^Nf$.
\end{lemma}

\begin{proof}
It suffices to consider a monomial $x^\alpha$ of total degree
$s=|\alpha|\ge d$.  Expanding gives
$(x+tw)^\alpha =\sum_{\beta\le\alpha} \binom{\alpha}{\beta}x^{\alpha-\beta}t^{|\beta|}w^\beta$.
The terms with $|\beta|=0$ and $|\beta|=1$ are respectively
$x^\alpha$ and $t\sum_iw_i\partial_i(x^\alpha)$.  If
$r=|\beta|\ge2$, then
$x^{\alpha-\beta}\in\m^{s-r}$ and $w^\beta\in\m^{rq}$, so the
corresponding term has $\m$-order at least
$(s-r)+rq=s+r(q-1)\ge d+2(q-1)$.
Summing over the monomials of $f$ proves the first assertion.\\

Under the additional hypotheses, both the linear term and the remainder lie
in $\m^{N+1}k[[t,x]]$, hence
$f(x+tw)\equiv f(x)\pmod{\m^{N+1}}$.  Finally $q\ge1$, so the linear part
in $x$ of $x\mapsto x+tw$ has matrix $I+tA$, whose determinant has constant
term $1$.  Thus the substitution is an automorphism over $k[[t]]$.
\end{proof}

\begin{theorem}\label{thm:smooth}
Let $f$ be isolated, $d=\ord(f)$, and $\eta=\eta(f)$.  For every $N\ge2\eta-d-1$, the scheme-theoretic stabilizer $H_N\subseteq\R_N$, that is, the closed subgroup scheme
representing $N$-jet coordinate changes that fix $j^Nf$, is smooth.
Consequently the orbit map is separable and
\[
 T_{j^Nf}(\R_N\cdot j^Nf)
 =\frac{\m\Jac(f)+\m^{N+1}}{\m^{N+1}}.
\]
\end{theorem}

\begin{proof}
Let $v\in\m R^n$ represent a tangent vector to $H_N$, so
$\delta(v)=\sum_i v_i\partial_i f\in\m^{N+1}$.  The numerical hypothesis and
$\eta\ge d$ imply $N\ge\eta-1$.  Hence
$\m^{N+1} =\m^{N+1-\eta}\m^\eta \subseteq\m^{N+1-\eta}\m\Jac(f) =\m^{N+2-\eta}\Jac(f)$.
Choose $w\in\m^qR^n$, where $q=N+2-\eta\ge1$, such that
$\delta(w)=\delta(v)$, and put $z=v-w$.  Then $\delta(z)=0$.
Lemma~\ref{lem:integrate} gives a stabilizer arc $\alpha_z(t)$ with tangent
vector $z$.  Lemma~\ref{lem:order-taylor} gives a stabilizer arc
$\alpha_w(t):x\mapsto x+tw$, because
$d+2(q-1)=d+2(N+1-\eta)\ge N+1$.
Their composition $\alpha_z(t)\circ\alpha_w(t)$ is a $k[[t]]$-valued point of the
scheme-theoretic stabilizer $H_N$ and has tangent vector $z+w=v$ modulo
$\m^{N+1}$.  Thus every tangent vector of $H_N$ is represented by a formal arc in $H_N$.\\

Since $\Spec k[[t]]$ is reduced, every morphism
$\Spec k[[t]]\to H_N$ factors uniquely through the reduction
$(H_N)_{\mathrm{red}}$.  Consequently every tangent vector of $H_N$ lies in
$T_e(H_N)_{\mathrm{red}}$; the reverse inclusion is tautological, and hence
\[
 T_eH_N=T_e(H_N)_{\mathrm{red}}.
\]
By Lemma~\ref{lem:standard-inputs}\ref{std:group}, $(H_N)_{\mathrm{red}}$ is smooth.  The
reduction has the same underlying topological space as $H_N$, so
\[
 \dim_kT_eH_N
 =\dim_kT_e(H_N)_{\mathrm{red}}
 =\dim_e(H_N)_{\mathrm{red}}
 =\dim_eH_N.
\]
Thus the local ring of $H_N$ at the identity has embedding dimension equal to its Krull
dimension and is regular.  Since $k$ is perfect, regularity of this finite-type group scheme at
the identity implies smoothness there; translations by group-valued points then give smoothness
everywhere.  Finally $\R_N$ is smooth, and the orbit dimension formula yields
\[
 \rank(d_eo_{f,N})
 =\dim\R_N-\dim H_N
 =\dim(\R_N\cdot j^Nf).
\]
The differential is therefore onto the tangent space of the orbit, and
Lemma~\ref{lem:standard-inputs}\ref{std:sep} gives separability.
\end{proof}

\begin{corollary}\label{cor:mu-sep}
If $\mu(f)=\mu$, then $\m^\mu\subseteq \Jac(f)$, hence
$\m^{\mu+1}\subseteq\m\Jac(f)$ and $\eta(f)\le\mu+1$.  Consequently every
$N\ge2\mu-1$ is a separable level; in particular every $N\ge2\mu$ is
simultaneously sufficiently large for right determinacy and separability.
\end{corollary}

\begin{proof}
The Artin local algebra $\frac{R}{\Jac(f)}$ has length $\mu$, so its maximal ideal has
$\mu$-th power zero.  The rest follows from Theorem~\ref{thm:smooth} and
$d\ge2$.
\end{proof}

\begin{remark}
Greuel--Pham observed that contact orbit maps may be inseparable in positive
characteristic, but reported that no nonseparable example was known for the right group
when the hypersurface singularity is isolated \cite[Section~5]{GreuelPham2020}.
Theorem~\ref{thm:smooth} shows that such an example cannot occur in sufficiently high
jet order: every isolated hypersurface singularity has separable right orbit maps from
the explicit level $N\ge2\eta(f)-\ord(f)-1$ onward.
\end{remark}

\begin{corollary}\label{cor:smooth-stabilizers-arbitrary-field}
Let $K$ be any field of positive characteristic, and let
$f\in\m^2\subseteq K[[x_1,\ldots,x_n]]$ have finite
Milnor number. Put $d=\ord(f)$ and
$\eta=\min\{a\ge1:\m^a\subseteq\m\Jac(f)\}$. Then the scheme-theoretic stabilizer of $j^Nf$ in the
right jet group is smooth over $K$ whenever $N\ge\max\{1,2\eta-d-1\}$.
\end{corollary}

\begin{proof}
Extend scalars to an algebraic closure $\overline K$.
The containment $\m^\eta\subseteq\m\Jac(f)$ persists
in $\overline K[[x]]$, so the extended germ remains
isolated and its corresponding exponent is at most
$\eta$. Its order remains $d$.
Theorem~\ref{thm:smooth} therefore makes the
base-changed stabilizer smooth over $\overline K$.
Stabilizers for the finite jet action commute with
field extension, and smoothness descends under the
faithfully flat extension $K\subseteq\overline K$.
\end{proof}

We now define the Frobenius invariant used throughout the rest of the paper.  The
structural role of Frobenius in regular local rings goes back to the theorem of Kunz
characterizing regularity by flatness of Frobenius \cite{Kunz1969}.  Here we use the explicit
power-series structure: the invariant records deformation directions coming from the
common kernel of the derivations in characteristic $p$.\\
Let $S=R^p=k[[x_1^p,\ldots,x_n^p]]$ be the image of the Frobenius
endomorphism $F_R:R\to R$, $h\mapsto h^p$, and let $\nF=\langle x_1^p,\ldots,x_n^p\rangle S$.
Since $k$ is perfect, $S$ is exactly the common kernel of the derivations
$\partial_1,\ldots,\partial_n$.

\begin{definition}\label{def:Fmod}
For an isolated germ $f$, put
\[
 \begin{aligned}
 Q_F(f)&=S\cap\m\Jac(f),\\
 A_F(f)&=\frac{S}{Q_F(f)},\\
 \Fmod(f)&=\dim_k\left(\frac{\nF}{Q_F(f)}\right)
 =\length_k\left(A_F(f)\right)-1.
 \end{aligned}
\]
The \emph{Frobenius type} is the $\Aut_k(S)$-orbit of $Q_F(f)$.
\end{definition}

\begin{proposition}\label{prop:Fbasic}
The algebra $A_F(f)$, its Frobenius type, and $\Fmod(f)$ are invariant under
right equivalence.  Moreover $\Fmod(f)<\infty$ if and only if $f$ is
isolated.
\end{proposition}

\begin{proof}
If $g=\varphi(f)$ with $\varphi\in\R$, the chain rule and invertibility of the
Jacobian matrix of $\varphi$ give $\Jac(g)=\varphi(\Jac(f))$.  Also $\varphi(S)=S$ and
$\varphi(\nF)=\nF$, since $S$ is the common kernel of the derivations.  Hence
$Q_F(g)=\varphi(Q_F(f))$.\\

If $f$ is isolated, choose $a$ with $\m^a\subseteq \Jac(f)$.  For $pr\ge a+1$,
$\nF^r\subseteq\m^{pr}\subseteq\m\Jac(f)$, so $A_F(f)$ has finite length.
Conversely, if $\nF^r\subseteq Q_F(f)$, then
$\langle x_1^p,\ldots,x_n^p\rangle^r\subseteq\m\Jac(f)$.  For a monomial
$x^\alpha$ write $\alpha_i=pq_i+r_i$ with $0\le r_i<p$.  If
$\sum_iq_i\le r-1$, then $|\alpha|\le p(r-1)+n(p-1)$.  Hence
$\m^{p(r-1)+n(p-1)+1}\subseteq\langle x_1^p,\ldots,x_n^p\rangle^r\subseteq\m\Jac(f)\subseteq \Jac(f)$,
so $\Jac(f)$ is $\m$-primary.  This argument does not presuppose isolatedness.
\end{proof}

For the modality comparison we use the finite-dimensional tangent algebra
$M_f=\frac{R}{\m\Jac(f)}$. The image of the nonconstant $p$-th
powers in $M_f$ records the Frobenius deformation
directions surviving modulo $\mathfrak m  \Jac( f )$. The next proposition compares those directions with proper
and right modality.

\begin{proposition}
\label{prop:Flower}
For every isolated $f$, one has $\Fmod(f)\le\propermod(f)$ and $\Fmod(f)\le\rmod(f)$.
\end{proposition}

\begin{proof}
Let $d=\Fmod(f)$ and choose polynomial representatives
$h_1,\ldots,h_d\in\nF$ whose classes form a basis of $\frac{\nF}{Q_F(f)}$.
Put $V=\Span_k\{h_1,\ldots,h_d\}$.  Then $V\cap\m\Jac(f)=0$.  For
$g_t=f+\sum_a t_ah_a$, all first derivatives of the $h_a$ vanish, so
$\Jac(g_t)=\Jac(f)$ and $\mu(g_t)=\mu(f)$.\\

Choose $N$ sufficiently large that $\m^{N+1}\subseteq\m\Jac(f)$, all $h_a$
survive in $\mathcal J_N$, right determinacy holds, and
Theorem~\ref{thm:smooth} applies uniformly to the family.  The affine
parameter space has tangent space $j^NV$, whereas the orbit tangent space at
every member is
$\frac{\m\Jac(f)+\m^{N+1}}{\m^{N+1}}$.  These spaces meet trivially.  Hence a
positive-dimensional right-equivalence fiber would have a smooth point with
a nonzero tangent vector in their intersection, a contradiction.  Thus the
general right-equivalence fiber is zero-dimensional, giving
$d\le\rmod(f)$.\\

The same directions occur in a semiuniversal unfolding with trivial section,
and the corresponding $d$-plane is $\mu$-constant.  Hence
$d\le\propermod(f)$.
\end{proof}

\begin{definition}
Set
\[
 \varepsilon_F(f)=
 \length_R\left(\frac{\m\Jac(f)}{Q_F(f)R}\right),
 \qquad
 \gamma_p(f)=
 \dim_k\left(\frac{R}{\m\Jac(f)+\langle x_1^p,\ldots,x_n^p\rangle}\right).
\]
By the Nakayama lemma, $\gamma_p(f)$ is the minimal number of generators of
$M_f=\frac{R}{\m\Jac(f)}$ as an $A_F(f)$-module; this is why we call it the generator number.
\end{definition}

\begin{theorem}
\label{thm:length}
If $d=\Fmod(f)$, then
\begin{equation}\label{eq:exact-length}
 p^n(d+1)=\mu(f)+n+\varepsilon_F(f).
\end{equation}
Moreover
\begin{equation}\label{eq:gamma-bound}
 \mu(f)+n\le\gamma_p(f)(d+1)
 \le\gamma_p(f)(\rmod(f)+1),
 \qquad\gamma_p(f)\le p^n.
\end{equation}
In particular
\begin{equation}\label{eq:main-bound}
 \rmod(f)\ge\Fmod(f)\ge
 \left\lceil\frac{\mu(f)+n}{p^n}\right\rceil-1.
\end{equation}
\end{theorem}

\begin{proof}
The ring $R$ is free of rank $p^n$ over $S$, with basis
$x_1^{a_1}\cdots x_n^{a_n}$, $0\le a_i<p$.  Hence
$\dim_k\left(\frac{R}{Q_F(f)R}\right)=p^n(d+1)$.
The exact sequence
$0\to\frac{\m\Jac(f)}{Q_F(f)R}\to\frac{R}{Q_F(f)R}\to\frac{R}{\m\Jac(f)}\to0$
and
$0\to\frac{\Jac(f)}{\m\Jac(f)}\to\frac{R}{\m\Jac(f)}\to\frac{R}{\Jac(f)}\to0$
give \eqref{eq:exact-length}, since
$\dim_k\left(\frac{\Jac(f)}{\m\Jac(f)}\right)=n$ by Lemma~\ref{lem:regular}.\\

Furthermore
$Q_F(f)=\Ann_S(M_f)$,
because an element of $S$ annihilates $M_f$ exactly when it kills $[1]$.
Thus $M_f$ is a finite $A_F(f)$-module, minimally generated by
$\gamma_p(f)$ elements.  Nakayama gives a surjection
$A_F(f)^{\gamma_p(f)}\twoheadrightarrow M_f$, whence
$\mu(f)+n\le\gamma_p(f)(d+1)$.
Finally $\frac{M_f}{\nF M_f}$ is a quotient of
$\frac{R}{\langle x_1^p,\ldots,x_n^p\rangle}$, so $\gamma_p(f)\le p^n$.
\end{proof}

\begin{corollary}\label{cor:equality}
Equality $\mu(f)+n=\gamma_p(f)(\Fmod(f)+1)$ holds exactly when $M_f$ is free over
$A_F(f)$ of rank $\gamma_p(f)$.  Equality
$\mu(f)+n=p^n(\Fmod(f)+1)$ holds exactly when $\m\Jac(f)=Q_F(f)R$.
\end{corollary}

\begin{proof}
The first assertion follows because the surjection
$A_F(f)^{\gamma_p(f)}\twoheadrightarrow M_f$ is an isomorphism exactly when the
finite lengths agree.  The second follows from the kernel
$\frac{\m\Jac(f)}{Q_F(f)R}$ of $\frac{R}{Q_F(f)R}\twoheadrightarrow\frac{R}{\m\Jac(f)}$.
\end{proof}

\begin{proposition}\label{prop:Fimage}
The Frobenius map on $M_f$ satisfies $\im F\cong A_F(f)$ and $\Fmod(f)+1=\dim_k\im F$.  If $b_e(f)=\dim_k\im F^e$, then $\rmod(f)\ge b_e(f)-1$, $b_1(f)\ge b_2(f)\ge\cdots\ge1$, and $b_e(f)=1$ for $e\gg0$.
\end{proposition}

\begin{proof}
The image of Frobenius on $R$ is $S$, and the kernel of $S\to M_f$ is
$Q_F(f)$.  Replacing $p$ by $p^e$ in the argument of
Proposition~\ref{prop:Flower} gives $\rmod(f)\ge b_e(f)-1$.  The images
are nested, and a sufficiently high Frobenius power kills the maximal ideal
of the Artin algebra $M_f$.
\end{proof}

Iterating Frobenius gives a nested sequence of finite-dimensional quotients and corresponding numerical bounds.  For $e\ge1$, set
$S_e=k[[x_1^{p^e},\ldots,x_n^{p^e}]]$, $\mathfrak n_e=\langle x_1^{p^e},\ldots,x_n^{p^e}\rangle S_e$,
$Q_e=S_e\cap\m\Jac(f)$, $A_e=\frac{S_e}{Q_e}$, and $b_e=\dim_kA_e$.  Also set
$\gamma_e(f)=\dim_k\left(\frac{M_f}{\mathfrak n_eM_f}\right)$ and
$\varepsilon_e(f)=\length_R\left(\frac{\m\Jac(f)}{Q_eR}\right)$.

\begin{theorem}\label{thm:Ftower}
For every $e\ge1$:
\begin{enumerate}[label=\textup{(\roman*)},leftmargin=*]
\item $A_e\cong\im F^e$ and $p^{en}b_e=\mu(f)+n+\varepsilon_e(f)$.
\item $\mu(f)+n\le\gamma_e(f)b_e$ and $\gamma_e(f)\le p^{en}$.
\item $\rmod(f)\ge b_e-1$ and $\propermod(f)\ge b_e-1$.
\item $b_1\ge b_2\ge\cdots\ge1$, and $b_e=1$ for $e\gg0$.
\end{enumerate}
\end{theorem}

\begin{proof}
The proof is identical to that of Theorem~\ref{thm:length}, replacing $S$
by $S_e$ and $p^n$ by $p^{en}$.  The nonconstant elements of $S_e$ have
zero first derivatives, so the same Frobenius-slice argument gives both
modality inequalities.
\end{proof}

\begin{definition}\label{def:Frobenius-Loewy-length}
The \emph{Loewy length} of a local Artin algebra $(A,\mathfrak q)$ is the least integer
$\ell\ge1$ such that $\mathfrak q^\ell=0$.  We denote the Loewy length of $A_F(f)$ by
$\ell_F(f)=\min\{\ell\ge1\mid\nF^\ell\subseteq Q_F(f)\}$ and put
$B_F(f)=p(\ell_F(f)-1)+n(p-1)$.
\end{definition}

\begin{lemma}\label{lem:loewy}
For isolated $f$, we have $\m^{B_F(f)+1}
 \subseteq\langle x_1^p,\ldots,x_n^p\rangle^{\ell_F(f)}
 \subseteq Q_F(f)R
 \subseteq\m\Jac(f)$. Consequently, we get that $\eta(f)\le B_F(f)+1$.
\end{lemma}

\begin{proof}
The middle containments are immediate.  Write a monomial exponent as
$\alpha_i=pq_i+r_i$, $0\le r_i<p$.  If
$\sum_iq_i\le\ell_F(f)-1$, then
$|\alpha|\le p(\ell_F(f)-1)+n(p-1)=B_F(f)$.
Thus every monomial of degree at least $B_F(f)+1$ lies in
$\langle x_1^p,\ldots,x_n^p\rangle^{\ell_F(f)}$.
\end{proof}

The Loewy estimate can now be inserted simultaneously into the determinacy, separability, and Mather--Yau bounds.  By a \emph{right classification radius} we mean an integer that is simultaneously a right-determinacy level, a separability level for the right jet orbit, and a level at which the higher Milnor algebra is a complete right-equivalence invariant.

\begin{theorem}\label{thm:eta-radius}
Let $f\in\m^2$ be isolated and set  $N_\eta(f)=2\eta(f)-\ord(f)$. Then
\begin{enumerate}[label=\textup{(\roman*)},leftmargin=*]
\item $f$ is right $N_\eta(f)$-determined;
\item every right orbit map at level
      $N\ge N_\eta(f)-1$ is separable;
\item the $k[[t]]$-algebra
      \[
         M_{N_\eta(f)}(f)
         =\frac{R}{\m^{N_\eta(f)}\Jac(f)}
      \]
      is a complete right-equivalence invariant.
\end{enumerate}
Consequently,
\[
   2B_F(f)-\ord(f)+2
\]
is also a right classification radius.
\end{theorem}

\begin{proof}
Since
$\m^{\eta(f)}\subseteq\m\Jac(f)$, we have that $\m^{\eta(f)+1}\subseteq\m^2\Jac(f)$, and so from Lemma~\ref{lem:standard-inputs}\ref{std:det} we have that $f$ is right $(2\eta-d)$-determined. In addition, Assertion \textup{(ii)} is Theorem~\ref{thm:smooth}. Finally, for $d_J=\ord(\Jac(f))\ge1$ we have that 
\[
   \left\lfloor\frac{N_\eta(f)+\ord(f)+d_J+1}{2}\right\rfloor
   =
   \eta(f)+\left\lfloor\frac{d_J+1}{2}\right\rfloor
   \ge\eta(f)+1.
\]
Thus the hypothesis of
Lemma~\ref{lem:standard-inputs}\ref{std:MY} holds,
proving \textup{(iii)}.
Finally $\eta\le B_F(f)+1$ gives the stated
Frobenius-Loewy radius.
\end{proof}

\begin{corollary}\label{cor:lowmodideal}
If $r=\rmod(f)$, then
\[
 \langle x_1^p,\ldots,x_n^p\rangle^{r+1}\subseteq\m\Jac(f),\qquad
 \eta(f)\le pr+n(p-1)+1.
\]
\end{corollary}

\begin{proof}
Since $\Fmod(f)\le r$, the maximal ideal of $A_F(f)$ has $(r+1)$-st power
zero.  Apply Lemma~\ref{lem:loewy}.
\end{proof}

We first record the Frobenius algebras of the first
few lengths.  This gives a concrete picture of the invariant through Frobenius modality
three.\\

Write $y_i=x_i^p$, so $S=k[[y_1,\ldots,y_n]]$.  Every automorphism of $S$
lifts to one of $R$: take coefficientwise $p$-th roots of the images of the
$y_i$.  Thus Frobenius type is a right-coordinate invariant.\\

If $d=\Fmod(f)$, then $Q_F(f)$ has colength $d+1$ and contains
$\nF^{d+1}$.  Hence fixed-$d$ Frobenius-type classification is the finite dimensional
$\Aut_k\!\left(\frac{S}{\nF^{d+1}}\right)$-orbit problem for the colength-$(d+1)$ ideals realized
as $Q_F(f)$.

\begin{proposition}\label{prop:lowtypes}
Up to $\Aut_k(S)$, every colength-$(d+1)$ ideal $Q\subseteq S$ for
$d\le3$ has one of the following forms, with undisplayed variables killed
linearly:
\[
\begin{array}{c|l}
d&Q\\ \hline
0&\langle y_1,\ldots,y_n\rangle,\\[1mm]
1&\langle y_2,\ldots,y_n,y_1^2\rangle,\\[1mm]
2&\langle y_2,\ldots,y_n,y_1^3\rangle,\\
 &\langle y_3,\ldots,y_n\rangle+\langle y_1,y_2\rangle^2,\\[1mm]
3&\langle y_2,\ldots,y_n,y_1^4\rangle,\\
 &\langle y_3,\ldots,y_n,y_1y_2,y_2^2,y_1^3\rangle,\\
 &\langle y_3,\ldots,y_n,y_1^2,y_2^2\rangle,\\
 &\langle y_4,\ldots,y_n\rangle+\langle y_1,y_2,y_3\rangle^2.
\end{array}
\]
In characteristic $2$ there is in addition the nonalternating
$(1,2,1)$ type
\[
 \langle y_3,\ldots,y_n,y_1y_2,y_1^2+y_2^2\rangle.
\]
Only types realized as $S\cap\m\Jac(f)$ occur for hypersurface germs.
\end{proposition}

\begin{proof}
The possible Hilbert functions of a local Artin algebra of length at most
four are $(1)$, $(1,1)$, $(1,1,1)$, $(1,2)$, $(1,1,1,1)$, $(1,2,1)$, and
$(1,3)$.  All cases except $(1,2,1)$ follow by choosing a generator in each nonzero graded
piece and lifting a basis of the embedding space.  In the remaining case multiplication
induces, after choosing a generator of the one-dimensional vector space $\mathfrak q^2$, a
nonzero symmetric bilinear form on the two-dimensional vector space
$\frac{\mathfrak q}{\mathfrak q^2}$.  Changing generators acts by congruence and rescaling the
socle rescales the form.  Over an algebraically closed field of characteristic different
from $2$, the form has rank $1$ or $2$.  In characteristic $2$, the nondegenerate case
has two congruence types, alternating and nonalternating.  Writing normal forms for these
bilinear forms gives exactly the displayed ideals.
\end{proof}

\begin{proposition}
\label{prop:Fmod-zero-criterion}
For an isolated $f\in\m^2$, the following are equivalent:
\begin{enumerate}[label=(\roman*),leftmargin=*]
\item $\Fmod(f)=0$;
\item $Q_F(f)=\nF$;
\item $x_i^p\in\m\Jac(f)$ for every $i$;
\item Frobenius on $M_f=\frac{R}{\m\Jac(f)}$ kills its maximal ideal.
\end{enumerate}
If these conditions hold, then
\[
 \m^{n(p-1)+1}\subseteq\m\Jac(f),\qquad
 \mu(f)+n\le p^n,
\]
and $f$ is right $2n(p-1)$-determined.
\end{proposition}

\begin{proof}
The local algebra $A_F(f)=\frac{S}{Q_F(f)}$ has length one exactly when $Q_F(f)=\nF$.
Since $\nF=\langle x_1^p,\ldots,x_n^p\rangle S$, this is equivalent to (iii).
Because $k$ is perfect, (iii) is equivalent to (iv).
The remaining assertions follow from Lemma~\ref{lem:loewy},
Theorem~\ref{thm:eta-radius}, and Theorem~\ref{thm:length}.
\end{proof}

Put $B=n(p-1)$, $N=2B$, and
$\mathcal J_N^{(2)}=\frac{\m^2}{\m^{N+1}}$.  Let
$\mathcal Z_{n,p}\subseteq\mathcal J_N^{(2)}$ be the set of jets $u$ satisfying
$x_i^p\in\m\Jac(\widetilde u)+\m^{N+1}$ for every $i$, where $\widetilde u$ is any
lift of $u$.  The definition is independent of
the lift, since changing a lift by $\m^{N+1}$ changes its partial derivatives
by $\m^N$ and hence $\m\Jac(\widetilde u)$ by terms in $\m^{N+1}$.

\begin{theorem}
\label{thm:Fmod-zero-classification}
The jet map induces a bijection
\[
 \frac{\{f\in\m^2\mid\Fmod(f)=0\}}{\!\overset{\R}{\sim}}
 \ \longrightarrow
 \frac{\mathcal Z_{n,p}}{\R_N},
 \qquad N=2n(p-1).
\]
Equivalently, every such germ is right equivalent to a polynomial of degree
at most $N$, and two such germs are right equivalent exactly when their
$N$-jets lie in the same $\R_N$-orbit.
\end{theorem}

\begin{proof}
The forward implication follows from
Proposition~\ref{prop:Fmod-zero-criterion}.  Conversely let
$u\in\mathcal Z_{n,p}$ and let $f$ be a lift.  Since
$\m^{B+1}\subseteq\langle x_1^p,\ldots,x_n^p\rangle$,
the defining condition gives
$\m^{B+1}\subseteq\m\Jac(f)+\m^{2B+1}$.
Set
$L=\frac{\m^{B+1}+\m\Jac(f)}{\m\Jac(f)}$.
Because $\m^{2B+1}=\m^B\m^{B+1}$, the preceding inclusion implies
$L\subseteq\m^B L\subseteq\m L$.  Since also $\m L\subseteq L$, we obtain
$L=\m L$, and the Nakayama lemma gives $L=0$.  Hence
$\m^{B+1}\subseteq\m\Jac(f)$ and therefore
$\m^{2B+1}\subseteq\m\Jac(f)$.  The congruences defining
$\mathcal Z_{n,p}$ consequently become exact:
$x_i^p\in\m\Jac(f)$.  Thus $\Fmod(f)=0$. Theorem~\ref{thm:eta-radius} makes every such germ right $N$-determined, so
right equivalence is exactly the $\R_N$-orbit relation on $N$-jets.
\end{proof}

\begin{corollary}\label{cor:Fmod-zero-explicit}
\begin{enumerate}[label=(\roman*),leftmargin=*]
\item If $n=1$, the isolated germs $f$ with $\Fmod(f)=0$ are exactly
$x^d$, up to right equivalence, with $2\le d\le p-1$.
\item If $p=2$, an isolated germ $f$ has $\Fmod(f)=0$ if and only if $n$ is even
and
$f\overset{\R}{\sim} x_1x_2+x_3x_4+\cdots+x_{n-1}x_n$.
\end{enumerate}
\end{corollary}

\begin{proof}
For $n=1$, write $f=\sum_{j\ge2}a_jx^j$.  Since $f$ is isolated, there is a
least $d$ such that $a_d\ne0$ and $p\nmid d$, and then $\Jac(f)=\langle x^{d-1}\rangle$ up to a unit,
so $\m\Jac(f)=\langle x^d\rangle$.  The image of $k[[x^p]]$ in $\frac{R}{\m\Jac(f)}$ has basis
$1,x^p,x^{2p},\ldots,x^{\left\lfloor\frac{d-1}{p}\right\rfloor p}$.  Hence
$\Fmod(f)=\left\lfloor\frac{d-1}{p}\right\rfloor$.  Thus $\Fmod(f)=0$ forces $d\le p$, and since
$p\nmid d$ we have $d\le p-1$.  There is then no positive multiple of $p$ below $d$,
so $d=\ord(f)$ and $f=x^du(x)$ with $u(0)\ne0$.  Because $p\nmid d$, the unit
$u$ has a $d$-th root in $k[[x]]$, giving $f\overset{\R}{\sim} x^d$.\\

For $p=2$, apply Lemma~\ref{lem:standard-inputs}\ref{std:split}:
$f\overset{\R}{\sim} g(x_1,\ldots,x_c)+q$.  Put
$\mathfrak m_g=\langle x_1,\ldots,x_c\rangle\subseteq k[[x_1,\ldots,x_c]]$.  Suspension invariance gives
$\Fmod(g)=0$.  If $c>0$, then $\Jac(g)\subseteq\mathfrak m_g^2$, so
$\mathfrak m_g\Jac(g)\subseteq\mathfrak m_g^3$, contradicting
$x_i^2\in\mathfrak m_g\Jac(g)$.
Thus $c=0$, and the remaining nonsingular quadratic form is hyperbolic over
the algebraically closed field $k$.
\end{proof}

\begin{remark}
For $p\ge5$, isolated homogeneous cubics $f$ in three variables have
$\Fmod(f)=0$ but may have positive right modality.  Indeed their three
quadratic partial derivatives form a regular sequence, so the Jacobian
algebra has socle degree $3$, hence $\m^4\subseteq \Jac(f)$ and
$\m^5\subseteq\m\Jac(f)$.  Since $p\ge5$, one has $x_i^p\in\m\Jac(f)$.
Thus $\Fmod(f)=0$ is not equivalent to right simplicity.
\end{remark}

Finally, we compare the Frobenius directions with more general deformation directions that preserve the Jacobian ideal.  This gives a decreasing persistence filtration between proper modality, right modality, and Frobenius modality.\\
For $r\ge0$ define
\[
 \mathcal P_r(f)=
 \{h\in\m\mid\partial_i h\in\m^r\Jac(f)\text{ for every }i\},
 \qquad
 d_r(f)=
 \dim_k\left(\frac{\mathcal P_r(f)+\m\Jac(f)}{\m\Jac(f)}\right).
\]
Thus $\mathcal P_r(f)$ consists of perturbation directions whose gradient is
$r$ steps deep in the Jacobian ideal, and $d_r(f)$ counts the directions that survive
modulo the right tangent image $\m\Jac(f)$; the decreasing sequence $d_r(f)$ is the
persistence profile used below.

\begin{proposition}\label{prop:persistence}
For every isolated $f$,
\[
 \rmod(f)\ge d_0(f)\ge d_1(f)\ge\cdots\ge\Fmod(f),
\qquad
 d_0(f)\le\propermod(f).
\]
If $d(f)=\ord(\Jac(f))$ and $B_F=B_F(f)$, then
\[
 d_r(f)=\Fmod(f)\qquad
 r\ge\max\{0,B_F-d(f)\}.
\]
\end{proposition}

\begin{proof}
The inclusions
$\mathcal P_{r+1}\subseteq\mathcal P_r$ and
$S\cap\m\subseteq\mathcal P_r$ give monotonicity and the lower Frobenius
bound.  Choose independent classes represented by
$h_1,\ldots,h_e\in\mathcal P_r$.  Write the gradient of $h_a$ as
$A_a(\partial_1f,\ldots,\partial_nf)^{\mathsf T}$ with $A_a\in M_n(\m^r)$.  For
$g_t=f+\sum_at_ah_a$, its gradient is obtained from $(\partial_1f,\ldots,\partial_nf)^{\mathsf T}$ by
multiplication with $I+\sum_at_aA_a$.  This matrix is invertible on a neighborhood
of the origin in parameter space,
so $\Jac(g_t)=\Jac(f)$.  The same finite-jet transversality argument as in
Proposition~\ref{prop:Flower} gives $e\le\rmod(f)$, and for $r=0$ the
family is $\mu$-constant, giving $e\le\propermod(f)$.\\

If $r+d(f)\ge B_F$, then every partial derivative of
$h\in\mathcal P_r$ lies in $\m^{B_F}$.  Consider a monomial $c x^\alpha$ occurring in $h$ with
total degree at most $B_F$.  If some $\alpha_i$ is not divisible by $p$, then its image
under $\partial_i$ is a nonzero monomial of degree $<B_F$; distinct monomials of $h$
cannot cancel after applying the same $\partial_i$.  Thus every exponent of every such
monomial is divisible by $p$.  Since $k$ is perfect, the truncation of $h$ through degree
$B_F$ lies in $S$.  Therefore
$h=s+u$ with $s\in S\cap\m$ and $u\in\m^{B_F+1}$.  By
Lemma~\ref{lem:loewy}, $u\in\m\Jac(f)$.  Thus every surviving class is represented
by a Frobenius direction.
\end{proof}

\section{Splitting, boundedness, and proper modality}\label{sec:boundedness}

In this section, we combine the Frobenius bounds with
Lemma~\ref{lem:standard-inputs}\ref{std:split}.  We first bound
the essential corank, Milnor number, determinacy order, and multiplicity in terms of the
right modality.  We then study the Milnor-constant stratum of a semiuniversal unfolding
and prove that proper modality equals right modality.\\

Let $f\overset{\R}{\sim} g(x_1,\ldots,x_c)+Q(z_1,\ldots,z_{n-c})$ be the splitting-lemma
decomposition, where $Q$ is nonsingular and $g$ is the residual germ, and write
$\m_x=\langle x_1,\ldots,x_c\rangle\subseteq k[[x_1,\ldots,x_c]]$.  We call $c=c(f)$ the \emph{essential corank}; equivalently, it is the number of
variables that remain after all nonsingular quadratic summands have been split off.  Two isolated hypersurface germs are
\emph{stably right equivalent} if, after adjoining finitely many variables and adding
nonsingular quadratic forms in the new variables, the resulting germs are right
equivalent.  By Lemma~\ref{lem:standard-inputs}\ref{std:split}, the stable
right-equivalence class is represented by
the residual germ.

\begin{proposition}\label{prop:susp}
With this notation, $A_F(f)\cong A_F(g)$,
$\Fmod(f)=\Fmod(g)$, $\ell_F(f)=\ell_F(g)$, and $\mu(f)=\mu(g)$.
\end{proposition}

\begin{proof}
The partial derivatives of the nonsingular quadratic block generate the quadratic
variables.  Every nonconstant $p$-th-power monomial involving one of those variables
therefore belongs to $\m\Jac(f)$, while setting the quadratic variables equal to zero
maps $\m\Jac(f)$ onto $\m\Jac(g)$.
\end{proof}

Applying Theorem~\ref{thm:length} to the residual germ and using the suspension
invariance in Proposition~\ref{prop:susp} gives a bound involving only the residual
number of variables.  In particular, the variables belonging to the nonsingular
quadratic summand do not occur in the exponent.  We obtain
\begin{equation}\label{eq:essential-bound}
 \mu(f)+c(f)\le p^{c(f)}\bigl(\Fmod(f)+1\bigr).
\end{equation}

\begin{lemma}\label{lem:char2-corank}
Assume $p=2$, and let $g$ be an isolated residual germ
in $c$ variables. Then we have that 
$\rmod(g)\ge c+\binom c3$. 
\end{lemma}

\begin{proof}
The assertion is immediate for $c=0$. Assume $c>0$.
Let $V_2^{\mathrm{sq}}=\Span_k\{x_1^2,\ldots,x_c^2\}$ and $V_3=k[x_1,\ldots,x_c]_3$. Since $\Hess(g)(0)=0$, its quadratic part belongs to
$V_2^{\mathrm{sq}}$. Consider the unfolding $g+h_2+h_3$ for $(h_2,h_3)\in V_2^{\mathrm{sq}}\oplus V_3$. Its parameter-to-$3$-jet map is injective, and the
parameter space has dimension $c+\binom{c+2}{3}$.\\

Write a coordinate change as $\varphi(x)=Ax+u(x)$ with
$u\in\m^2k[[x]]^c$. In characteristic two, substituting
$\varphi$ into a square quadratic form produces no cubic
contribution from $u$. Substitution into a cubic form
affects its cubic part only through $A$.
Consequently the action on the $3$-jets in this family
factors through $\GL_c$.\\

Every right-equivalence fiber in the parameter space is
therefore contained in a $\GL_c$-orbit at the $3$-jet
level, and has dimension at most $c^2$.
After restricting to a neighborhood of the origin where
the fibers are isolated, the relative modality is at least $c+\binom{c+2}{3}-c^2=c+\binom c3$. The modality of an unfolding is at most the right
modality of its central germ, proving the assertion.
\end{proof}

For $D\ge0$ define
\[
 C_{p,D}=\begin{cases}
 \max\left\{c\in\mathbb N\mid\binom c3\le D\right\},&p\ne2,\\
 \max\left\{
   c\in\mathbb Z_{\ge0}:
   c+\binom c3\le D
\right\},&p=2.
 \end{cases}
\]
Put
\[
 B_{p,D}=pD+C_{p,D}(p-1),\qquad
 N_{p,D}=2B_{p,D},\qquad
 \widehat N_{p,D}=\max\{2,N_{p,D}\}.
\]

\begin{theorem}
\label{thm:uniform}
Fix $p>0$ and $D\ge0$, and let $f$ be isolated with $\rmod(f)\le D$.  After splitting
off a nonsingular quadratic block, the residual germ satisfies:
\begin{enumerate}[label=\textup{(\roman*)},leftmargin=*]
\item it has at most $C_{p,D}$ variables;
\item its Milnor number is at most $p^{C_{p,D}}(D+1)-C_{p,D}$;
\item its right classification radius is at most $N_{p,D}$;
\item if it is nonzero, its multiplicity is at most $p(D+1)$.
\end{enumerate}
\end{theorem}

\begin{proof}
For $p\ne2$, the residual germ has order at least three, and
\cite[Proposition~2.18]{GreuelNguyen2016} gives
$\rmod(g)\ge\binom{c+\ord(g)-1}{\ord(g)}-c^2\ge\binom c3$.
For $p=2$ use Lemma~\ref{lem:char2-corank}.  Thus $c\le C_{p,D}$ in all
characteristics.\\

The Milnor bound follows from \eqref{eq:essential-bound} and
$\Fmod(g)\le\rmod(g)\le D$.  Theorem~\ref{thm:eta-radius}, applied in $c$ variables, gives the radius $2pD+2c(p-1)\le N_{p,D}$.\\

Finally, Corollary~\ref{cor:order} applied to $g$ gives
\[
 D+1\ge
 \binom{c+\left\lceil\frac{\ord(g)}{p}\right\rceil-1}{c}.
\]
For $c>0$, the right-hand side is at least
$\left\lceil\frac{\ord(g)}{p}\right\rceil$, proving $\ord(g)\le p(D+1)$.
\end{proof}

The residual determinacy estimate also gives a determinacy order for the original germ,
not only for its stable class.  The point is that the partial derivatives of the
nonsingular quadratic block generate all split variables.  Hence every sufficiently
high-order monomial involving a split variable already lies in the full Jacobian
contribution, and the residual containment extends to the original power-series ring.

\begin{corollary}
\label{cor:uniform-full-determinacy}
Fix $p>0$ and $D\ge0$.  Every isolated hypersurface germ $f$ with
$\rmod(f)\le D$ is right $\widehat N_{p,D}$-determined, independently of the
number of ambient variables.
\end{corollary}

\begin{proof}
Write $f\overset{\R}{\sim} g(x_1,\ldots,x_c)+Q(z_1,\ldots,z_{n-c})$ as in
Lemma~\ref{lem:standard-inputs}\ref{std:split}.  If $c=0$, the germ is a nonsingular quadratic form and is right
$2$-determined.  Assume $c>0$ and put
$B=p(\ell_F(f)-1)+c(p-1)=B_F(g)$.  Since
$\ell_F(f)-1\le\Fmod(f)\le D$ and $c\le C_{p,D}$, one has
$B\le B_{p,D}$.  Lemma~\ref{lem:loewy}, applied to the residual germ, gives
$\m_x^{B+2}\subseteq\m_x^2\Jac(g)$.  On the other hand the partial derivatives of the
nonsingular quadratic form $Q$ generate the ideal
$\langle z_1,\ldots,z_{n-c}\rangle$.  Since $B\ge1$, every monomial of total
degree at least $B+2$ that involves a split variable belongs to
$\m^2\langle z_1,\ldots,z_{n-c}\rangle$.  Therefore, in the full power-series
ring,
\[
 \m^{B+2}\subseteq\m^2\Jac(f).
\]

Lemma~\ref{lem:standard-inputs}\ref{std:det} shows that $f$ is
right $(2B-\ord(f)+2)$-determined, hence right $2B$-determined
because $\ord(f)\ge2$. Since
$2B\le N_{p,D}\le\widehat N_{p,D}$, the assertion follows.
\end{proof}

\begin{corollary}
\label{cor:stable-milnor-growth}
Let $r=\rmod(f)$.  Then
\[
 \mu(f)\le p^{\,2+\sqrt[3]{6r}}(r+1).
\]
Consequently, for fixed $p$, any sequence $(f_j)$ of isolated hypersurface singularities with
$\mu(f_j)\to\infty$ has $\rmod(f_j)\to\infty$, even if the ambient dimension is allowed to
tend to infinity.
\end{corollary}

\begin{proof}
The definition of $C_{p,r}$ and Lemma~\ref{lem:char2-corank} imply
$\binom c3\le r$ for the residual essential corank in every characteristic.  Hence
$c(c-1)(c-2)\le6r$ and therefore $c\le2+\sqrt[3]{6r}$.  Apply
\eqref{eq:essential-bound} and $\Fmod(f)\le r$.
\end{proof}

\begin{corollary}
Fix $p>0$.  If $\{f_j\}_{j=1}^\infty$ is any sequence of isolated hypersurface singularities in
characteristic $p$ with $\mu(f_j)\to\infty$, then
$\rmod(f_j)\to\infty$.
\end{corollary}

\begin{proof}
This is immediate from Corollary~\ref{cor:stable-milnor-growth}.
\end{proof}

\begin{remark}
Nguyen proved that $\mu(f)\le4p$ for singularities of right modality at most two
\cite[Corollary~3.7]{Nguyen2019}.  Corollary~\ref{cor:stable-milnor-growth} gives an analogous
bounded-Milnor-number statement in every finite modality and, moreover, makes the bound
independent of the ambient dimension.  This is compatible with the broader finiteness
phenomenon conjectured by Greuel--Nguyen for singularities with an arbitrary number of
moduli \cite{GreuelNguyen2016}.
\end{remark}

\begin{corollary}\label{cor:order}
If $r=\ord(f)$, then
\[
 \Fmod(f)\ge
 \binom{n+\left\lceil\frac{r}{p}\right\rceil-1}{n}-1.
\]
\end{corollary}

\begin{proof}
Since $\m\Jac(f)\subseteq\m^r$,
$Q_F(f)\subseteq S\cap\m^r =\nF^{\left\lceil\frac{r}{p}\right\rceil}$.
Compare colengths.
\end{proof}

\begin{proposition}\label{prop:BP}
Let $f=\sum_{i=1}^n x_i^{a_i}$ with $p\nmid a_i$, and put $b_i=a_i-1$.  Then
\[
 \Fmod(f)+1=
 \prod_{i=1}^n
 \left\lceil\frac{b_i}{p}\right\rceil
 +\#\{i\mid p\mid b_i\}.
\]
\end{proposition}

\begin{proof}
Up to units,
$\Jac(f)=\langle x_1^{b_1},\ldots,x_n^{b_n}\rangle$.  The monomials outside $\m\Jac(f)$ are those
with every exponent strictly below the corresponding $b_i$, together with the
exceptional pure monomials $x_i^{b_i}$.  Intersecting this set with the
$p$-lattice gives the stated count.
\end{proof}

We now study proper modality.  We realize the Milnor number as a local fiber length in a
finite flat family and compare the corresponding length strata with Rosenlicht strata.
Recall Definition~\ref{def:proper-modality}.  We use the standard semiuniversal
unfolding with trivial section
$F(x,t)=f(x)+\sum_{j=1}^{b_f}t_jg_j(x)$, where $b_f=\mu(f)+n-1$ and the classes of the
$g_j$ form a basis of $\frac{\m}{\m\Jac(f)}$.  We write
$\mathcal C_{\mu(f)}(F)=\{t\mid\mu(F_t)=\mu(f)\}$ for its Milnor-constant stratum near
the origin.

\begin{theorem}
\label{thm:pmodlower}
For every isolated $f$, one has $\propermod(f)\le\rmod(f)$.
\end{theorem}

\begin{proof}
Put $\mu=\mu(f)$ and choose an irreducible component
$C\subseteq\mathcal C_\mu(F)$ through $0$ of dimension $\propermod(f)$.
Take $N\ge2\mu$ sufficiently large.  By
Corollary~\ref{cor:mu-sep}, every $\mu$-constant member has separable
$N$-jet orbit, and the orbit tangent space has constant codimension
$\mu+n-1$.  At the origin the semiuniversal parameter space is complementary
to this tangent space; hence transversality persists on a dense open subset
of $C$.\\

A positive-dimensional right-equivalence fiber would have a smooth point with
a nonzero tangent vector lying both in the parameter tangent space and in the
orbit tangent space.  This contradicts transversality.  Hence the general
right-equivalence fiber on $C$ is zero-dimensional, and right-completeness
gives
$\rmod(f)\ge\dim C=\propermod(f)$.
\end{proof}

Put $M=\mu(f)$.  By Lemma~\ref{lem:standard-inputs}\ref{std:det}, we may replace
$f$ by a right-equivalent polynomial
representative.  Choose polynomials
$h_1,\ldots,h_{M-1}\in\m^2$ so that
$x_1,\ldots,x_n,h_1,\ldots,h_{M-1}$
represent a basis of $\frac{\m}{\m\Jac(f)}$, and set
$H(x,t)=f(x)+\sum_{j=1}^{M-1}t_jh_j(x)$.
Then $F(x;a,t)=H(x,t)+\sum_i a_ix_i$ is the standard semiuniversal unfolding
with trivial section: the classes of $x_1,\ldots,x_n,h_1,\ldots,h_{M-1}$ form a
basis of $\frac{\m}{\m\Jac(f)}$, the tangent space of the deformation functor with section.  Eliminating the $a_i$ from the
critical equations gives
\[
 \Pi:
 \mathbb A^{n+M-1}_{u,t}
 \longrightarrow
 \mathbb A^{n+M-1}_{a,t},
 \qquad
 (u,t)\longmapsto
 \bigl(-\partial_1H(u,t),\ldots,-\partial_nH(u,t),t\bigr).
\]
For $c=(u,t)$ define
$G_c(y)= H(u+y,t)-H(u,t) -\sum_i\partial_iH(u,t)y_i$.
Then $G_c\in\langle y_1,\ldots,y_n\rangle^2$, and the completed local ring of the fiber of $\Pi$ at
$c$ is
$\frac{k[[y]]}{\Jac_y(G_c)}$, where
$\Jac_y(G_c)=\langle \frac{\partial G_c}{\partial y_1},\ldots,\frac{\partial G_c}{\partial y_n}\rangle$.
Thus the local scheme-theoretic fiber length is $\mu(G_c)$.\\

Let $c_0=(0,0)$ and $b_0=\Pi(c_0)$.  Since $f$ is isolated, $c_0$ is an isolated
point of the fiber over $b_0$.  By \cite[Lemma~37.41.1, Tag~02LK]{stacks-project}, after an
étale base change around $b_0$ there is an open neighborhood of $c_0$ that is finite over
the new base.  Shrinking this neighborhood and the base, we may moreover arrange that the
chosen neighborhood contains no other point of the special fiber.\\

Let $c_0=(0,0)$ and $b_0=\Pi(c_0)$. Since $f$ is isolated, $c_0$ is an isolated point of the fiber over $b_0$. By \cite[Lemma~37.41.1, Tag~02LK]{stacks-project}, after an étale base change around $b_0$ there is an open neighborhood of $c_0$ finite over the new base. Shrink so that this neighborhood contains no other point of the special fiber, and replace it by the connected component containing $c_0$. Because it is étale over the smooth $(u,t)$-space, this source is regular; a connected regular scheme is integral. Replace the target by the connected component containing the image of $c_0$ and shrink once more. The resulting source and target are smooth integral varieties of the same dimension, and the finite map between them is surjective. Since the source local rings are Cohen--Macaulay and the target local rings are regular, miracle flatness \cite[Lemma~10.128.1, Tag~00R4]{stacks-project} makes the map flat.\\

We have obtained a finite flat surjective map
\[
 \pi:X\longrightarrow Y.
\]
By construction, the special fiber of this chosen local factor is supported only at the
distinguished point, and its completed local algebra is
$\frac{k[[y]]}{\Jac(f)}$, of length $M$.  Hence the finite flat rank of $\pi$ is exactly $M$.
Write $x_0\in X$ for the distinguished point.  We call $\pi$ the \emph{critical
cover} associated with the chosen semiuniversal unfolding; its points record critical
points of nearby fibers together with their deformation parameters.  For a closed point $x\in X$, put $\lambda(x)= \dim_k\mathcal O_{\pi^{-1}(\pi(x)),x}$. Thus $\lambda(x)=\mu(G_x)$.
For an irreducible subvariety, its generic local length
means the common value of $\lambda$ on a dense open
subset of its closed points.\\

Let $\iota=[k(X):k(Y)]_{\mathrm{insep}}$ be the inseparable degree of the critical
cover.  For a closed point $x\in X$, let $e_\pi(x)$ denote the Gaffney--Lazarsfeld
local degree: after strict henselization at $y=\pi(x)$, if the generic fiber of the
local factor specializing to $x$ is $\prod_jL_j$ over the fraction field $F$, then
$e_\pi(x)=\sum_j[L_j:F]_{\mathrm{sep}}$; see
\cite[\S1, especially (1.2)--(1.4)]{Lazarsfeld1980Thesis}.

\begin{lemma}\label{lem:lambda-qe}
For every closed point $x\in X$, $\lambda(x)=\iota e_\pi(x)$.  Moreover, $\iota$ is a
power of $p$ dividing $M$, and every local Milnor number on the critical cover is
divisible by $\iota$.
\end{lemma}

\begin{proof}
Let $y=\pi(x)$ and let
$A=\mathcal O_{Y,y}^{\mathrm{sh}}$
be the strict henselization, with fraction field $F$.  Since $\pi$ is finite
flat,
$X_A=X\times_Y\Spec A$
is finite flat over $A$.  Because $A$ is strictly henselian, $X_A$ decomposes
into the local factors corresponding to the points of the geometric fiber.
Let $B_x$ be the factor corresponding to $x$.  Its rank is the local special
fiber length:
$\rank_A B_x=\lambda(x)$.\\

Since $Y$ is normal, $A$ is a normal local domain, and
$\frac{F}{k(Y)}$ is separable algebraic.  Thus separable base change preserves the
inseparable degree of $\frac{k(X)}{k(Y)}$.  The generic algebra
$B_x\otimes_AF$
is a finite product of fields $L_1,\ldots,L_s$, and every $\frac{L_j}{F}$ has
inseparable degree $\iota$.  Hence
$\lambda(x) =\dim_F(B_x\otimes_AF) =\sum_j[L_j:F] =\iota\sum_j[L_j:F]_{\mathrm{sep}}$.
By the preceding étale-local definition of local degree, the last sum is
$e_\pi(x)$.
Therefore $\lambda(x)=\iota\,e_\pi(x)$.\\

Finally
$M=[k(X):k(Y)] =\iota[k(X):k(Y)]_{\mathrm{sep}}$,
so $\iota$ is a power of $p$ dividing $M$.
\end{proof}

\begin{remark}
The local degree used here is the algebro-geometric local number of sheets of a finite
map.  Gaffney--Lazarsfeld used this invariant to study higher ramification loci of finite
coverings \cite{GaffneyLazarsfeld1980}; the form and semicontinuity statements used
above are developed in the thesis of Lazarsfeld \cite{Lazarsfeld1980Thesis}.  The factor
$\iota=[k(X):k(Y)]_{\mathrm{insep}}$ is the extra positive-characteristic thickness that is
absent for a generically separable complex covering.
\end{remark}

For $m\ge1$ put $\Lambda_{\ge m}=\{x\in X\mid\lambda(x)\ge m\}$.  By
Lemma~\ref{lem:lambda-qe}, the local length is the constant inseparable factor
$\iota$ times the local degree $e_\pi$.  Since the local degree is upper
semicontinuous, each $\Lambda_{\ge m}$ is closed.  We will use these closed length
strata to compare proper and right modality.

\begin{proposition}
\label{prop:critical-length-pmod}
After shrinking $X$ around $x_0$, the following holds for every closed point $x\in X$. Put $h=G_x$ and $m=\mu(h)=\lambda(x)>0$. Then
\[
\codim_{X,x}\Lambda_{\ge m}=m-1-\propermod(h).
\]
Equivalently, $\dim_x\Lambda_{\ge m}=M+n-m+\propermod(h)$.
\end{proposition}

\begin{proof}
Choose $N\ge2M$ and put
$E_N=\frac{\m^2}{\m^{N+1}}$
in the $y$-variables.  For $v\in E_N$ choose a lift $\widetilde v$ and set
$\mu_N(v)= \dim_k \left(\frac{k[[y]]}{\Jac(\widetilde v)+\m^{M+1}}\right)$.
Changing the lift changes its partial derivatives only by elements of
$\m^N\subseteq\m^{M+1}$, so $\mu_N$ is well defined.\\

Suppose $\mu_N(v)=m\le M$.  The maximal ideal of the Artin algebra
$\frac{k[[y]]}{\Jac(\widetilde v)+\m^{M+1}}$
has $m$-th power zero, so
$\m^m\subseteq \Jac(\widetilde v)+\m^{M+1}$.
Put
\[
 L_v=\frac{\m^m+\Jac(\widetilde v)}{\Jac(\widetilde v)}.
\]
Modulo $\Jac(\widetilde v)$, the inclusion $\m^m\subseteq\Jac(\widetilde v)+\m^{M+1}$ gives $L_v\subseteq \m^{M+1-m}L_v$. Since $m\le M$, one has $M+1-m\ge1$, and therefore $L_v\subseteq\m L_v$. Nakayama's lemma gives $L_v=0$.
Thus $\m^m\subseteq \Jac(\widetilde v)$, and therefore
$\m^{M+1}\subseteq \Jac(\widetilde v)$.  Consequently
$\mu(\widetilde v)=\mu_N(v)=m$.
The converse is immediate.  Hence
$\mathcal M_m=\{v\in E_N\mid\mu_N(v)=m\}$
is exactly the Milnor-number-$m$ locus and is a locally closed determinantal
subset.\\

At $h$ one has
$T_h= \frac{\m\Jac(h)}{\m^{N+1}} \subseteq E_N, \qquad \dim_k\left(\frac{E_N}{T_h}\right) =\dim_k\left(\frac{\m^2}{\m\Jac(h)}\right) =m-1$.
Choose a linear complement $V$ to $T_h$.  A basis of $V$, together with
$y_1,\ldots,y_n$, gives a basis of $\frac{\m}{\m\Jac(h)}$, so
$h+\sum_i a_i y_i+v,\qquad v\in V$,
is semiuniversal with trivial section.  For $m>0$, its Milnor-number-$m$ locus is
contained in $a=0$.  Therefore
$\dim_0\{v\in V\mid\mu(h+v)=m\} =\propermod(h)$.\\

The map
$\Theta: \R_N\times(j^Nh+V)\to E_N, \qquad (\varphi,j^Nh+v)\mapsto\varphi(j^Nh+v)$
has surjective differential at $(e,j^Nh)$ because the group directions give
$T_h$ and $V$ gives a complement.  Thus $\Theta$ is smooth there.  Since
right equivalence preserves Milnor number,
$\codim_{E_N,j^Nh}\mathcal M_m =m-1-\propermod(h)$.\\

Now let
$\Phi:X\to E_N,\qquad z\mapsto j^NG_z$.
In the original critical-incidence coordinates,
$D_j= \frac{\partial G}{\partial t_j} = h_j(u+y)-h_j(u) -\sum_i\partial_i h_j(u)y_i$.
At $x_0$ these classes form a basis of
$\frac{\m^2}{\m\Jac(f)}$.  Consider the finite relative module
$\mathcal E_N=\frac{\langle y_1,\ldots,y_n\rangle^2}{\langle y_1,\ldots,y_n\rangle \Jac_y(G)+\langle y_1,\ldots,y_n\rangle^{N+1}}$.
The cokernel of the sections $D_1,\ldots,D_{M-1}$ vanishes at $x_0$ and
therefore vanishes after shrinking $X$.  Every local fiber length is at most
the total rank $M$, hence $\mu(G_z)\le M$ and
$\langle y_1,\ldots,y_n\rangle^{N+1}\subseteq\langle y_1,\ldots,y_n\rangle \Jac(G_z)$.  Thus the fiber of $\mathcal E_N$ at
$z$ is exactly
$\frac{\m^2}{\m\Jac(G_z)}$.
It follows that the critical-family directions surject onto
$\frac{E_N}{T_{G_z}}$.\\

Therefore
$\Psi: \R_N\times X\to E_N,\qquad (\varphi,z)\mapsto\varphi(\Phi(z))$
is smooth at $(e,x)$.  Smooth base change preserves local codimension, so
$\codim_{X,x} \{z\mid\mu(G_z)=m\} =m-1-\propermod(h)$.
Finally $\lambda$ is upper semicontinuous and $\lambda(x)=m$, so after
shrinking around $x$ there are no points with $\lambda\ge m+1$.  Hence the
exact-$m$ locus agrees locally with $\Lambda_{\ge m}$.
\end{proof}

The next step connects orbit parameters with the length strata just constructed.  We
start from a stratum realizing right modality and show that it gives a family of proper
moduli.  Let $Y$ carry an algebraic family of germs and let $\iota_N:Y\to \mathcal J_N$ be its
parameter-to-jet map for a sufficiently large $N$.  We call
$\{z\in Y\mid\iota_N(z)\in\R_N\cdot\iota_N(y)\}$ the \emph{right-equivalence fiber} of
$y$; for sufficiently large $N$ this set is independent of the chosen jet level.

\begin{lemma}\label{lem:rosenlicht-to-pmod}
Let $Y$ be irreducible and suppose the family has constant Milnor number $m>0$.  If the
general right-equivalence fiber has dimension $i$, then a general member $g$ satisfies
$\propermod(g)\ge\dim Y-i$.
\end{lemma}

\begin{proof}
Choose a general member $g$ and a jet level $N\ge2m$ sufficiently high for
right determinacy and separability.  Let
$T_g= T_{j^Ng}(\R_N\cdot j^Ng)$
and choose a linear complement
$\mathcal J_N=T_g\oplus U$.
Then
$\dim U=m+n-1$.
The action map
$\Theta: \R_N\times(j^Ng+U)\to \mathcal J_N$
is smooth at $(e,j^Ng)$.\\

Form the fiber product, using $\iota_N:Y\to \mathcal J_N$ and the action map $\Theta$,
$Z=Y\times_{\mathcal J_N}\bigl(\R_N\times(j^Ng+U)\bigr)$
near the point lying over $g$.  The projection $Z\to Y$ is a base change of
$\Theta$, hence smooth of relative dimension
$\rho= \dim\R_N+\dim U-\dim \mathcal J_N$.
Since the orbit codimension is $m+n-1=\dim U$, one has
$\rho=\dim\Stab_{\R_N}(j^Ng)$, and this stabilizer dimension
is constant on the Milnor-number-$m$ locus.\\

Consider the projection
$\sigma:Z\to U$.
For fixed $u$, the image in $Y$ of a fiber of $\sigma$ lies in a single
right-equivalence fiber, whose general dimension is at most $i$.
For fixed $u$ and a fixed member of that right-equivalence fiber, the choices
of the group element form a stabilizer coset of dimension $\rho$.
Hence the general fiber of $\sigma$ has dimension at most $i+\rho$.
Therefore
$\dim\overline{\sigma(Z)} \ge \dim Z-(i+\rho) = \dim Y-i$.\\

Every $j^Ng+u$ occurring in $\sigma(Z)$ is right equivalent to a member of
$Y$, hence has Milnor number $m$.  Shrink the slice $g+U$ around $g$ so that
$\mu\le m$, which is possible by upper semicontinuity.  On the other hand the condition
$\mu\ge m$ is closed on the finite jet space (it is a determinantal rank condition), so
the closure of $\sigma(Z)$ in this neighborhood is still $m$-constant.  It contains
$0\in U$ in its closure.  Consequently
$\propermod(g)\ge\dim\overline{\sigma(Z)}\ge\dim Y-i$.
\end{proof}

\begin{lemma}[Higher purity for local degree]\label{lem:lazarsfeld-purity}
Let $\pi:X\to Y$ be a finite surjective morphism of smooth integral varieties and let
$x_0\in X$.  Write $d=e_\pi(x_0)$.  After replacing $X$ and $Y$ by sufficiently small
neighborhoods of $x_0$ and $\pi(x_0)$, the following holds.  If $T\subseteq X$ is an
irreducible subvariety through $x_0$ whose generic local degree $e_\pi(T)$ is strictly less
than $d$, then there is an irreducible codimension-one subvariety $T_1\subseteq T$ through
$x_0$ with
\[
 e_\pi(T_1)>e_\pi(T).
\]
Consequently, if $e_\pi(T)\ge e_0$, then there is an irreducible
$T^*\subseteq T$ through $x_0$ with generic local degree $d$ and
\[
 \codim_T(T^*)\le d-e_0.
\]
\end{lemma}

\begin{proof}
The codimension-one increase is the local form of the higher-purity theorem of Lazarsfeld
\cite[Theorem~2.2 and statement~(2.4)]{Lazarsfeld1980Thesis}.  Iterating it, the generic
local degree increases by at least one at each step and can never exceed $d$.  Starting from
a generic degree at least $e_0$, at most $d-e_0$ steps are required.
\end{proof}

\begin{theorem}\label{thm:pmod-rmod-general}
For every isolated hypersurface singularity over an algebraically closed field
of characteristic $p>0$,
\[
 \propermod(f)=\rmod(f).
\]
\end{theorem}

\begin{proof}
Theorem~\ref{thm:pmodlower} gives
$\propermod(f)\le\rmod(f)$.  Put $r=\rmod(f)$ and assume $r>0$.\\

Use the semiuniversal unfolding
$F(x;a,t)=H(x,t)+\sum_i a_ix_i$
introduced above.  By the Rosenlicht parameter-count description of modality,
there is an irreducible locally closed subset $Y_0$ of the parameter space,
with $0\in\overline{Y_0}$, such that the general right-equivalence fiber on
$Y_0$ has dimension $i$ and
$\dim Y_0-i=r$.
Replacing $Y_0$ by a dense open subset, we may suppose that its Milnor number
is constant, say $m$.  Since $r>0$, one must have $m>0$. This is true since  Since $\dim Y_0-i=r>0$, the general member of $Y_0$ cannot be smooth, because all smooth germs form a single right-equivalence class. Thus $m>0$. In the chosen semiuniversal coordinates, $\frac{\partial F}{\partial x_i}(0;a,t)=a_i$. Hence every parameter value with $a\ne0$ gives a smooth germ. The dense Milnor-number-$m$ part of $Y_0$ is therefore contained in the linear subspace $a=0$; after replacing $Y_0$ by this dense open subset, we may assume $a=0$ identically on $Y_0$.\\

By Lemma~\ref{lem:rosenlicht-to-pmod}, for a general member $g$ of $Y_0$,
$\propermod(g)\ge r$.\\

Because $a=0$, the point $u=0$ is a critical point of every member of
$Y_0$.  Thus the dense Milnor-number-$m$ part of $Y_0$ has a canonical lift to
the critical incidence space.  Pull this lifted set through the étale
localization defining $X$, and choose an irreducible component $Z$ whose
closure contains $x_0$.  For general $z\in Z$,
$\lambda(z)=m,\qquad \propermod(G_z)\ge r$.
Put
$D=M+n-m+r$.
Proposition~\ref{prop:critical-length-pmod} gives
$\dim_z\Lambda_{\ge m}\ge D$
for general $z\in Z$.\\

After shrinking around $x_0$, write
$\Lambda_{\ge m}=S_1\cup\cdots\cup S_N$
as its irreducible decomposition.  The union of those $S_j$ of dimension at
least $D$ contains a dense open subset of $Z$, hence contains $Z$.  Since
$Z$ is irreducible, it lies in one such component, say $S$.  Thus
$x_0\in S,\qquad \dim S\ge M+n-m+r$.\\

Let
$\iota=[k(X):k(Y)]_{\mathrm{insep}}$.
By Lemma~\ref{lem:lambda-qe},
$e_\pi(x_0)=\frac{M}{\iota},\qquad e_\pi(S)\ge\frac{m}{\iota}$.
Since $e_\pi(S)\ge\frac{m}{\iota}$ and
$e_\pi(x_0)=\frac{M}{\iota}$, Lemma~\ref{lem:lazarsfeld-purity} produces an irreducible
subvariety $S^*\subseteq S$ through $x_0$ with generic local degree
$\frac{M}{\iota}$ and
\[
 \codim_S(S^*)\le\frac{M-m}{\iota}.
\]
By Lemma~\ref{lem:lambda-qe}, the generic local length on $S^*$ is therefore $M$.
Since $\Lambda_{\ge M}$ is closed, $S^*\subseteq\Lambda_{\ge M}$, and
\[
 \dim S^*\ge M+n-m+r-\frac{M-m}{\iota}.\\
\]

Applying Proposition~\ref{prop:critical-length-pmod} at $x_0$ gives
$\dim_{x_0}\Lambda_{\ge M} =n+\propermod(f)$.
Consequently
$\propermod(f) \ge r+(M-m)\left(1-\frac{1}{\iota}\right) \ge r =\rmod(f)$.
Together with the reverse inequality already proved, this gives $\propermod(f)=\rmod(f)$.
\end{proof}

\begin{remark}
Gabrielov proved equality of modality and proper modality for isolated complex analytic
hypersurface singularities \cite{Gabrielov1974}.  In positive characteristic, Nguyen
proved the equality for univariate germs \cite{Nguyen2014} and for all germs of right
modality at most two \cite{Nguyen2019}; in the latter paper the general equality was
formulated as Conjecture~2.2.  Theorem~\ref{thm:pmod-rmod-general} proves that
conjecture for arbitrary isolated hypersurface singularities.
\end{remark}

\begin{corollary}
\label{cor:insep-forces-pmod}
With $\iota$ as above,
\[
 \propermod(f)=\rmod(f)
 \ge
 \mu(f)\left(1-\frac{1}{\iota}\right).
\]
If $p\nmid\mu(f)$, then $\iota=1$.  If $p=2$ and $n$ is odd, then
$\iota\ge2$ and
\[
 \propermod(f)=\rmod(f)\ge\frac{\mu(f)}{2}.
\]
\end{corollary}

\begin{proof}
Since $\lambda=\iota e_\pi$ and $M=\mu(f)$,
$\Lambda_{\ge M} = R_{\frac{M}{\iota}-1}(\pi)$
in Lazarsfeld notation.  His Theorem~2.2 gives
$\codim_{X,x_0} \Lambda_{\ge M} \le\frac{M}{\iota}-1$.
Proposition~\ref{prop:critical-length-pmod} gives
$\codim_{X,x_0} \Lambda_{\ge M} =M-1-\propermod(f)$,
which proves the inequality. Since $\iota$ is a power of $p$ dividing $M$, the condition $p\nmid M$ forces
$\iota=1$.  In characteristic $2$, the Hessian in the $x$-variables has zero diagonal and is
alternating.  The Jacobian matrix of the critical map $\Pi$ is block triangular, with
this Hessian (up to sign) as its $u$-block and an identity matrix in the $t$-block.
If $n$ is odd, the Hessian determinant vanishes identically, so $d\Pi$ is not generically
of full rank.  Since $\Pi$ is finite dominant, it is generically inseparable and
$\iota\ge2$.
\end{proof}

\begin{corollary}
\label{cor:char2-odd-corank}
Assume $p=2$ and let $c=c(f)$ be odd.  Then
\[
 \rmod(f)\ge\frac{\mu(f)}{2}\ge2^{c-1}.
\]
\end{corollary}

\begin{proof}
Split off the nonsingular quadratic part and let $g$ be the residual germ in $c$
variables.  Suspension invariance gives $\mu(g)=\mu(f)$ and
$\rmod(g)=\rmod(f)$.  Since $c$ is odd, Corollary~\ref{cor:insep-forces-pmod}
applied to $g$ gives $\rmod(g)\ge\frac{\mu(g)}{2}$.  For a residual germ in
characteristic two one has $\Jac(g)\subseteq\m_x^2$.  The Jacobian ideal is an
$\m_x$-primary parameter ideal in the regular local ring $k[[x_1,\ldots,x_c]]$;
hence its colength equals its Hilbert--Samuel multiplicity.  Monotonicity of
multiplicity under inclusion gives
$\mu(g)=e(\Jac(g))\ge e(\m_x^2)=2^c$.  Combining the two inequalities proves the
claim.
\end{proof}

Combining the Frobenius lower bound, the comparison with proper modality, and the
equality of proper and right modality summarizes the two lower-bound mechanisms used
above.  The first comes from deformation directions surviving modulo the right tangent
image, while the second comes from the Frobenius length estimate:
\[
 \Fmod(f)\le d_0(f)\le\propermod(f)=\rmod(f),
 \qquad
 \propermod(f)\ge
 \left\lceil\frac{\mu(f)+n}{p^n}\right\rceil-1.
\]

\begin{proposition}\label{prop:rmod-upper}
Every isolated hypersurface germ satisfies $\propermod(f)=\rmod(f)\le\mu(f)-1$.
\end{proposition}

\begin{proof}
Choose a basis
$x_1,\ldots,x_n,h_1,\ldots,h_{\mu-1}$
of $\frac{\m}{\m\Jac(f)}$ with $h_j\in\m^2$.  A nonzero linear parameter in the
corresponding semiuniversal unfolding makes the germ smooth, and all smooth
germs form one right orbit.  Hence every positive-dimensional quotient is
contained in the $(\mu-1)$-dimensional subspace on which the linear
parameters vanish.
\end{proof}

\begin{remark}
    If $n=1$, then we have that $\Fmod(f)=\propermod(f)=\rmod(f)
 =\left\lfloor\frac{\mu(f)}{p}\right\rfloor$. This is true since the right-modality formula is the univariate classification of Nguyen
\cite{Nguyen2014} and the Frobenius formula follows from the fact that
$\m\Jac(f)=\langle x^{\mu+1}\rangle$, and
Theorem~\ref{thm:pmod-rmod-general} identifies proper and right modality.
\end{remark}

\section{Finite classification at bounded modality}\label{sec:classification}

The goal of this section is to show that, after nonsingular quadratic suspension, bounded
right modality gives a finite algebraic parametrization.  More precisely, we
construct finitely many modular families whose dimensions are bounded by the modality
and whose numbers of variables and degrees are bounded explicitly in terms of $p$ and
$D$.\\

Retain the notation $C_{p,D}$, $B_{p,D}$, $N_{p,D}$, and $\widehat N_{p,D}$ from Section~\ref{sec:boundedness}.  For $c\ge1$ write
$\m_c=\langle x_1,\ldots,x_c\rangle\subseteq k[[x_1,\ldots,x_c]]$; the analogous notation $\m_d$ is used when a second variable count occurs.  For brevity put
$C=C_{p,D}$ and
$M_{p,D}=p^C(D+1)-C$.
For $c,L\ge1$ write
$\mathcal H_{c,L}=\left\{j^Lh\in\frac{\m_c^2}{\m_c^{L+1}}\mid\Hess(h)(0)=0\right\}$.  This is a closed, right-invariant locus of jets with no nonsingular
quadratic summand left to split off.  By Lemma~\ref{lem:standard-inputs}\ref{std:split},
every stable class has
a residual representative in some $\mathcal H_{c,L}$; in characteristic $2$, a square
quadratic term has zero polar Hessian and is allowed in the residual part.

\begin{definition}
Let an algebraic group $G$ act on a variety $X$.  A locally closed subvariety
$T\subseteq X$ is \emph{orbit-finite} if every $G$-orbit meets $T$ in finitely many
points.  Thus the parameters are genuine moduli up to a finite ambiguity; this is the
finite-fiber notion usually called a \emph{modular family} in singularity classification.
A finite collection of such subvarieties is a \emph{finite modular atlas} for an
invariant constructible subset if every orbit meets at least one of them.
\end{definition}

\begin{lemma}\label{lem:finite-modular-principle}
Let an algebraic group $G$ act on a variety $X$, and let $W\subseteq X$ be invariant and
constructible with $G$-modality at most $D$.  Then $W$ admits a finite modular atlas
$T_1,\ldots,T_s$ with $\dim T_i\le D$.
\end{lemma}

\begin{proof}
Choose a Rosenlicht stratification and reduce to a geometric quotient
$q:V\to Y$ with $V\subseteq W$ and $V,Y$ irreducible.  The definition of modality gives
$\dim Y\le D$.\\

The generic fiber $V_{k(Y)}$ has a closed point whose residue field is finite over
$k(Y)$.  The closure $T$ of that point in $V$ is dominant and generically finite over
$Y$.  After shrinking $Y$ it is quasi-finite and surjective, so it meets every orbit
over that open set in finitely many points.  Apply the same construction over the
complement.  Noetherian induction terminates, and repeating this on the finitely many
Rosenlicht pieces proves the claim.
\end{proof}

\begin{lemma}[Cancellation of nonsingular quadratic suspensions]
\label{lem:stable-cancellation}
Let $g\in\m_c^2$ and $h\in\m_d^2$ be residual germs, so that neither has a
nonsingular quadratic summand.  If $g$ and $h$ are stably right equivalent, then
$c=d$ and $g\overset{\R}{\sim}h$.
\end{lemma}

\begin{proof}
By stable right equivalence there are nonsingular quadratic forms $Q$ and $\widetilde Q$ such
that $g+Q\overset{\R}{\sim}h+\widetilde Q$.  The uniqueness statement in
Lemma~\ref{lem:standard-inputs}\ref{std:split} identifies the residual parts of
right-equivalent germs, also in characteristic $2$; see \cite{GreuelPfister2026}.  Hence the residual variable counts agree and the
residual germs are right equivalent.
\end{proof}

\begin{theorem}
\label{thm:finite-modular-classification}
Fix $p>0$ and $D\ge0$.  Up to nonsingular quadratic suspension, the isolated germs $f$ with
$\rmod(f)\le D$ are covered by finitely many algebraic families $F_i(x,t)$, $t\in T_i$, such
that:
\begin{enumerate}[label=\textup{(\roman*)},leftmargin=*]
\item $\dim T_i\le D$, the number of variables is at most $C_{p,D}$, and
$\deg_xF_i\le N_{p,D}$;
\item every stable right-equivalence class occurs;
\item each stable class occurs at only finitely many points of each $T_i$.
\end{enumerate}
The $T_i$ may be refined so that Milnor number and right modality are constant.
\end{theorem}

\begin{proof}
Let $M=M_{p,D}$ and choose
$L\ge\max\{2M,M+1,N_{p,D}\}$.  The residual corank $c=0$ contributes one stable
class, represented by the zero residual germ; include it as a zero-dimensional family.
We may therefore assume $c\ge1$.  For each $1\le c\le C_{p,D}$ let
$E_{c,L}=\frac{\m_c^2}{\m_c^{L+1}}$ and let $\R_{c,L}$ be the algebraic right-jet group.\\

For a jet $v$ choose a lift $h$ and define $\mu_L(v)$ as the dimension of
$\frac{k[[x_1,\ldots,x_c]]}{\Jac(h)+\m_c^{M+1}}$.  This is independent of the lift because
$L\ge M+1$.  The condition $\mu_L(v)\le M$ is determinantal and hence constructible.
If $\mu_L(v)=m\le M$, the maximal ideal of the length-$m$ truncated Jacobian algebra
has $m$-th power zero, so $\m_c^m\subseteq \Jac(h)+\m_c^{M+1}$.  The Nakayama lemma upgrades
this to $\m_c^m\subseteq \Jac(h)$; hence the truncated length equals the actual Milnor
number.  The converse is immediate.\\

Intersect with the closed locus $\mathcal H_{c,L}$.  Since $L\ge2M$, every isolated germ
there is right $L$-determined, so its formal right modality is the modality of its
$L$-jet.  Upper semicontinuity of modality for algebraic group actions
\cite[Proposition~2.7]{GreuelNguyen2016} makes the locus $W_{c,D}$ of modality at
most $D$ invariant and constructible.\\

Theorem~\ref{thm:uniform} shows that every stable class with modality at most $D$ has a
residual representative in one of the finitely many $W_{c,D}$.  Apply
Lemma~\ref{lem:finite-modular-principle}.  The tautological $L$-jets over the resulting
orbit-finite multisections are algebraic polynomial families.  Since every fiber is
right $N_{p,D}$-determined, truncate to degree $N_{p,D}$.  This does not change the
represented formal orbit.  It also preserves orbit-finiteness: if two truncated
parameters lie in the same $N_{p,D}$-jet orbit, their corresponding formal germs are
right equivalent by $N_{p,D}$-determinacy.  Conversely, because all representatives
under consideration are residual, Lemma~\ref{lem:stable-cancellation} shows that two of
them represent the same stable class only if they are right equivalent.  Thus each
stable class occurs at only finitely many parameters on each multisection.  Finally
refine by the finitely many Milnor numbers and integer modality values.
\end{proof}

\begin{remark}
Greuel--Nguyen observed that, for each fixed positive characteristic, only finitely many
right-simple singularities occur and conjectured an analogous finiteness phenomenon for
an arbitrary number of moduli \cite{GreuelNguyen2016}.  Nguyen subsequently gave
explicit finite lists of modular normal-form families in modalities one and two
\cite{Nguyen2019}.  Theorem~\ref{thm:finite-modular-classification} gives a stable,
finite-atlas form of this finiteness phenomenon in every modality.  Thus it also
realizes the general finiteness expectation stated in the introduction of
\cite{GreuelNguyen2016}.  It is a finite algebraic parametrization statement rather
than an explicit list of normal forms.  For a complementary ADE-type classification
over complete regular local rings, organized by Cohen--Macaulay type rather than right
modality, see \cite{SvorayADE2025}.
\end{remark}

\begin{corollary}\label{cor:stable-finite-classification}
For every fixed characteristic and finite modality bound, stable right classification is
a finite algebraic classification problem: finitely many modular families of dimension
at most the modality suffice.  In modality zero this gives finiteness of stable right-simple classes without using an explicit normal-form classification.
\end{corollary}

\section{Applications to \texorpdfstring{$F$}{F}-singularities}\label{sec:F-applications}

In this section, we apply the classification results of the previous section to
$F$-singularity invariants.  We first recall the definitions that we use, and then prove
uniform results for the $F$-pure threshold, $F$-jumping numbers, test ideals, and
positive-characteristic Bernstein--Sato roots.\\

Throughout this section $k$ is algebraically closed of characteristic $p>0$.  A ring
$A$ of characteristic $p$ is \emph{$F$-finite} if the Frobenius pushforward $F_*A$ is a
finite $A$-module.  Here $F_*^eA$ denotes the additive group of $A$, regarded as an
$A$-module by $a\cdot F_*^eb=F_*^e(a^{p^e}b)$.  Since $k$ is perfect, the regular
power-series ring $R$ is $F$-finite.\\

For $e\ge0$ put $q=p^e$.  For an ideal $I\subseteq R$, its \emph{$q$-th Frobenius power} is
$I^{[q]}=\langle a^q\mid a\in I\rangle$; in particular
$\m^{[q]}=\langle x_1^q,\ldots,x_n^q\rangle$.  We write
$\nu_e(f)=\max\{a\in\mathbb N\mid f^a\notin\m^{[p^e]}\}$.  The \emph{$F$-pure threshold}
of the principal ideal $\langle f\rangle$ at $\m$ is
$\fpt(f)=\lim_{e\to\infty}\frac{\nu_e(f)}{p^e}$; equivalently, it is the first
positive jumping number of the generalized test-ideal filtration.\\

The \emph{$q$-th Frobenius root} $I^{[\frac{1}{q}]}$ is the unique smallest ideal $J$ such that
$I\subseteq J^{[q]}$.  In a regular $F$-finite ring the generalized test ideal of a
principal ideal is, for all sufficiently large $e$,
$\tau(f^t)=\langle f^{\lceil tp^e\rceil}\rangle^{[\frac{1}{p^e}]}$; the positive real numbers at which
$\tau(f^t)$ changes are the \emph{$F$-jumping numbers} of $\langle f\rangle$
\cite{BMS2008}.  We call their set in a specified interval the \emph{numerical
$F$-jumping spectrum}; this forgets the test ideals themselves and remembers only the
jumping numbers.  A homomorphism of $A$-modules is \emph{pure} if it remains injective
after tensoring with every $A$-module.  A reduced $F$-finite ring $A$ is
\emph{$F$-pure} if the Frobenius map $A\to F_*A$ is pure (equivalently, for an
$F$-finite ring, split), and it is \emph{strongly $F$-regular} if for every
$c$ outside the union of the minimal primes some $A$-linear map
$F_*^eA\to A$ sends $F_*^ec$ to $1$ \cite{HochsterHuneke1994}.
Finally, $\tau_{\mathrm{fg}}(A)$ denotes the \emph{finitistic test ideal}, the
intersection of $\Ann_A(0_M^*)$ over all finitely generated $A$-modules
$M$.  Here $0_M^*$ is the \emph{tight closure of zero}: an element $z\in M$ lies in
$0_M^*$ if there is $c\in A$ outside the union of the minimal primes such that
$cF_M^e(z)=0$ in $F_A^e(M)$ for all sufficiently large $e$, where $F_A^e(M)$ denotes
the $e$-fold Frobenius functor applied to $M$.  These are the conventions used in the
sources cited below.

\begin{remark}
The Fedder criterion made Frobenius splitting of hypersurfaces computable
\cite{Fedder1983}.  Hara--Yoshida introduced generalized test ideals
\cite{HaraYoshida2003}, while Musta\c{t}\u{a}--Takagi--Watanabe introduced and
studied $F$-thresholds and related them to the jumping behavior of test ideals
\cite{MTW2005}.  Blickle--Musta\c{t}\u{a}--Smith proved discreteness and rationality
in the regular $F$-finite setting \cite{BMS2008}.  Musta\c{t}\u{a} then constructed
positive-characteristic Bernstein--Sato polynomials and related their roots to
$F$-jumping numbers \cite{Mustata2009}; Blickle--St\"abler extended this framework to
test modules \cite{StaeblerBS}.
\end{remark}

Let $f$ be isolated, put $r=\rmod(f)$, and write
$f\overset{\R}{\sim}g(x_1,\ldots,x_c)+Q$ as in
Lemma~\ref{lem:standard-inputs}\ref{std:split}, with
$c=c(f)$.  By Proposition~\ref{prop:susp}, the Frobenius Loewy length is unchanged by
nonsingular quadratic suspension.  Therefore the residual quantity
$B_F^{\mathrm{res}}=p(\ell_F(f)-1)+c(p-1)$ is exactly the bound $B_F(g)$ for the
residual germ.  We use this quantity in the following threshold estimate.

\begin{theorem}
\label{thm:fpt-modality}
\begin{enumerate}[label=\textup{(\roman*)},leftmargin=*]
\item $\displaystyle \fpt(f)\ge\frac{1}{\max\{2,2B_F^{\mathrm{res}}\}}$.
\item If $c>0$, then $\displaystyle
\fpt(f)\ge\frac{1}{2B_F^{\mathrm{res}}}\ge\frac{1}{2pr+2c(p-1)}$.
\item If $c=0$, then $\fpt(f)=\frac{1}{2}$ in one variable and odd characteristic; if the
quadratic form contains a hyperbolic plane, then $\fpt(f)=1$.
\end{enumerate}
\end{theorem}

\begin{proof}
Apply Theorem~\ref{thm:eta-radius} to the residual germ $g$.  Since
$\ell_F(g)=\ell_F(f)$ by Proposition~\ref{prop:susp}, a right coordinate change in the
$x$-variables sends $g$ to a polynomial $g_{\le 2B_F^{\mathrm{res}}}$ of degree at most
$2B_F^{\mathrm{res}}$.  Extending that coordinate change trivially across the variables
of $Q$ shows that $f$ is right equivalent to
$g_{\le 2B_F^{\mathrm{res}}}+Q$, whose degree is at most
$\max\{2,2B_F^{\mathrm{res}}\}$.\\

If $h$ is a nonzero polynomial of degree at most $N$, put
$a=\left\lfloor\frac{q-1}{N}\right\rfloor$ for $q=p^e$.  Then $h^a$ is nonzero and has
total degree less than $q$, hence cannot lie in $\langle x_1^q,\ldots,x_n^q\rangle$.  Therefore
$\nu_e(h)\ge a$, and passage to the limit gives $\fpt(h)\ge\frac{1}{N}$.  This proves the
first bound.  If $c>0$, then $B_F^{\mathrm{res}}\ge c(p-1)\ge1$, so the maximum is
$2B_F^{\mathrm{res}}$; moreover
$\ell_F(f)-1\le\Fmod(f)\le r$, which gives the second bound.\\

If $c=0$, Lemma~\ref{lem:standard-inputs}\ref{std:split} leaves only the
nonsingular quadratic form.  Over an
algebraically closed field, a nonsingular quadratic form in at least two variables
contains a hyperbolic plane $uv$.  Then $Q^{q-1}$ contains the monomial
$u^{q-1}v^{q-1}$ with nonzero coefficient, so $\nu_e(Q)=q-1$ and $\fpt(Q)=1$.
The only remaining case is one variable in odd characteristic, where $Q$ is right
equivalent to $x^2$ and $\fpt(Q)=\frac{1}{2}$.
\end{proof}

\begin{proposition}\label{prop:fpt-refined}
If $c=c(f)$, then
\[
 \fpt(f)\ge
 \frac{1-\frac{1}{\ord(f)}}{p^c(\Fmod(f)+1)-c}.
\]
\end{proposition}

\begin{proof}
By \cite[Corollary~2.14]{Svoray2026},
$\fpt(f)\ge\frac{1}{\Tj(f)}\left(1-\frac{1}{\ord(f)}\right)$.  Since $\Tj(f)\le\mu(f)$, apply
\eqref{eq:essential-bound}.
\end{proof}

\begin{proposition}
\label{prop:Freg-new}
Assume $p\nmid N$ and let $f_N=xy-z^N$.  Then $\frac{k[[x,y,z]]}{\langle f_N\rangle}$ is strongly
$F$-regular and $\fpt(f_N)=1$, whereas
$\Fmod(f_N)=\propermod(f_N)=\rmod(f_N)=\left\lfloor\frac{N-1}{p}\right\rfloor$.
\end{proposition}

\begin{proof}
The hyperbolic block $xy$ splits off, leaving $z^N$.  The univariate formula and
Theorem~\ref{thm:pmod-rmod-general} give the modality statement.  The ring is the tame
cyclic quotient $k[[u,v]]^{\boldsymbol{\mu}_N}$ for the action $\zeta\cdot(u,v)=(\zeta u,\zeta^{-1}v)$, with $x=u^N$, $y=v^N$, and $z=uv$.  Since $p\nmid N$,
the invariant ring is a direct summand of the regular ring and hence strongly
$F$-regular by \cite{HochsterHuneke1994}.  It is $F$-pure, so the principal
$F$-pure threshold is one.
\end{proof}

\begin{corollary}
\label{cor:uniform-m-adic-constancy}
Fix $p>0$ and $D\ge0$.  Let $f$ be isolated with $\rmod(f)\le D$ and
$h\in\m^{\widehat N_{p,D}+1}$.  Then:
\begin{enumerate}[label=\textup{(\roman*)},leftmargin=*]
\item $f+h\overset{\R}{\sim}f$ and $\fpt(f+h)=\fpt(f)$;
\item $f+h$ has the same numerical $F$-jumping spectrum as $f$;
\item for some local coordinate automorphism $\varphi$ we have that 
$\tau((f+h)^\lambda)=\varphi(\tau(f^\lambda))$ for every $\lambda\ge0$.
\end{enumerate}
\end{corollary}

\begin{proof}
The first assertion is precisely Corollary~\ref{cor:uniform-full-determinacy}.  The
$F$-pure threshold and the set of $F$-jumping numbers are invariant under local
coordinate changes.  Generalized test ideals commute with local isomorphisms, localization, and completion in the regular $F$-finite setting \cite[Proposition~2.13]{BMS2008}; in particular they are functorial under ring automorphisms, which gives the displayed identity.
\end{proof}

\begin{remark}
Hernández--Núñez-Betancourt--Witt proved local $\m$-adic constancy of the
$F$-pure threshold and, with a deeper perturbation order, literal equality of all test
ideals $\tau(f^\lambda)$ for $0\le\lambda<1$ around a fixed polynomial with an
isolated singularity \cite{HNBW2018}.  Their explicit radii depend on the Jacobian
colength and the ambient dimension.  Corollary~\ref{cor:uniform-m-adic-constancy}
has a different form: its radius is uniform on the whole bounded-modality locus and is
independent of ambient dimension, while the test ideals are identified up to the
coordinate change supplied by right determinacy.
\end{remark}

Set $C=C_{p,D}$ and $d_{p,D}=\widehat N_{p,D}$.  Put
$\delta_p=0$ if $p=2$ and $\delta_p=1$ if $p>2$, and set
$n_{p,D}=C+\delta_p$.  Let $e_0$ be the least integer with
$p^{e_0}>d_{p,D}$ and set $H_{p,D}=\binom{d_{p,D}+n_{p,D}}{n_{p,D}}$.
Define
\[
 \Xi_{p,D}=\bigcup_{\substack{0\le a\le e_0+H_{p,D}\\1\le b\le H_{p,D}}}
 \left\{\frac{j}{p^a(p^b-1)}:1\le j\le p^a(p^b-1)\right\}.
\]

\begin{theorem}
\label{thm:universal-Fjump-spectrum}
If $f$ is isolated and $\rmod(f)\le D$, then every $F$-jumping number of $\langle f\rangle$ in
$(0,1]$ belongs to the finite set $\Xi_{p,D}$.  Consequently only finitely many
numerical $F$-jumping spectra, and in particular only finitely many $F$-pure thresholds,
can occur at fixed $p,D$, independently of the original ambient dimension.
\end{theorem}

\begin{proof}
After splitting, either the quadratic block contains a hyperbolic plane or it has
rank at most one.  Over an algebraically closed field the latter alternative can occur
only in odd characteristic. If the quadratic block contains no hyperbolic plane, then
it is zero in characteristic two and has dimension at most
one in odd characteristic. Hence, in the non-hyperbolic case,
Theorem~\ref{thm:finite-modular-classification} reduces the germ to a principal
polynomial in at most $n_{p,D}=C+\delta_p$ variables and of degree at most
$d_{p,D}$.  Test ideals and their jumping numbers are unchanged by the local
coordinate changes used here and commute with localization and completion in this
regular $F$-finite setting; see \cite[Proposition~2.13]{BMS2008}.  Proposition~3.8 of
\cite{BMS2008}, with one generator, therefore puts every global $F$-jumping exponent of
this polynomial in the displayed finite denominator set.  A local jump at the maximal
ideal is necessarily a global jump: if the two adjacent test ideals become distinct
after localization at the maximal ideal, then they were already distinct before
localization.  Hence every local $F$-jumping number occurring here belongs to the same
finite denominator set.\\

If the quadratic block has rank at least two, it contains a hyperbolic plane $uv$;
by Lemma~\ref{lem:standard-inputs}\ref{std:split} the residual term $h$ involves
variables disjoint from $u,v$.  For
every $q=p^e$, the monomial $u^{q-1}v^{q-1}$ occurs with coefficient one in
$(h+uv)^{q-1}$ and cannot cancel.  Thus the hypersurface is $F$-pure and its first jump is $1$; there are no
jumps in $(0,1)$.
\end{proof}

All reduced denominators occurring in the finite candidate set
$\Xi_{p,D}$ are bounded by the single integer
$Q_{p,D}=p^{e_0+H_{p,D}}(p^{H_{p,D}}-1)$.  This gives a uniform separation between
distinct candidates and will allow one Frobenius approximation to determine the
$F$-pure threshold.

\begin{corollary}\label{cor:Fjump-separation}
Any two distinct $F$-jumping numbers in $(0,1]$ arising at modality at most $D$
differ by at least $\frac{1}{Q_{p,D}^2}$.
\end{corollary}

\begin{proof}
Every reduced denominator is at most $Q_{p,D}$.  Two distinct reduced fractions with
denominators at most $Q_{p,D}$ differ by at least $\frac{1}{Q_{p,D}^2}$.
\end{proof}

\begin{theorem}
\label{thm:one-level-fpt}
Choose $E_{p,D}$ so that $q_{p,D}=p^{E_{p,D}}>Q_{p,D}^2+1$.  Then the single integer
$\nu_{E_{p,D}}(f)$ determines $\fpt(f)$ for every isolated germ with
$\rmod(f)\le D$.
\end{theorem}

\begin{proof}
For every principal germ and every $e\ge0$ one has the elementary interval
$\frac{\nu_e(f)}{p^e}\le\fpt(f)\le\frac{\nu_e(f)+1}{p^e}$.  Indeed, the left inequality follows from
the defining Frobenius sequence, while $f^{\nu_e(f)+1}\in\m^{[p^e]}$ and Frobenius
iteration give the right inequality.  At $e=E_{p,D}$ the interval has length $\frac{1}{q_{p,D}}$, which is
strictly smaller than the separation in Corollary~\ref{cor:Fjump-separation}.
Theorem~\ref{thm:universal-Fjump-spectrum} therefore leaves exactly
one possible $F$-pure threshold.
\end{proof}

The assertions about $F$-pure thresholds concern nonzero
equation families. The zero residual class is represented
by $x^2$ in odd characteristic and by $uv$ in every
characteristic. When recording the numerical invariants
of suspensions, we also consider one-variable quadratic
suspensions in odd characteristic and hyperbolic
suspensions. Every family containing a hyperbolic plane
has $F$-pure threshold $1$.

\begin{theorem}
\label{thm:fpt-stratification}
The modular atlas of Theorem~\ref{thm:finite-modular-classification} can be refined into
finitely many locally closed parameter spaces on which Milnor number, right modality,
and $F$-pure threshold are all constant.  On every irreducible modular family the
$F$-pure threshold is constant on a dense open subset.
\end{theorem}

\begin{proof}
Fix $e=E_{p,D}$ and $q=p^e=q_{p,D}$.  For a polynomial family $F(x,t)$, the condition
$F_t^r\in\m^{[q]}$ is closed because it is the vanishing of finitely many coefficient
functions after reduction modulo $\langle x_1^q,\ldots,x_n^q\rangle$.  Hence
$\nu_e(F_t)=r$ is locally closed.  Theorem~\ref{thm:one-level-fpt} converts these
finitely many strata into constant-$F$-pure-threshold strata.  On an irreducible
parameter space, one of the finitely many locally closed strata contains the generic
point and therefore contains a dense open subset; this gives the final assertion.
\end{proof}

\begin{corollary}
\label{cor:test-uniform}
Let $A=\frac{R}{\langle f\rangle}$ be reduced, write $\m_A=\frac{\m}{\langle f\rangle}$, and assume $\rmod(f)\le D$.  Then:
\begin{enumerate}[label=\textup{(\roman*)},leftmargin=*]
\item the number of $F$-jumping coefficients in $[0,1]$ is at most $M_{p,D}+1$;
\item $\length_k\left(\frac{A}{\tau_{\mathrm{fg}}(A)}\right)\le M_{p,D}$;
\item if $E=pD+C_{p,D}(p-1)+1$, then
$\m_A^E\subseteq\Jac(f)A\subseteq\tau_{\mathrm{fg}}(A)$.
\end{enumerate}
\end{corollary}

\begin{proof}
Katzman--Lyubeznik--Zhang \cite{KLZ2011} bound the number of jumps by $\Tj(f)+1$.
Since $\Tj(f)\le\mu(f)\le M_{p,D}$, the first assertion follows.  Because $A$ is a reduced hypersurface complete local algebra over the algebraically
closed field $k$, it is absolutely reduced and equidimensional.  Jiang
\cite[Theorem~4.1]{Jiang2021} therefore gives
$\Jac(f)A\subseteq\tau_{\mathrm{fg}}(A)$.  Lemma~\ref{lem:loewy}, applied to the residual germ as in
Proposition~\ref{prop:susp}, gives the displayed power containment; quadratic variables already lie in the Jacobian ideal.  Finally,
\[
 \frac{A}{\Jac(f)A}\cong\frac{R}{\langle f\rangle+\Jac(f)}
\]
is the Tjurina algebra, and $\Jac(f)A\subseteq\tau_{\mathrm{fg}}(A)$ shows that
$\frac{A}{\tau_{\mathrm{fg}}(A)}$ is its quotient.
\end{proof}

Following Blickle--St\"abler, for each Frobenius level $e$ there is a finite set of
positive-characteristic Bernstein--Sato roots attached to $\langle f\rangle$. The corollary below concerns
only its eventual level-by-level behavior, as stated explicitly:

\begin{corollary}\label{cor:finite-BS-sequences}
At fixed $p,D$, only finitely many eventual sequences of level-$e$
positive-characteristic Bernstein--Sato root sets occur for isolated hypersurfaces of right
modality at most $D$.
\end{corollary}

\begin{proof}
By Theorem~\ref{thm:universal-Fjump-spectrum}, only finitely many numerical $F$-jumping spectra in $(0,1]$ occur at fixed $p,D$. In the non-hyperbolic case, the proof of that theorem replaces the germ, after splitting off its nonsingular quadratic part, by a bounded-degree polynomial in a uniformly bounded number of variables whose localization and completion have the same numerical $F$-jumping spectrum. In the hyperbolic case the spectrum in $(0,1]$ is already fixed. By compatibility of generalized test ideals with completion \cite[Proposition~2.13]{BMS2008}, we may therefore apply Blickle--Stäbler to these polynomial representatives. Their Theorem~5.4 gives, for all sufficiently large $e$, the root set
$\left\{\frac{\lceil\lambda p^e\rceil-1}{p^e}:\lambda\text{ is an }F\text{-jumping number in }(0,1]\right\}$.
Since only finitely many jumping spectra occur, only finitely many eventual level-by-level Bernstein--Sato root sequences occur.
\end{proof}

For matrices,
prescribed-jet transversality reduces questions about formal equivalence to the geometry
of a finite stratification by rank.  For hypersurfaces, the Frobenius algebra and the
critical cover reduce questions about right equivalence and modality to finite-dimensional
algebraic data.  In particular, once the characteristic and the modality are bounded, the
remaining classification problem and the $F$-singularity invariants considered above are
controlled by finitely many algebraic families.

\end{document}